\documentclass[ama]{wileynjd-v1}

\articletype{\small{Haemodynamics Quantities of Interest from Ultrasound Imaging} \\[0.1cm]
\Large{Research article - Fundamental}}%

\usepackage[utf8]{inputenc}
\usepackage{fullpage}
\usepackage{amsmath, amsthm, amssymb}
\usepackage{thmtools, thm-restate}
\usepackage{dsfont}
\usepackage{mathtools}
\mathtoolsset{showonlyrefs=true} 
\usepackage{graphicx}    
\usepackage{subfigure}
\usepackage{color}
\usepackage{xcolor}

\definecolor{darkolivegreen}{rgb}{0.33, 0.42, 0.18}
\definecolor{celestialblue}{rgb}{0.29, 0.59, 0.82}

\usepackage{hyperref}
\hypersetup{
pdftitle={GLM2020},
pdfauthor={Olga Mula, Damiano Lombardi and Felipe Galarce}, 
}

\usepackage{algorithm}
\usepackage{algpseudocode}
\usepackage{soul}

\numberwithin{equation}{section} 

\newcommand{\cF}{\ensuremath{\mathcal{F}}}

\newcommand{\cH}{\ensuremath{\mathcal{H}}}

\newcommand{\cK}{\ensuremath{\mathcal{K}}}
\newcommand{\cL}{\ensuremath{\mathcal{L}}}
\newcommand{\cM}{\ensuremath{\mathcal{M}}}
\newcommand{\cN}{\ensuremath{\mathcal{N}}}
\newcommand{\cO}{\ensuremath{\mathcal{O}}}
\newcommand{\cP}{\ensuremath{\mathcal{P}}}

\newcommand{\cS}{\ensuremath{\mathcal{S}}}

\newcommand{\cV}{\ensuremath{\mathcal{V}}}
\newcommand{\cW}{\ensuremath{\mathcal{W}}}

\newcommand{\bE}{\ensuremath{\mathbb{E}}}

\newcommand{\bG}{\ensuremath{\mathbb{G}}}

\newcommand{\bN}{\ensuremath{\mathbb{N}}}

\newcommand{\bP}{\ensuremath{\mathbb{P}}}

\newcommand{\bR}{\ensuremath{\mathbb{R}}}

\newcommand{\rY}{\ensuremath{\mathrm{Y}}}

\def\[{\left[}
\def\]{\right]}
\def\<{\langle}
\def\>{\rangle}
\def\({\left(}
\def\){\right)}
\def\[{\left [}
\def\]{\right]}
\def\({\left(}
\def\){\right)}

\newcommand{\innerp}[2]{\< #1, #2 \>}
\newcommand{\norm}[1]{\Vert #1 \Vert}

\newcommand{\Vn}{\ensuremath{{V_n}}}
\newcommand{\Wm}{\ensuremath{{W_m}}}

\newcommand{\wc}{\ensuremath{\mathrm{wc}}}
\newcommand{\ms}{\ensuremath{\mathrm{ms}}}

\newcommand{\pbdw}{\ensuremath{\text{(pbdw)}}}
\newcommand{\wcpbdw}{\ensuremath{\text{(wc, pbdw)}}}
\newcommand{\mspbdw}{\ensuremath{\text{(ms, pbdw)}}}
\newcommand{\wcaff}{\ensuremath{\text{(wc, aff)}}}

\newcommand{\wcaffk}{\ensuremath{\text{(wc, k)}}}

\newcommand{\msaffk}{\ensuremath{\text{(ms, k)}}}

\newcommand{\dist}{\operatorname{dist}}

\newcommand{\opt}{\ensuremath{\mathrm{opt}}}
\renewcommand{\sup}{\text{sup}}
\renewcommand{\inf}{\text{inf}}
\renewcommand{\max}{\text{max}}
\renewcommand{\min}{\text{min}}

\DeclareMathOperator*{\argmin}{arg\,min}

\newcommand{\dx}{\ensuremath{\mathrm dx}}
\newcommand{\ds}{\ensuremath{\mathrm ds}}

\DeclareMathOperator*{\vspan}{span}

\newcommand{\eps}{\ensuremath{\varepsilon}}

\newcommand{\inlet}{\text{in}}
\newcommand{\outlet}{\text{out}}
\newcommand{\HR}{\textrm{HR}}
\newcommand{\obs}{\text{obs}}
\newcommand{\unobs}{\text{unobs}}
\newcommand{\wss}{\cS}
\newcommand{\vorticity}{\Theta}
\providecommand{\abs}[1]{\lvert#1\rvert}

\newcommand{\om}[1]{\textcolor{black}{#1}} 
\newcommand{\fg}[1]{\textcolor{black}{#1}} 
\newcommand{\fgal}[1]{\textcolor{black}{#1}} 
\newcommand{\dl}[1]{\textcolor{black}{#1}} 

\usepackage[top=1in, bottom=1in, left=0.7in, right=0.7in, footskip=0.3in]{geometry}

\begin{document}

\title{
Reconstructing Haemodynamics Quantities of Interest from Doppler Ultrasound Imaging
}

\author{Felipe Galarce*$^1$}
\author{Damiano Lombardi$^1$}
\author{Olga Mula$^{1,2}$}

\authormark{Galarce \textsc{et al}}

\address[1]{Centre de Recherche INRIA de Paris \& Laboratoire Jacques-Louis Lions, France.}
\address[2]{Paris-Dauphine University, PSL Research University, CNRS, UMR 7534, CEREMADE, France.}

\corres{*2 Rue Simone Iff 75589 Paris Cedex 12. Bureau A316. \email{felipe.galarce-marin@inria.fr}}

\abstract[Summary]{The present contribution deals with the estimation of haemodynamics Quantities of Interest by exploiting Ultrasound Doppler measurements. A fast method is proposed, based on the \fg{Parameterized Background Data-Weak (PBDW) method}. Several methodological contributions are described: a sub-manifold partitioning is introduced to improve the reduced-order approximation, two different ways to estimate the pressure drop are compared, and an error estimation is derived. A \om{fully synthetic} test-case on a realistic common carotid geometry is presented, showing that the proposed approach is promising in view of realistic applications.}


\maketitle

\section{Introduction}
In biomedical engineering, most realistic applications have to deal with data assimilation. The problem to be solved consists in providing predictions on Quantities of Interest (QoI) given observations of the system which are often partial and noisy. The present work is a contribution to this topic and focuses on the reconstruction of haemodynamics QoI by exploiting Doppler Ultrasound Imaging. The proposed methodology is however general and can easily be extended to a broad set of other applications. Doppler Ultrasound, in its different modes, is one of the most used, clinically available technologies to monitor blood flows in the heart cavity and in several segments of the vascular tree. Its main advantages are that it is fast, non-invasive, and cheap. Its main drawback lies in the space resolution: the observed quantity (more precisely defined in Section \ref{sec:instantiation}) amounts to the noisy average in some voxels of one or two components of the velocity field.

In several applications related to the cardiovascular system, the QoI to be predicted are:
\begin{enumerate}
\item \emph{The complete 3D velocity field} and some quantities associated to it, say, for instance, the maximal velocity (\cite{hata1987,galarce2020}). 
\item \emph{Pressure and pressure drop} (\cite{hatle1978,hatle1979,hatle1980,mates1978,funamoto2013}): this is particularly relevant, since it is one of the main indicators of the severity of stenoses and eventual arterial blockages. The direct measure of a pressure (or even a pressure drop) could be performed by implanting a catheter, hence in a rather invasive way. 
\item \emph{The vorticity} (\cite{mehregan2014,hirtler2016,charonko2013,sotelo2018}): this quantity is monitored especially in the heart cavity and around cardiac valves. A too large vorticity could induce, for instance, haemolysis.
\item \emph{The wall shear stress} (\cite{gibson1993,reneman2006,shojima2004}): this is related to the mechanical stress that the blood exerts on the endothelial cells, of paramount importance in aneurysms and plaque formation. 
\end{enumerate}

The QoI listed have to be reliably estimated \textit{in vivo}, with the additional constraint of being estimated fast, ideally in real time. For this, two main approaches are available. The first consists in a purely data-driven strategy where learning techniques are used to build an approximation of the observable-to-QoI map given a sufficiently large dataset. The second consists in using an \emph{a priori} description of the physics involved by means of a mathematical model, often given in the form of a Partial Differential Equation (PDE), and then solve an inverse problem. On the one hand, since we are dealing with space fields estimation, the pure data-driven learning approach will in general need an exceedingly large data set to meet the accuracy constraints of the application. On the other hand, discretising the system of Partial Differential Equations and solving the inverse problem will in general result in a prohibitive computational cost, thus leading to unacceptable computing times. These facts motivate the use of mixed approaches combining an \emph{a priori} knowledge coming from an available, potentially inexact physical model of the system, and the \emph{a posteriori} knowledge coming from the data. One recent example in this direction is \cite{kissas2020}, where a physics-informed machine learning approach to estimate pressure in blood vessels from MRI was proposed. In this work, we use a different methodology based on reduced-order modelling of parametrized PDEs.

Our contribution \dl{is to extend some existing methods involving reduced modelling and to adapt them to the present context in order} to propose a systematic methodology to estimate the above five QoI in close to real time, and to assess its feasibility in non trivial numerical examples involving the carotid artery. However, due to our lack of real ultrasound images, our experiments present certain limitations: we have worked with synthetically generated images and have used an admittedly simple Gaussian modelling of the ultrasound noise (Doppler ultrasound images present a very involved space-time structure which is not the main topic of our work and we refer to \cite{ledoux1997,bjaerum2002,demene2015} for further details on this matter). The PDE model considered to describe the haemodynamics is the system of incompressible Navier-Stokes equations, which is generally acknowledged to be accurate for large vessels such as the carotid artery. We therefore assume that there is no model error and that the true system is governed by these equations. Note in addition that this assumption also comes from the fact that it is not possible to study the impact of the model error without real measurements. 

The structure of the paper is as follows. In Section \ref{sec:RecMeth} we describe the reconstruction method which we use. The method is very general and its main mathematical foundations have been established in previous works (see \cite{MM2013, MMPY2015, MPPY2015, MMT2016, BCDDPW2017, ABGMM2017, Taddei2017}). We make a presentation that alternates between a summary of the general mathematical theory, and its particular application to our problem of interest. One relevant point to remark is that so far the methodology has mainly focused on reconstructing spatial fields from observable quantities. In our case, this concerns the reconstruction of the 3D velocity field. One relevant novelty with respect to previous contributions is that we show that it is possible to reconstruct unobserved quantities such as the pressure field or the pressure drop in our problem. This is possible by making a \emph{joint reconstruction of observable and unobservable fields}, which are velocity and pressure in our case. We explain this  idea in \ref{sec:joint_up}. The reconstruction of the wall shear stress and the vorticity are discussed in Sections \ref{subsubsec:wss} and \ref{subsubsec:vort} respectively. The numerical experiments on a carotid bifurcation are given in Section \ref{sec:noise-free_test} for the case of noiseless images. We examine the effect of noise in Section \ref{sec:cls_test}.

\section{Reconstruction methods}
\label{sec:RecMeth}

In this section we present the reconstruction methods that we use in this work. We alternate between abstract mathematical statements and their translation to our specific problem of interest in order to make the presentation as pedagogical as possible. To simplify the discussion, the presentation is done for noiseless measurements. At the end of the paper, we will discuss how to take noise into account.

\subsection{State estimation and recovery algorithms: abstract setting}
\label{sec:SE}

Our problem enters into the following setting, for which solid mathematical foundations have been developed in recent years. The relevant references will be cited throughout the discussion.

Let $\Omega$ be a domain of $\bR^d$ for a given dimension $d\geq 1$.
We work on a Hilbert space $V$ defined over $\Omega$ which is relevant for the problem under consideration. As we will see in the following, we may change $V$ depending on our needs. However, once the space is fixed, the subsequent developments have to remain consistent with this choice. The space is endowed with an inner product $\<\cdot, \cdot\>$ and induced norm $\Vert \cdot \Vert$.

Our goal is to recover an unknown function $u\in V$ from $m$ measurement observations
\begin{equation}
\label{eq:meas}
\ell_i(u),\quad i=1,\dots,m,
\end{equation}
where the $\ell_i$ are linearly independent linear forms over $V$. In many applications, each $\ell_i$ models a sensor device which is used to collect the measurement data $\ell_i(u)$. In our particular application, the observations come in the form of an image and each $\ell_i$ models the response of the system in a given pixel as Figure \ref{fig:doppler-img} illustrates. We denote by $\omega_i \in V$ the Riesz representers of the $\ell_i$. They are defined via the variational equation
$$
\left< \omega_i, v \right> = \ell_i(v),\quad \forall v \in V.
$$
Since the $\ell_i$ are linearly independent in $V'$, so are the $\omega_i$ in $V$ and they span an $m$-dimensional space
$$
\Wm={\rm span}\{\omega_1,\dots,\omega_m\} \subset V.
$$
The observations $\ell_1(u),\dots, \ell_m(u)$ are thus equivalent to knowing the orthogonal projection
\begin{equation}
\omega =P_\Wm u.
\end{equation}
In this setting, the task of recovering $u$ from the measurement observation $\omega$ can be viewed as building a recovery algorithm
$$
A:\Wm\mapsto V
$$
such that $A(P_{\Wm}u)$ is a good approximation of $u$ in the sense that $\Vert u - A(P_{\Wm}u) \Vert$ is small.

Recovering $u$ from the measurements $P_\Wm u$ is a very ill-posed problem since $V$ is generally a space of very high or infinite dimension so, in general, there are infinitely many $v\in V$ such that $P_\Wm v = \omega$. It is thus necessary to add some a priori information on $u$ in order to recover the state up to a guaranteed accuracy. In the following, we work in the setting where $u$ is a solution to some parameter-dependent PDE of the general form
\begin{equation*}
\cP(u, y) = 0,
\end{equation*}
where $\cP$ is a differential operator and $y$ is a vector of parameters that describes some physical property and lives in a given set $\rY\subset \bR^p$. Therefore, our prior on $u$ is that it belongs to the set
\begin{equation}
\label{eq:manifold}
\cM \coloneqq \{ u(y) \in V \; :\; y\in\rY \},
\end{equation}
which is sometimes referred to as the {\em solution manifold}. The performance of a recovery mapping $A$ is usually quantified in two ways:
\begin{itemize}
\item If the sole prior information is that $u$ belongs to the manifold $\cM$, the performance is usually measured by the worst case reconstruction error
\begin{equation}
\label{eq:err-A-wc}
E_{\wc}(A,\cM) = \sup_{u\in\cM} \Vert u - A(P_\Wm u) \Vert \, .
\end{equation}
\item In some cases $u$ is described by a probability distribution $p$ on $V$ supported on $\cM$. This distribution is itself induced by a probability distribution on $\rY$ that is assumed to be known. When no information about the distribution is available, usually the uniform distribution is taken. In this Bayesian-type setting, the performance is usually measured in an average sense through the mean-square error
\begin{equation}
\label{eq:err-A-ms}
E^2_{\ms}(A,\cM) = \bE\left( \Vert u - A(P_\Wm u) \Vert^2\right) = \int_V \Vert u - A(P_\Wm u)\Vert^2 dp(u) \, ,
\end{equation}
and it naturally follows that $E_{\ms}(A,\cM)\leq E_{\wc}(A,\cM) $.
\end{itemize}

\subsection{Instantiation to the application of interest}
\label{sec:instantiation}
In our case, $\Omega \subset \bR^3$ is a portion of a human carotid artery as given in Figure \ref{fig:geometry}. The boundary $\Gamma\coloneqq \partial \Omega$ is the union of the inlet part $\Gamma_i$ where the blood is entering the domain, the outlets $\Gamma_{o,1}$ and $\Gamma_{o,2}$ where the blood is exiting the domain after a bifurcation, and the walls $\Gamma_w$.

\begin{figure}[!htbp]
  \centering
   \includegraphics[width=0.6\textwidth]{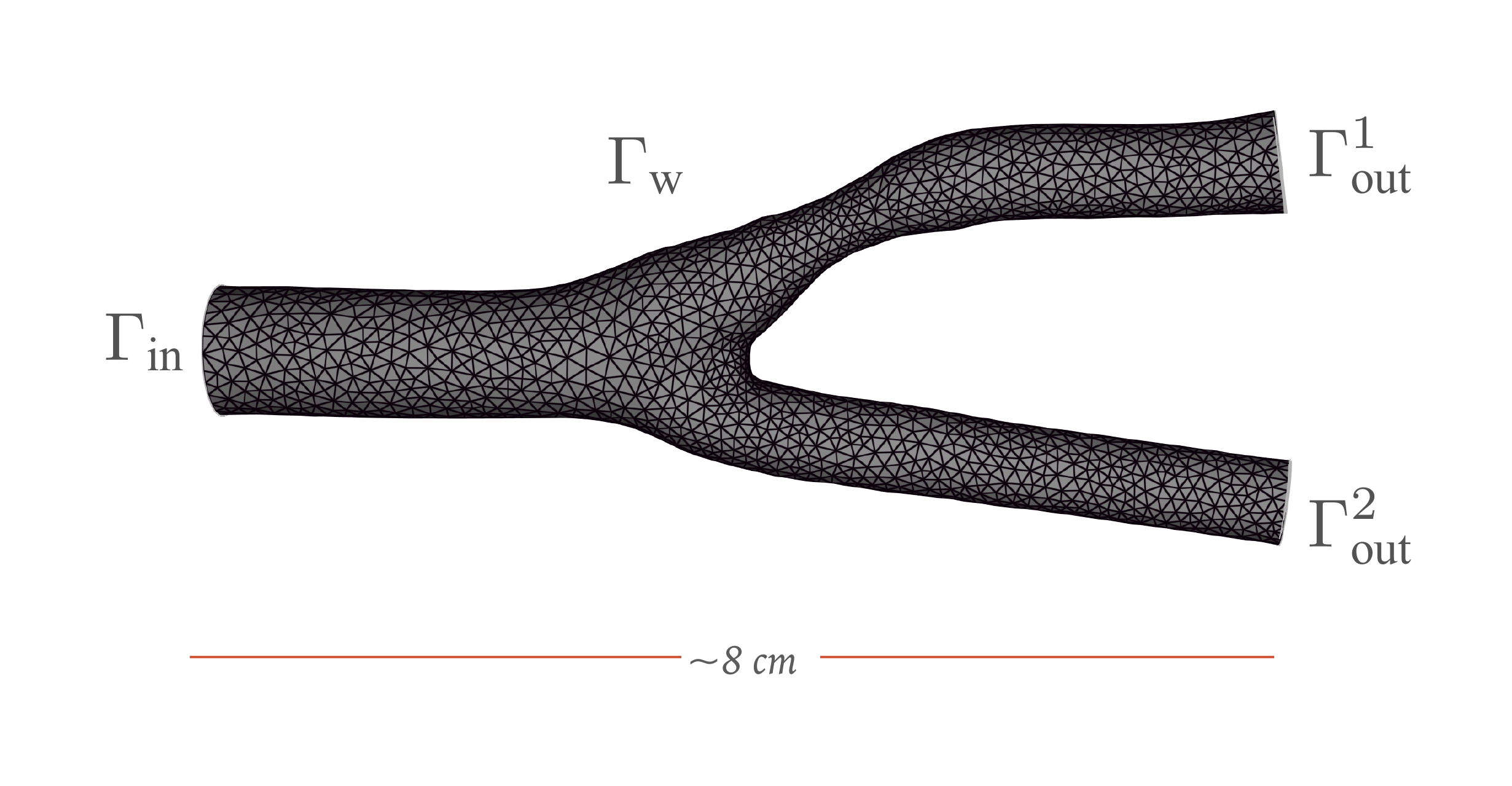}
   \includegraphics[height=1cm,angle=-90]{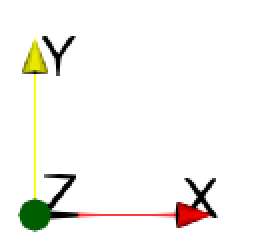}
  \caption{Domain $\Omega$ used in the simulations. Note the small stenosis in the upper part of the bifurcation. 
  }
  \label{fig:geometry}
\end{figure}

Our goal is to reconstruct for all time $t$ on an interval $[0, T]$ (with $T>0$) the full 3D velocity field $u$ and pressure $p$ in the whole carotid $\Omega$. Additionally, we also want to reconstruct related quantities like the wall-shear stress, the vorticity and the pressure drop between the inlet and the outlets. The Doppler ultrasound device gives images with a certain time frequency. Each image contains \emph{partial} information on the blood velocity in a subdomain of the carotid. Depending on the ultrasound technology, we are either given the projection of the velocity along the direction $b$ of the ultrasound probe (CFI mode), or along a plane (VFI mode). Figure \ref{fig:doppler-img} illustrates both imaging techniques.

\begin{figure}[!htbp]
\centering
\subfigure[Color flow image (CFI) \label{fig:cfi}]{ 
\includegraphics[height=5cm]{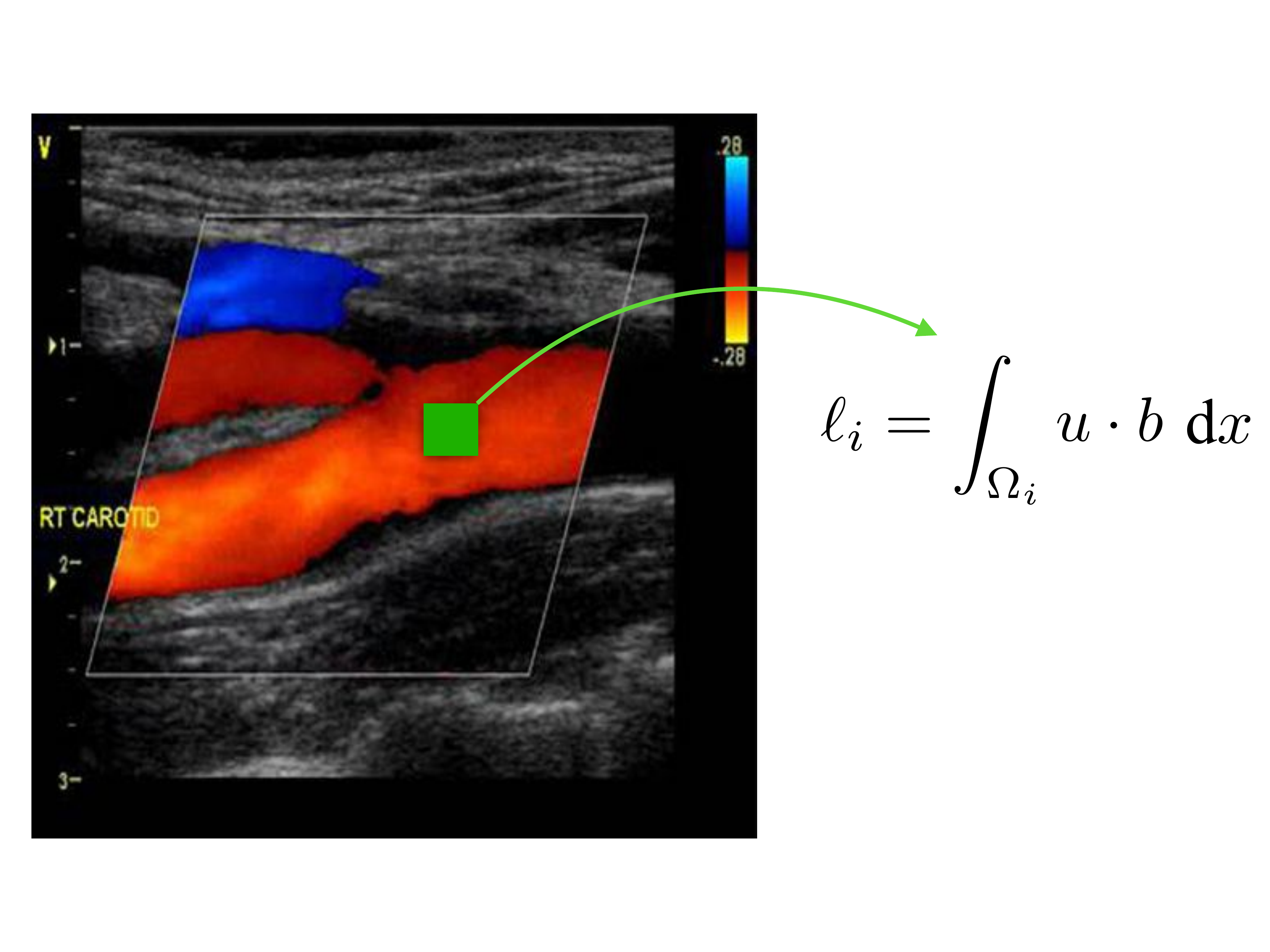}}
\subfigure[Vector flow image (VFI) \label{fig:vfi}]{ 
 \includegraphics[height = 5 cm]{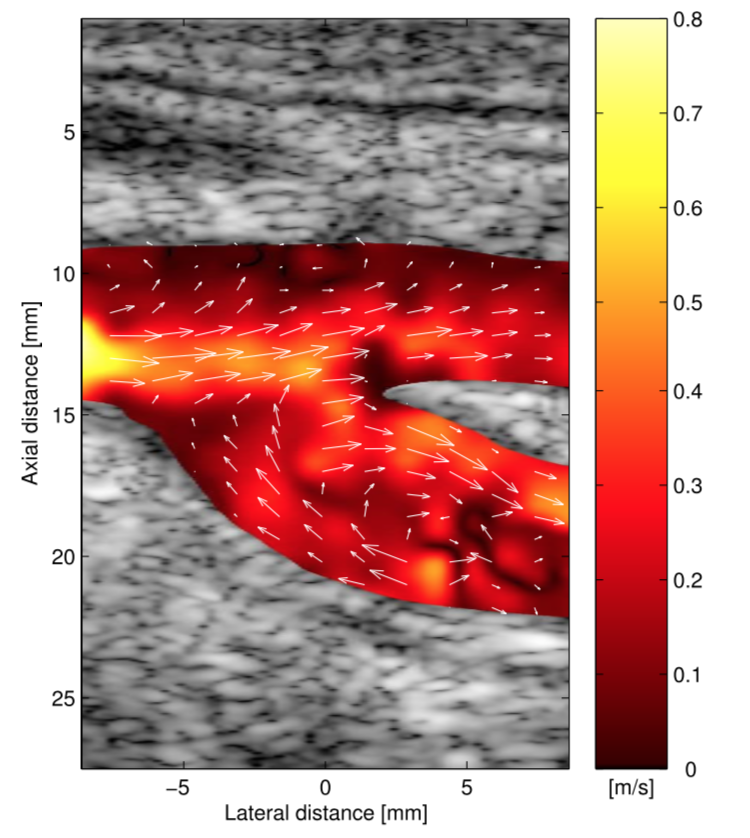}}
\caption{Velocity image of the common carotid bifurcation.}
\label{fig:doppler-img}
\end{figure}

Due to our lack of real images, our experiments are fully synthetic, and we work with an idealized version of CFI images to generate measurements. For each time $t$, a given image is a local average in space of the velocity projected onto the direction in which the ultrasound probe is steered. More specifically, we consider a partition of $\Omega = \cup_{i=1}^m \Omega_{i} $ into $m$ disjoint subdomains (voxels) $\Omega_{i}$. Then, from each CFI image we collect
\begin{equation}
\ell_i(u) = \int_{\Omega_i} u \cdot b ~ \dx,\quad 1\leq i\leq m,
\label{eq:the_measures1D}
\end{equation}
where $b$ is a unitary vector giving the direction of the ultrasound beam. According to what has been exposed in section \ref{sec:SE}, the $\ell_i$ are linear functionals from a certain Hilbert space $V$ which, in our case, is yet to be defined.

Note that, in this setting, pressure is an unobserved, invisible quantity so it cannot be recovered directly from the measurements. In other words, we cannot build a reconstruction mapping $A$ from the velocity observations to the pressure $p$ and provide good recovery guarantees in terms of the recovery error \eqref{eq:err-A-wc} or \eqref{eq:err-A-ms}. However, we prove in the following that it is actually possible to build a mapping $A$ to reconstruct \emph{jointly} the couple $(u,p)$ with good recovery guarantees.

As it seems natural, the joint reconstruction strategy requires necessarily some physical modelling to help to reduce the ill-posedness in $u$, and especially the one in $p$. This important ingredient comes, in our case, from the following incompressible Navier-Stokes equations defined on $\Omega\times [0,T]$. For a fluid with density $\rho\in\mathbb{R}^+$ and dynamic viscosity $\mu\in\mathbb{R}^+$, we search for all $t\in [0,T]$ the couple $(u(t),p(t))$ of velocity and pressure such that

\begin{equation}
\left\lbrace
\begin{aligned}
\rho \frac{\partial u}{\partial t}(t)  + \rho \(\nabla u(t)\) u(t)- \mu \Delta u(t) + \nabla p(t) = 0,\quad\text{ in } \Omega  \\ 
\nabla \cdot u(t) = 0,\quad \text{ in } \Omega.
\end{aligned}
\right.
\label{eq:Navier_Stokes}
\end{equation}

The equations are closed by prescribing an initial condition and boundary conditions. We defer their detailed description to section \ref{sec:sampling_M} for the sake of brevity in the current discussion. At this point, the essential information is that a weak formulation of this equation makes the problem be well posed when we seek $(u(t), p(t))$ in
$$
V = U \times P = [H^1(\Omega)]^3 \times L^2(\Omega),
$$
which is the ambient space we work with and measure errors later on.

Our Navier-Stokes model involves some parameters $y \in Y\in \bR^p$, e.g., the heart rate. An important detail is that we consider time as a parameter so $t$ will be one of the coordinates of $y$. A manifold $\cM$ is generated by the variations of the parameters
$$
\cM \coloneqq \{  \left( u(y), p(y) \right) \in V \, :\, y \in \rY \}.
$$

\subsection{Optimal reconstruction algorithms}
\label{sec:optimalAlgos}
In general, one would like to use an algorithm $A$ that is optimal in the sense of minimizing
$$
\inf_{A: \Wm \to V} E_{\wc}(A,\cM),\quad \text{or} \quad \inf_{A: \Wm \to V} E_{\ms}(A,\cM).
$$
However, as discussed in \cite{CDDFMN2019}, optimal algorithms are difficult to compute and even to characterize for general sets $\cM$. In this respect, the following is known:
\begin{itemize}
\item The problem of finding an algorithm $A$ that minimizes $E_{\wc}(A,\cM)$ is called \emph{optimal recovery}. It has been extensively studied for convex sets $\cM$ that are balls of smoothness classes. This is however not the case in the current setting since the solution manifold $\cM$ introduced in \eqref{eq:manifold} usually has a complex geometry. We know that as soon as $\cM$ is bounded there is a simple mathematical description of an optimal algorithm in terms of Chebyshev centers of certain sets. However, this algorithm cannot be easily computed due to the geometry and high dimensionality of the manifold.
\item The problem of finding an algorithm $A$ that minimizes $E_{\ms}(A,\cM)$ falls into the scope of \emph{Bayesian} or \emph{learning problems}. As explained in \cite{CDDFMN2019}, if the probability distribution on the manifold $\cM$ is Gaussian, the optimal algorithm can easily be characterized and computed. However, the assumption on a Gaussian distribution is very strong and will not hold in general so finding a computable optimal algorithm in the mean-square sense is also an open problem.
\end{itemize}
These theoretical difficulties motivate the search for suboptimal yet fast and good recovery algorithms. One vehicle for this has been to build linear recovery algorithms $A\in\cL(\Wm, V)$ (see \cite{MMPY2015, MPPY2015, CDDFMN2019}). However, since in general it is not clear that linear algorithms will be optimal, we use in this work piecewise linear  reconstruction algorithms as a trade-off between optimality and computational feasibility and rapidity.


\subsection{Piecewise linear reconstruction algorithms using reduced modeling}
\label{sec:algos}
Reduced models are a family of methods that produce each a hierarchy of spaces $(V_n)_{n\geq 1}$  that approximate the solution manifold well in the sense that
\begin{equation}
\label{eq:error-manifold}
\eps_n \coloneqq \sup_{u\in\cM} \dist(u, V_n)  \,
,\quad \text{or} \quad
\delta^2_n \coloneqq \bE\left(  \dist(u, V_n)^2\right) \, 
\end{equation}
decays rapidly as $n$ grows for certain classes of PDEs. The term $\dist(u, V_n)$ denotes the distance from $u$ to the space $V_n$, which is given by its projection error onto $V_n$,
$$
\dist(u, V_n) = \Vert u - P_{V_n} u \Vert, \quad \forall u \in V.
$$
Several methods exist to build spaces such that $(\eps_n )_{n\geq1}$ or $(\delta_n )_{n\geq1}$ decay fast. Some families are the reduced basis method (see \cite{RHP2007}), the (Generalized) Empirical Interpolation Method (see \cite{BMNP2004, MM2013, MMT2016}), Proper Orthogonal Decomposition (POD, \cite{sirovich1987,berkooz1993}) and low-rank methods (see \cite{CDS2011, CD2015acta}).

Linear reconstruction algorithms $A:W_m\to V$ that make use of reduced spaces $V_n$ are the Generalized Empirical Interpolation Method (GEIM) introduced in \cite{MM2013} and further analyzed in \cite{MMPY2015, MMT2016} and the Parameterized Background Data-Weak Approach (PBDW) introduced in \cite{MPPY2015} and further analyzed in \cite{BCDDPW2017}. Some extensions have been proposed to address measurement noise (see, e.g., \cite{ABGMM2017, Taddei2017}) and  other recovery algorithms involving reduced modelling have also been recently proposed (see \cite{KGV2018}). 

\subsubsection{PBDW, a linear recovery algorithm}
\label{sec:linearPBDW}
In this work, we take PBDW as a starting point for our recovery algorithm. Given a measurement space $\Wm$ and a reduced model $V_n$ with $1\leq n\leq m$, the PBDW algorithm $$A^\pbdw_{m,n}:\Wm \to V$$ gives for any $\omega \in \Wm$ a solution of
\begin{equation*}
\min_{u \in \omega + W^\perp } \dist(u, V_n).
\end{equation*}
This optimization problem has a unique minimizer 
\begin{equation}
\label{eq:uStar}
A^\pbdw_{m,n}(\omega) = u_{m,n}^*(\omega) \coloneqq \argmin_{u = \omega + W^\perp } \dist(u, V_n).
\end{equation}
as soon as $n\leq m$ and $\beta(\Vn,\Wm)>0$, which is an assumption to which we adhere in the following. For any pair of closed subspaces $(E,F)$ of $V$, $\beta(E, F)$ is defined as
\begin{equation}
\label{eq:infsup}
\beta(E,F):=\inf_{e\in E}\sup_{f\in F}\frac {\<e,f\>}{\|e\|\, \|f\|}=\inf_{e\in E}\frac {\|P_{F} e\|}{\|e\|} \in [0,1]
\end{equation}

As proven in \ref{appendix:linear-pbdw}, an explicit expression of $u_{m,n}^*(\omega)$ is
\begin{equation}
\label{eq:uStarExplicit}
u_{m,n}^*(\omega) = v^*_{m,n}(\omega) + \omega - P_\Wm v^*_{m,n}(\omega)
\end{equation}
with
\begin{equation}
\label{eq:vStarExplicit}
v^*_{m,n}(\omega) = \( P_{\Vn | \Wm} P_{\Wm | \Vn} \)^{-1} P_{\Vn | \Wm}(\omega),
\end{equation}
where, for any pair of closed subspaces $(X,Y)$ of $V$, $P_{X | Y}: Y \to X$ is the orthogonal projection into $X$ restricted to $Y$. The invertibility of the operator $P_{\Vn | \Wm} P_{\Wm | \Vn}$ is guaranteed under the above conditions.

Formula \eqref{eq:uStarExplicit} shows that $A^\pbdw_n$ is a bounded linear map from $\Wm$ to $\Vn\oplus(W_m\cap V_n^\perp)$. For any $u\in V$, the reconstruction error is bounded by
\begin{equation}
\Vert u - A^\pbdw_{m,n}(\omega) \Vert
\leq \beta^{-1}(V_n, \Wm) \Vert u - P_{\Vn\oplus(W_m\cap V_n^\perp)} u\Vert
\leq \beta^{-1}(V_n, \Wm) \Vert u - P_{\Vn}u \Vert
\label{eq:pbdw_bound}
\end{equation}
Depending on whether $\Vn$ is built to address the worst case or mean square error, the reconstruction performance over the whole manifold $\cM$ is bounded by
\begin{equation}
\label{eq:err-wc-pbdw}
e_{m,n}^\wcpbdw \coloneqq
E_{\wc}(A^\pbdw_{m,n}, \cM) \leq \beta^{-1}(V_n, \Wm) \max_{u\in \cM}\dist(u, V_n\oplus(V_n^\perp\cap\Wm)) \leq \beta^{-1}(V_n, \Wm) \, \eps_n,
\end{equation}
or
\begin{align}
e_{m,n}^\mspbdw \coloneqq
E_{\ms}(A^\pbdw_{m,n}, \cM)
&\leq \beta^{-1}(V_n, \Wm) \bE\left(\dist(u, V_n\oplus(V_n^\perp\cap\Wm)^2\right)^{1/2}  \nonumber\\
&\leq \beta^{-1}(V_n, \Wm) \, \delta_n, \label{eq:err-ms-pbdw}
\end{align}
Note that $\beta(\Vn, \Wm)$ can be understood as a stability constant. It can also be interpreted as the cosine of the angle between $\Vn$ and $\Wm$. The error bounds involve the distance of $u$ to the space $V_n\oplus(V_n^\perp\cap\Wm)$ which provides slightly more accuracy than the reduced model $\Vn$ alone. 
\dl{The additional approximation power given by $V_n^\perp\cap\Wm$ comes from the term $\omega- P_{W_m}v^*_{m,n}(\omega)$. As originally discussed in \cite{MPPY2015}, this term helps to correct to some extent the model error in the sense of making the final reconstruction be in a larger space than $V_n$, which is built from the parametric PDE model.} In the following, to ease the reading we will write errors only with the second type of bounds that do not involve the correction part on $V_n^\perp\cap\Wm$.

An important observation is that for a fixed measurement space $W_m$ (which is the setting in our numerical tests), the error functions
$$
n\mapsto e_{m,n}^\wcpbdw,\quad \text{and} \quad n\mapsto e_{m,n}^\mspbdw
$$
reach a minimal value for a certain dimension $n^*_\wc$ and $n^*_\ms$ as the dimension $n$ varies from 1 to $m$. This behavior is due to the trade-off between:
\begin{itemize}
\item the improvement of the approximation properties of $V_n$ as $n$ grows ($\eps_n$ and $\delta_n \to 0$ as $n$ grows)
\item the degradation of the stability of the algorithm, given here by the decrease of $\beta(V_n, \Wm)$ to 0 as $n\to m$. When $n> m$, $\beta(V_n, \Wm)=0$.
\end{itemize}
As a result, the best reconstruction performance with  PBDW is given by
$$
e_{m,n^*_\wc}^\wcpbdw = \min_{1\leq n \leq m} e_{m,n}^\wcpbdw,
\quad\text{or}\quad
e_{m,n^*_\ms}^\mspbdw = \min_{1\leq n \leq m} e_{m,n}^\mspbdw.
$$

We finish this section with \dl{five} remarks:
\begin{enumerate}
\item \fg{
In our application, the manifold $\mathcal{M}$ will be a family of incompressible fluid flow solutions of a parametric incompressible Navier-Stokes equation. In this case, $V_n$ is usually built such that all functions $v \in V_n$ satisfy the divergence-free condition $\nabla \cdot v = 0$  As a result, a reconstruction with only $v^*_{m,n}$ (see Appendix A, equation A1) will yield a divergence-free approximation of the flux. Note however that the full PBDW reconstruction $u_{m,n}^*$ does not guarantee this property since the model bias corrector $\omega - P_{W_m} v_{m,n}^*$ may not be divergence free. The mass conservation of the reconstruction could be enforced, for instance, by considering its projection on divergence-free fields. That is to say, let us take a reconstructed vector field $u_{m,n}^*$. A Helmholtz decomposition leads to the vector field $v_{\text{Helmholtz}}$ and an scalar field $\phi_{\text{Helmholtz}}$ such that $u_{m,n}^* = \nabla \times v_{\text{Helmholtz}} - \nabla \phi_{\text{Helmholtz}}$. One might solve therefore the following Laplacian problem: Find $\phi_{\text{Helmholtz}} \in H^1(\Omega)$ such that:
\begin{equation}
\begin{aligned}
 - \Delta \phi_{\text{Helmholtz}} = \nabla \cdot u_{m,n}^* &  ~~\text{in}~~ \Omega \\
\phi_{\text{Helmholtz}} = 0 &~~\text{on} ~~ \Gamma_{\inlet} \\
\nabla \phi_{\text{Helmholtz}} \cdot n = 0 & ~~\text{on}~~ \Gamma_\outlet^1 \\
\nabla \phi_{\text{Helmholtz}} \cdot n = 0 & ~~\text{on}~~ \Gamma_\outlet^2 \\
\nabla \phi_{\text{Helmholtz}} \cdot n = 0 & ~~\text{on}~~ \Gamma_\text{w}
\end{aligned}
\end{equation}
The solenoidal field will thus be $u_{m,n}^* + \nabla \phi_{\text{Helmholtz}}$. Notice that the inlet boundary condition is arbitrary but nonetheless it does not influence the results since we are interested on the gradient of the field.
}
\item As already brought up, we do not consider model error in this work. However, the PBDW algorithm can correct to some extent the model misfit (through the term $\eta^*$, see \ref{appendix:linear-pbdw}).
\item Note that in the present setting the measurement space $\Wm$ is fixed and we will adhere to this assumption in the rest of the paper. In our application, this is reasonable since the nature and location of the sensors is fixed by the technology of the device and by the position of the probe which the medical doctor considers best. A different, yet related problem, would be to optimize the choice of the measurement space $\Wm$. Two works on this topic involving greedy algorithms are \cite{MMT2016, BCMN2018}. They have been done under the same setting involving reduced modelling that is presented in this work. More generally, the problem of optimal sensor placement has been extensively studied since the 1970's in control and systems theory (see, e.g. \cite{Ai,CK,YS}). One common feature with \cite{MMT2016, BCMN2018} is that the criterion to be minimized by the optimal location is nonconvex, which leads to potential difficulties when the number of sensors is large.
\item \om{Very often, the reduced spaces $V_n$ that are used come from forward reduced modelling methods which produce a sequence of nested spaces $V_1\subset V_2 \subset \dots$. These spaces have very good approximation properties in the sense that the sequence $(\eps_n)$ (or $(\delta_n)$) has a quasi-optimal decay rate which is comparable to the Kolmogorov width for elliptic and parabolic problems (see \cite{BCDDPW2011}). Therefore, they give an almost optimal linear approximation for the forward problem of mapping the parameters to their PDE solution. However, since here we are dealing with state estimation, which is an inverse problem, the spaces for forward modelling may not be optimal for reconstruction and the question of finding the optimal space arises. This topic goes beyond the present work but quantitative answers to the question have been given in  \cite{CDDFMN2019}. Also, in \cite{galarce2020}, several practical algorithms to build spaces $V_n$ better tailored for inverse problems were explored and tested in similar test cases as the ones of the present paper.}
\item \dl{The nature of the sensing device or imaging technique is given through the linear functionals $\ell_i$, which model the sensor's physical response. In the present work, the $\ell_i$ model Doppler Ultrasound images (see \eqref{eq:the_measures1D}). Other types of sensor devices or imaging techniques can be taken into account by adapting the definition of the $\ell_i$. For example, in applications involving  MRI, radio-astronomy or diffraction tomography, the $\ell_i$ give samples from the Fourier transform (see \cite{VGCR2010}).}
\end{enumerate}

\subsubsection{Piecewise linear reconstruction algorithm}
\label{sec:piecewise}

In our application, we can build an improved reconstruction algorithm by exploiting the fact that we are not only given a Doppler image at the time of reconstruction, but we also know in real-time the value of some parameters like time and the heart-rate of the patient. In other words, the vector or parameters $y$ can be decomposed into a vector $y^{\obs}$ of $p^{\obs}$ parameters ranging in $Y^{\obs}\subseteq \bR^{\obs}$ and a vector $y^{\unobs}$ of $p^{\unobs}$ unobserved parameters ranging in $Y^{\unobs}\subseteq \bR^{\unobs}$, with $p=p^{\obs}+p^{\unobs}$. In other words,
$$
y = (y^{\obs}, y^{\unobs}) \in Y^{\obs}\times Y^{\unobs}
$$
We can exploit this extra knowledge by building a partition of the parameter domain $\rY$ as follows: we first find an appropriate partition of the observed parameters into $K$ disjoint subdomains
$$
Y^{\obs} = \cup_{k=1}^K Y^{\obs}_k,\quad\text{and}\quad Y^\obs_k \cap Y^\obs_{k'}=\emptyset, \, k\neq k'.
$$
The strategy followed to find such a partition in our case is explained in Section \ref{subsec:optimal_n}. This yields a partition of the whole parameter domain
\begin{equation}
\label{eq:partition-Y}
Y = \cup_{k=1}^K Y_k, \quad \text{with } Y_k = Y_k^\obs \times Y^\unobs.
\end{equation}
This partition induces a decomposition of the manifold $\cM$ into $K$ different disjoint subsets $\cM_k$ such that
\begin{equation}
\label{eq:partition-M}
\cM = \cup_{k=1}^K \cM_k,\quad\text{and}\quad \cM_k\cap \cM_{k'}=\emptyset, \, k\neq k'.
\end{equation}

With this type of partition, we know in which subset we are at the time of reconstruction. We can thus build reduced models $(V_n^{(k)})_{n=1}$ for each subset $\cM_k$ and then reconstruct with the linear or affine PBDW. Proceeding similarly as in the previous
section, the reconstruction performance on subset $\cM_k$ is, for a fixed $n\leq m$,
\begin{equation*}
e_{m,n}^\wcaffk = E_{\wc}(A_{m,n}, \cM_k) \leq \beta^{-1}(V_n^{(k)}, \Wm) \, \eps^{(k)}_n,
\end{equation*}
or
\begin{equation*}
e_{m,n}^\msaffk = E_{\ms}(A_{m,n}, \cM_k) \coloneqq \bE\left( \Vert u - A_{m,n}(P_\Wm u)\Vert^2\right)^{1/2}\leq \beta^{-1}(V_n^{(k)}, \Wm) \, \delta^{(k)}_n.
\end{equation*}
The advantage of this piecewise approach is that the approximation errors  $(\eps^{(k)}_n)_n$ or $(\delta^{(k)}_n)_n$ in each subdomain $\cM_k$ may decay faster than in the whole manifold (sometimes even significantly faster if each partition deals with very different physical regimes).

The best reconstruction performance for $\cM_k$ is thus 
$$
e_{m,n^*_\wc(k)}^\wcaffk = \min_{1\leq n \leq m} e_{m,n}^\wcaffk,
\quad\text{or}\quad
e_{m,n^*_\ms(k)}^\msaffk = \min_{1\leq n \leq m} e_{m,n}^\msaffk.
$$
It follows that the performance in $\cM=\cup_{k=1}^K \cM_k$ is
$$
e_{m}^\wcaff = \max_{1\leq n \leq m} e_{m,n^*_\ms(k)}^\wcaffk,
\quad\text{or}\quad
e_{m,n^*_\ms(k)}^\msaffk = \sum_{k=1}^K \omega_k e_{m,n^*_\ms(k)}^\msaffk,
$$
where $\omega_k = p( u\in \cM_k )$.


\section{Reconstruction of non-observable Quantities of Interest in fluid flows}
Our task is to use Doppler velocity measurements taken from a fluid flow and to reconstruct:
\begin{itemize}
\item \textbf{Partially observable quantities:} the full 3D velocity flow in $\Omega$ and related quantities such as the wall shear stress and vorticity.
\item \textbf{Non-observable quantities:} the full 3D pressure flow in $\Omega$ and the pressure drop.
\end{itemize}
Our strategy to address this task consists essentially in two steps:
\begin{itemize}
\item We apply the piecewise linear reconstruction algorithm of section \ref{sec:piecewise} where the key is to do a \emph{joint reconstruction} of 3D velocity and pressure.
\item We then derive the related quantities of interest as a simple by-product (wall shear stress and vorticity).
\end{itemize}

\subsection{Joint reconstruction of velocity and pressure}
\label{sec:joint_up}
For the reasons explained in section \ref{sec:instantiation}, the couple $(u, p)$ of velocity and pressure belongs to the Cartesian product
$$
V = U \times P = [H^1(\Omega)]^d \times  L^2(\Omega)
$$
It is assumed to be the solution to the parameter-dependent Navier-Stokes equations \eqref{eq:Navier_Stokes} for some parameter $y\in Y$. Some elements $y^\obs$ are observed but others are not so we cannot directly solve \eqref{eq:Navier_Stokes} with the parameters set to $y$. We therefore use the piecewise linear reconstruction of section \ref{sec:piecewise}. For this, it is necessary to endow $V$ with the external direct sum and product structure to build a Hilbert space. That is, for any two elements $(u_1, p_1)$ and $(u_2, p_2)$ of $V=U\times P$ and any scalar $\alpha\in \bR$,
$$
(u_1, p_1)+(u_2, p_2) = (u_1 + u_2, p_1+p_2),\quad \alpha(u_1, p_1) = (\alpha u_1, \alpha p_1)
$$
The inner product is defined as the sum of component-wise inner products
$$
\left<  (u_1, p_1), (u_2, p_2)  \right>_V = \left<  u_1, u_2  \right>_U +  \left<  p_1, p_2  \right>_P, 
$$
and it induces a norm on $V$,
$$
\Vert (u,p) \Vert \coloneqq \left( \left<  (u, p), (u, p)  \right>_V\right)^{1/2}, \quad \forall (u,p) \in V.
$$

When we are given partial information on $(u,p)$ from Doppler velocity measures, we are given the projection
$$\omega = P_{W_m}(u,p)
$$
where $W_m$ is the observation space
$$
\Wm \coloneqq W_m^{(u)} \times \{0\} = \vspan\{\omega_1,\dots, \omega_m\} \times \{0\} \subset V
$$
and the $\omega_i$ are the Riesz representers in $U$ of each voxel $\ell_i \in U'$,
$$
\left< \omega_i , v \right>_U = \ell_i(v) = \int_{\Omega_i} v \cdot b\, \dx, \quad \forall v \in U.
$$

We are now in position to apply directly the reconstruction algorithms from section \ref{sec:RecMeth} to do the joint reconstruction of $(u,p)$ with the current particular choice of Hilbert space $V$ and observation space $\Wm$. We briefly instantiate here the main steps. Let us assume that we have a reduced model
$$
\Vn \coloneqq \vspan \{(u_1,p_1),\dots, (u_n, p_n)\}
$$
of dimension $n\leq m$ that approximates
$$
\cM \coloneqq \{  \left( u(y), p(y) \right) \in V \, :\, y \in \rY \}
$$
with accuracy
\begin{equation}
\label{eq:error-manifold}
\eps_n \coloneqq \sup_{(u,p)\in\cM} \dist((u,p), V_n)  \,
,\quad \text{or} \quad
\delta^2_n \coloneqq \bE\left(  \dist((u,p), V_n)^2\right) \, 
\end{equation}
and which is such that $\beta(\Vn, \Wm)>0$. Then, we can reconstruct with the linear PBDW method (see equation \eqref{eq:uStar}) which, in the present case, reads
\begin{equation}
A^\pbdw_{m,n}(\omega) = (u_{m,n}^*(\omega), p_{m,n}^*(\omega)) \coloneqq \argmin_{ (u, p) = \omega + W^\perp } \Vert (u,p) -P_{V_n}(u,p)\Vert.
\end{equation}
The worst and average reconstruction errors are bounded like in estimates \eqref{eq:err-wc-pbdw} and \eqref{eq:err-ms-pbdw}, that is
\begin{equation}
e_{m,n}^\wcpbdw =\max_{(u, p)\in \cM} \Vert (u,p)-(u_{m,n}^*(\omega), p_{m,n}^*(\omega)) \leq \beta^{-1}(V_n, \Wm) \, \eps_n,
\end{equation}
or
\begin{align}
e_{m,n}^\mspbdw 
&= \bE\left( \Vert u - A^{(\pbdw)}_{m,n}(P_\Wm u)\Vert^2\right)^{1/2} \leq \beta^{-1}(V_n, \Wm) \, \delta_n,
\end{align}
If we build a partition of the manifold $\cM$ based on observed parameters, we can reconstruct with the piecewise linear algorithm of section \ref{sec:piecewise}.

\om{Before moving to the next section, we would like to note that, in general, stability is degraded in the joint reconstruction compared to the single velocity reconstruction. In fact, if the reduced model is taken as a product of two reduced spaces, namely, if $V_n = V_{n_u}^{(u)}\times V_{n_p}^{(p)}$ with $V_{n_u}^{(u)}\subset U$ and $V_{n_p}^{(p)}\subset P$, we can easily prove that if the inf-sup constant in the single velocity space $\beta(V_{n_u}^{(u)}, W_m^{(u)})>0$ (with $V_{n_u}^{(u)}$ and $W_m^{(u)} \in U$), then the inf-sup $\beta(V_{n_u}^{(u)}\times V_{n_p}^{(p)}, W_m^{(u)}\times\{0\})$ of the joint reconstruction satisfies
$$
0 < \beta(V_{n_u}^{(u)}\times V_{n_p}^{(p)}, W_m^{(u)}\times\{0\}) \leq \beta(V_{n_u}^{(u)}, W_m^{(u)}).
$$
On the one hand, the left-hand side of the bound tells that the joint reconstruction is well-posed as soon as the single velocity reconstruction is. On the other hand, the right-hand side says that stability cannot be better than the one of the single velocity reconstruction.}

\subsection{Reconstruction of related quantities}

\subsubsection{Pressure drop}
\label{sec:pdrop}
The pressure drop is a quantity that has traditionally been of high interest to the medical community since it serves to assess, for instance, the severity of stenosis in large vessels due to the accumulation of fat in the walls. Decomposing the domain boundary $\partial \Omega$ of a generic arterial bifurcation into the inlet, the wall and the outlet parts
$$
\partial \Omega = \Gamma_{\inlet} \cup \Gamma_w \cup \Gamma_{\outlet}^1 \cup \ldots \cup \Gamma_{\outlet}^l,
$$
the quantities to retrieve are
\begin{equation}
\delta p_i = \frac{1}{\abs{\Gamma_{\inlet}}} \int_{\Gamma_{\inlet}} p ~ \ds - \frac{1}{\abs{\Gamma_{\outlet}^i}} \int_{\Gamma_{\outlet}^i} p ~\ds,
\label{eq:pdrop}
\end{equation}
for the outlet labels $i=1,\ldots,l$.

\paragraph{Method 1 -- From the joint reconstruction $(u^*_n, p^*_n)$:} If we reconstruct $(u^*_n, p^*_n)$, we can easily approximate the pressure drop by
\begin{equation*}
\delta p_i^* = \frac{1}{\abs{\Gamma_{\inlet}}} \int_{\Gamma_{\inlet}} p^*_n ~\ds  - \frac{1}{\abs{\Gamma_{\outlet}^i}} \int_{\Gamma_{\outlet}^i} p^*_n ~\ds
\end{equation*}
for $i=1,\ldots,l$.

As we will see in our numerical results, the pressure drop is approximated at very high accuracy with $\delta p_i^*$ . We next provide a theoretical justification.

For this, we remark that we can view $\delta p_i$ as a bounded linear mapping from $V=U\times P$ to $\bR$ defined as
$$
\delta p_i ( (u, p) ) = \frac{1}{\abs{\Gamma_{\inlet}}} \int_{\Gamma_{\inlet}} p ~ \ds - \frac{1}{\abs{\Gamma_{\outlet}^i}} \int_{\Gamma_{\outlet}^i} p ~ \ds,\quad \forall (u,p) \in V.
$$
Thus the reconstruction error is given by
$$
\vert \delta p_i ( (u,p) ) - \delta_i p ( (u^*_n,p^*_n) ) \vert .
$$
Exploiting the linearity of $\delta p_i$, one can derive the simple bound
\begin{align}
\vert \delta p_i ( (u,p) ) - \delta p_i ( (u^*_n,p^*_n) ) \vert
&= \vert \delta p_i \left( \left(u,p\right) - \left(u^*_n,p^*_n\right) \right) \vert \\
&\leq \Vert \delta p_i \Vert_{V'} \Vert (u,p) - (u^*_n,p^*_n) \Vert \\
&\leq \Vert \delta p_i \Vert_{V'} \beta^{-1}(\Vn, \Wm) \Vert (u,p) - P_{\Vn} (u,p) \Vert \\
&\leq \Vert \delta p_i \Vert_{V'} \beta^{-1}(\Vn, \Wm) \eps_n
\end{align}
where we have used \eqref{eq:err-wc-pbdw} between the second and the third line and where
$$
\Vert \delta p_i \Vert_{V'} \coloneqq \sup_{(u,p)\in V} \frac{\vert \delta p_i (u,p) \vert}{\Vert (u,p) \Vert} \geq 1 .
$$
As we will see below, this estimate is too coarse to account for the high reconstruction accuracy which is observed because the values $\beta(\Vn, \Wm)$ are close to zero and the product $\beta^{-1}(\Vn, \Wm) \eps_n$ is only moderately small. It is necessary to find a sharper estimate that involves finer constants in front of $\eps_n$ to account for the good reconstruction results. For this, observing that, by construction of $(u^*_n,p^*_n)$,
$$
P_{W_m}(u,p) = P_{W_m}(u^*_n,p^*_n),
$$
we have
$$
(u,p) - (u^*_n,p^*_n) \in \Wm^\perp
$$
so we can derive the new estimate
\begin{align}
\vert \delta p_i ( (u,p) ) - \delta p_i ( (u^*_n,p^*_n) ) \vert
&\leq \kappa_{m,n} \Vert (u,p) - (u^*_n,p^*_n) - P_{V_n}\left( (u,p) - (u^*_n,p^*_n)\right) \Vert \\
&\leq 2\kappa_{m,n} \eps_n
\label{eq:err-p-drop}
\end{align}
with
\begin{equation}
\kappa_{m,n} \coloneqq \sup_{(u,p) \in \Wm^\perp} \frac{ \vert \delta p_i ( (u,p) ) \vert}{ \Vert \dist\left( (u,p), V_n \right) \Vert}
\label{eq:kappa_bound_pdrop}
\end{equation}
As we illustrate in our numerical tests, the value of $\kappa_{m,n}$ is moderate and significantly smaller than the factor $\Vert \delta p_i \Vert_{V'} \beta^{-1}(\Vn, \Wm)$ of the previous estimate. As a result, the product $\kappa_{m,n} \eps_n$ is small, and we guarantee a reconstruction of the pressure with good accuracy.

\paragraph{Method 2 -- From the reconstruction of $u^*_n$ and the virtual works principle:} As an alternative to the joint reconstruction strategy, we can use a method introduced in \cite{bertoglio_pdrop} called Integral Momentum Relative Pressure estimator. As a starting point, it requires to work with a reconstruction $u^*_n$ of the velocity which, in our work, will be given by the PBDW method applied only to the reconstruction of the velocity field without pressure. We then estimate the pressure drop using the Navier-Stokes equations as follows. Assuming that $u^*_n$ satisfies perfectly the momentum conservation \eqref{eq:Navier_Stokes}, we test by a virtual and divergence free velocity field $v \in U$, 
\begin{equation}
\label{eq:IntNSv}
\underbrace{  \rho \int_{\Omega} \partial_t u^*_n \cdot v ~\dx}_{\displaystyle K(u^*_n,v)} 
\underbrace{ + \rho \int_{\Omega} (u^*_n \cdot \nabla u^*_n)\cdot v ~\dx}_{\displaystyle   I_{conv}(u^*_n,v)} 
\underbrace{ + \int_{\Omega} \nabla p \cdot v~\dx}_{\displaystyle I_{press}(p,v)} 
\underbrace{ - \mu\int_{\Omega} \Delta u^*_n \cdot v~\dx}_{ \displaystyle I_{visc}(u^*_n,v)} 
= 0.
\end{equation}
Using Green's identities, we can write
\begin{equation}
\begin{aligned}
I_{\text{conv}}(u^*_n,v) &= \rho \int_{\partial \Omega} (u^*_n \cdot n) (u^*_n \cdot v) ~\ds - \rho \int_{\Omega} (u^*_n \cdot \nabla v) \cdot u^*_n ~\dx. \\
I_{\text{visc}}(u^*_n,v) &= \mu\int_{\Omega} \nabla u^*_n : \nabla v ~\dx - \mu \int_{\partial \Omega}  (\nabla u^*_n \cdot n) \cdot v ~\ds. \\
I_{\text{press}}(u^*_n,v) &=  \int_{\partial \Omega} p  (v \cdot n) ~ \ds - \int_{\Omega} p  (\nabla \cdot v) ~\dx.  \\
\end{aligned}
\label{eq:int_by_parts}
\end{equation}
The current strategy requires to assume that the pressure field is constant over the inlet and outlets. Notice that, since $\nabla \cdot v = 0$, the following identity holds,
\begin{equation}
  I_{\text{press}}(p,v) = p \int_{\partial \Omega} v \cdot n ~ \ds= p_{in} \int_{\Gamma_\inlet} v \cdot n ~\ds + \sum_{i=1}^l p_{out}^i \int_{\Gamma_\outlet^i} v \cdot n ~\ds,
  \label{eq:h_p}
\end{equation}
where $p_{in}$ is the average pressure over $\Gamma_{in}$ and $p_{out}^i$ is the average pressure over the i-th outlet $\Gamma_{\text{out}}^i$. For $j=1,\dots, l$, we consider a function $v_j\in V$ satisfying $\nabla \cdot v_j = 0$ and $v_j = 0$ in $\Gamma_w$. Mass conservation for incompressible regimes implies
\begin{equation}
\label{eq:mass_cons_for_v}
\int_{\Gamma_{\inlet}} v_j \cdot n ~ \ds + \sum_{i=1}^l \int_{\Gamma_\outlet^i} v_j \cdot n ~ \ds = 0,
\end{equation}
for $j=1,\cdots,l$. As a result, it is possible to recover the mean pressure drop $x_j = p_{out}^j - p_{in}$ for each outlet $j=1,\cdots,l$ by solving an $l \times l$ system of equations 
\begin{equation}
F x = H(u^*_n),
\label{eq:Fmat}
\end{equation}
where $F \in \mathbb{R}^{l \times l}$ has entries
\begin{equation}
  F_{ij} = \int_{\Gamma_{\outlet}^j} v_i \cdot n ~\ds, 
  \label{eq:Aij_press_estimation}
\end{equation}
and,
\begin{equation}
  H_i(u^*_n)= -\left( I_{\text{visc}}(u^*_n,v_i) + I_{\text{conv}}(u^*_n,v_i) + K (u^*_n, v_i) \right).
  \label{eq:H_i}
\end{equation}
\om{The inversion of the system \eqref{eq:Fmat} is made trivial when the $v_i$ are chosen so that $F$ becomes diagonal (with nonzero entries). We remark that this is achieved if we choose the $v_i$ to be divergence-free and to have outgoing zero flux in all the outlets $\Gamma_{\text{out}}^j$ for $j\neq i$. For each $i=1,\dots, l$, $v_i$ can be characterized as the unique solution to the following Stokes problems:} Find $v_i \in U$ and $\lambda \in L^2(\Omega)$ auxiliary function such that:
\begin{equation}
\begin{aligned}
-\Delta v_i + \nabla \lambda &= (0,0,0) & \text{ in }& \Omega, \\
\nabla \cdot v_i &= 0 & \text{ in }& \Omega, \\
v_i &= ( 0,  0, 0 ) & \text{ on }& \Gamma_\text{w}, \\ 
v_i &= [ (1, 1, 1) \cdot n ] n & \text{ on }& \Gamma_{\inlet}, \\ 
v_i &= (0, 0, 0) & \text{ on }& \Gamma_{\outlet}^j & \forall j \neq i, \\
\dl{\( \frac{1}{2} \left( \nabla v_i + \nabla^T v_i \right)   + \lambda I_{3 \times 3}\)n} &\dl{=(0,0,0)}  & \dl{\text{ on }} & \dl{\Gamma_{out}^i}.
\end{aligned}
\label{eq:stokes_convenient_pdrop}
\end{equation}


In order to ensure good stability when doing the time integration of \eqref{eq:IntNSv}, we use the Cranck-Nicholson scheme.

We may note that this method requires the knowledge of the flow viscosity and density, and it assumes constant pressure over the inlets and outlets. This is contrast to the joint reconstruction approach which does not need these assumptions.

\subsubsection{Wall shear stress}
\label{subsubsec:wss}
The wall shear stress (WSS) has been proposed as an index of damage in vascular endothelial cells and atherosclerosis, a disease in which the blood coagulates close to the vessel walls. The works \cite{koskinas2009}, \cite{heo2014} or \cite{zarins1983} can serve as a reference. 


The WSS \fgal{is a mapping $\wss: U \rightarrow [H^{-1/2}(\Gamma_\text{w})]^3$, $\wss : \Omega \mapsto \bR^3$} that returns the tangential component of the force that the blood applies on the vessel wall
\begin{equation}
\wss(u) \coloneqq \dl{2\mu} \left\lbrace I - n \otimes n \right\rbrace \left( \frac{\nabla u + \nabla u^T}{2} n \right),\text{ on } \Gamma_\text{w}.
\label{eq:wss_def}
\end{equation}

\om{Note that \fgal{$\cS(u) \in [H^{-1/2}(\Gamma_\text{w})]^3$} because a velocity solution $u$ of the Navier-Stokes equations satisfies \fgal{$(\nabla u) n \in [L^2(\Omega)]^3$} and, by Green's formula, we can prove that $(\nabla u) n \vert_{\Gamma_\text{w}} \in \fgal{ [H^{-1/2}(\Gamma_\text{w})]^3}$.}
\om{
Our goal is therefore to compute the reconstruction error $\Vert \cS(u) - \cS(u^*) \Vert_{[H^{-1/2}(\Gamma_\text{w})]^3}$.  First, we have that
$$
\Vert \cS(u) - \cS(u^*) \Vert_{[H^{-1/2}(\Gamma_\text{w})]^3}
\lesssim \Vert \lambda_1 \Vert_{[H^{-1/2}(\Gamma_\text{w})]^3}
+
| \overline{\cS(u)} - \overline{\cS(u^*)} |
$$
\fgal{where $\overline{\cS(u)} \coloneqq \int_{\Gamma_\text{w}} \cS(u)(x) ~\dx$ and $\lambda_1$ is a vector field defined as
$$
\lambda_1 \coloneqq  \cS(u) - \cS(u^*) - \left( \overline{\cS(u)} - \overline{\cS(u^*)} \right) \; \in [H^{-1/2}(\Gamma_\text{w})]^3
$$}
is a function of zero mean.
}
\om{
We next prove that there exists a constant $C>0$ such that
\begin{equation}
\label{eq:g}
\Vert \lambda_1 \Vert_{[H^{-1/2}(\Gamma_\text{w})]^3} \leq C \Vert \phi_{\lambda_1} |_{\Gamma_\text{w}} \Vert_{\[L^2(\Gamma_w)\]^3} 
\end{equation}
where $\phi_{\lambda_1} \in [H^1(\Omega)]^3$ is the unique solution to the following homogeneous Laplace equation with Neumann boundary condition: Find $\phi_{\lambda_1} \in [H^1(\Omega)]^3$ with $\overline{\phi_{\lambda_1}}=0$ such that
\fgal{
\begin{equation}
\label{eq:neumann}
\int_\Omega \nabla \phi_{\lambda_1} : \nabla v ~\dx= \int_{\Gamma_\text{w}} \lambda_1 \cdot v~\ds = \left< \lambda_1, \text{Tr}(v) \right>_{ [H^{-1/2}(\Gamma_\text{w})]^3, [H^{1/2}(\Gamma_w)]^3}, \quad \forall v \in [H^1({\Omega})]^3,
\end{equation}
where $\text{Tr}:[H^1(\Omega)]^3\to [H^{1/2}(\Omega)]^3$} is the trace operator. By applying the Cauchy-Schwarz inequality in \eqref{eq:neumann}, we derive
$$
\left< \lambda_1, \text{Tr}(v) \right>_{ [H^{-1/2}(\Gamma_\text{w})]^3, [H^{1/2}(\Gamma_\text{w})]^3}
\leq \Vert \nabla \phi_{\lambda_1} \Vert_{[L^2(\Omega)]^3} \Vert \nabla v \Vert_{[L^2(\Omega)]^3}
\leq  \Vert \phi_{\lambda_1} \Vert_{[H^1(\Omega)]^3} \Vert v \Vert_{[H^1(\Omega)]^3},
$$
from which we deduce that
$$
\Vert \lambda_1 \Vert_{[H^{-1/2}(\Gamma_\text{w})]^3} \leq \Vert \phi_{\lambda_1} \Vert_{[H^1(\Omega)]^3}.
$$
Finally, by continuity of the mapping $\lambda_1 \in [L^2(\Gamma_\text{w})]^3 \mapsto \phi_{\lambda_1} \in [H^1(\Omega)]^3$, there exists a constant $C>0$ such that $\Vert \phi_{\lambda_1} \Vert_{[H^1(\Omega)]^3} \leq C \Vert \lambda_1 \Vert_{[L^2(\Gamma_\text{w})]^3}$. This yields our final, computable bound
$$
\Vert \cS(u) - \cS(u^*) \Vert_{[H^{-1/2}(\Gamma_\text{w})]^3}
\leq C ( \Vert \text{Tr}(\phi_{\lambda_1}) \Vert_{[L^2(\Gamma_\text{w})]^3} 
+
| \overline{\cS(u)} - \overline{\cS(u^*)} | ).
$$
}

\om{
Thus, to evaluate the reconstruction quality of the WSS at a given time $t$ we compute:
\begin{equation}
e_{\text{wss}}(t) = \frac{\Vert \text{Tr}(\phi_{\lambda_1}) \Vert_{[L^2(\Gamma_\text{w})]^3} + | \overline{\cS(u(t))} - \overline{\cS(u^*(t))} |}{\Vert \text{Tr}(\phi_{\lambda_2}) \Vert_{[L^2(\Gamma_\text{w})]^3} + | \overline{\cS(u(t))} |}
\label{eq:error_wss}
\end{equation}
where $\lambda_2 = \cS(u(t)) - \overline{\cS}(u(t))$.
}

\subsubsection{Vorticity}
\label{subsubsec:vort}
The vorticity is defined as $\vorticity = \nabla \times u$. It provides clinical information about the shear layer thickness, which has been correlated with thrombus formation and hemolysis \cite{bluestein2010}. In general, vorticity is connected to the assessment of the cardiovascular function and there have been efforts to reconstruct it from magnetic resonance images (see, for instance: \cite{garcia2013}).

The relative $L^2$ error in time for the vorticity reconstruction $\vorticity^* = \nabla \times u^* = \nabla \times A(P_{W_m}u)$ is given by
\begin{equation}
e_{\text{vorticity}}(t) = \frac{\norm{ \vorticity(t) - \vorticity^*(t)}}{\left( \int \norm{\vorticity(t)}^2 dt \right)^{1/2}},
\label{eq:err_vorticity_time}
\end{equation}

Since $u \in U=[H^1(\Omega)]^3$ with $\nabla\cdot u = 0$ (incompressible flow) and since we have the identity (see, e.g., \cite{gmj2005})
\begin{equation}
\norm{\nabla \times u }^2_{L^2{(\Omega)}} = \norm{\nabla u}^2_{L^2{(\Omega)}} + \norm{\nabla \cdot u }^2_{L^2{(\Omega)}} = \norm{\nabla u}^2_{L^2{(\Omega)}},
\label{eq:rot_identity}
\end{equation}
it follows by \eqref{eq:pbdw_bound} that
\begin{equation}
\Vert \vorticity - \vorticity^* \Vert_{L^2(\Omega)} \leq \Vert u - u^* \Vert_U \leq \beta^{-1}(V_n, W_m) \Vert u - P_{V_n} u \Vert_U
\label{eq:vorticity_bound}
\end{equation}

\section{Noise-free numerical test in a carotid geometry}
\label{sec:noise-free_test}

In what follows, the numerical experiments shown were computed using two softwares, one under continuous development and maintained by the COMMEDIA team at INRIA: the Finite Elements for Life Sciences and Engineering, FeLiScE, and another one implemented especially for this work: the Multi-physics for biomedicAl engineering and Data assimilation, MAD. In addition, tetrahedron meshing and optimization is done using Mmg (see \cite{mmg3d}).

\fg{In the tests, the synthetic ultrasound device is placed in such a way that we measure only in one half of the mid plane of the working domain as illustrated in Figure  \ref{fig:measures_CFI_syn}. The ultrasound wave forms an angle of $\pi/4$ respect to the dominant fluid direction in the main carotid branch, that is to say, $b = \[\sqrt{2}/2, \sqrt{2}/2, 0 \]$. Note that in this situation the device does not sense velocity after the artery bifurcation downstream. This mimics potential anatomical constrains which may arise in a real application context and which may prevent from taking images in the whole domain of interest. We illustrate that our method gives accurate reconstructions in the whole domain}
\begin{figure}[!htbp]
\centering
\includegraphics[height=8cm,angle = -90]{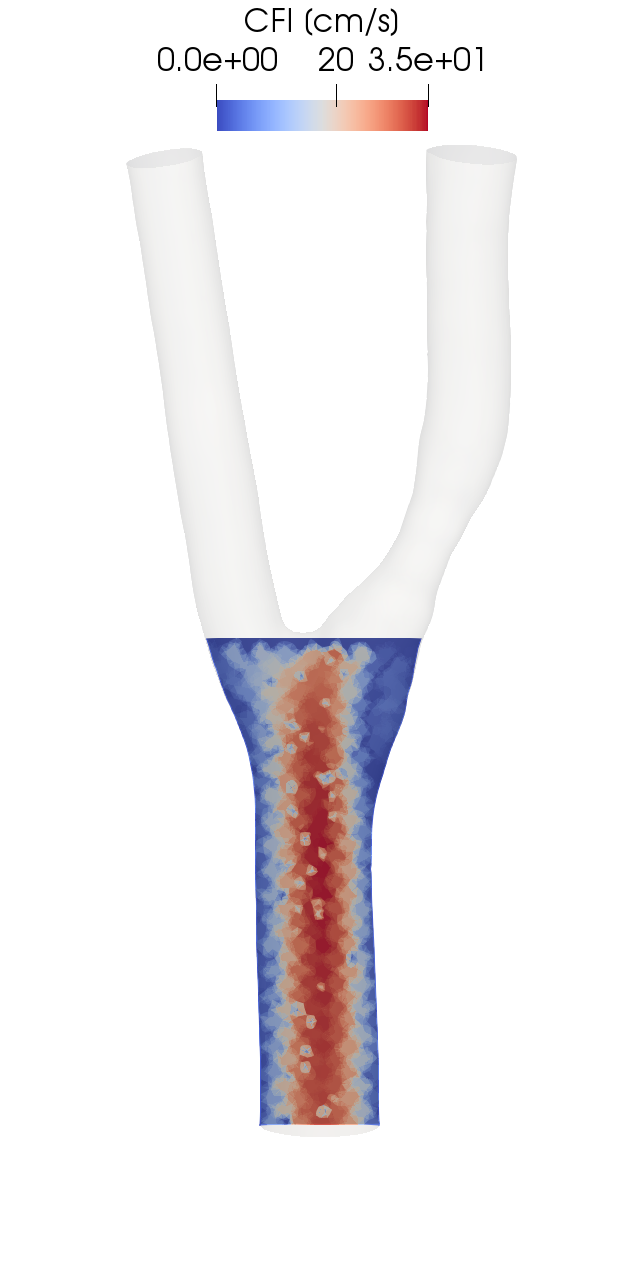}
\caption{\fg{
Synthetic CFI of the common carotid branch with 233 voxels of $0.15$ [cms] each (image from the systole period).}}
\label{fig:measures_CFI_syn}
\end{figure}

\subsection{Sampling $\cM$ with the Navier-Stokes equations}
\label{sec:sampling_M}

We consider the incompressible Navier-Stokes equations (NSE) as given in \eqref{eq:Navier_Stokes}. These equations are closed by adding a zero initial condition and the following boundary conditions:
\begin{itemize}
\item No-slip for the vessel wall, that is, $u = (0,0,0)^T$ on $\Gamma_\text{w}$.
\item The inlet boundary $\Gamma_\inlet$ lies in the $xz$ plane and we apply a Dirichlet condition $u = [0,u_{\text{in}},0]^T$. The component $u_{\text{in}}$ is a function $u_{\text{in}}(t, x, z) = u_0~ g(t) f(x, z)$, where:
\begin{itemize}
\item $u_0\in \mathbb{R}^+$ is an scaling factor. The function $g(t)$ is taken from flow curves in the common carotid arteries borrowed from \cite{blanco2014}. Its behavior is given in Figure \ref{fig:q_inflow_common_carotid}.
\item The function $f(x,z)$ is a 2D logit-normal distribution
\begin{equation}
\begin{aligned}
  f(x,z) =& \frac{\text{exp} \{ -0.5 \(\text{log} \left(\frac{x}{1-x}\right)  - s  \)^2 \}}{x(1-x)z(1-z)} \\
       -& \frac{\text{exp} \{ 0.5  \(\text{log} \left(\frac{z}{1-z}\right) \)^2 \}}{x(1-x)z(1-z)},
  \label{eq:logit_normal_inlet}
\end{aligned}
\end{equation}
where the parameter, $s\in \mathbb{R}^+$, controls the axial symmetry of the inlet flow. \\
\end{itemize}
\begin{figure}[!htbp]
  \centering
  \includegraphics[height = 5 cm]{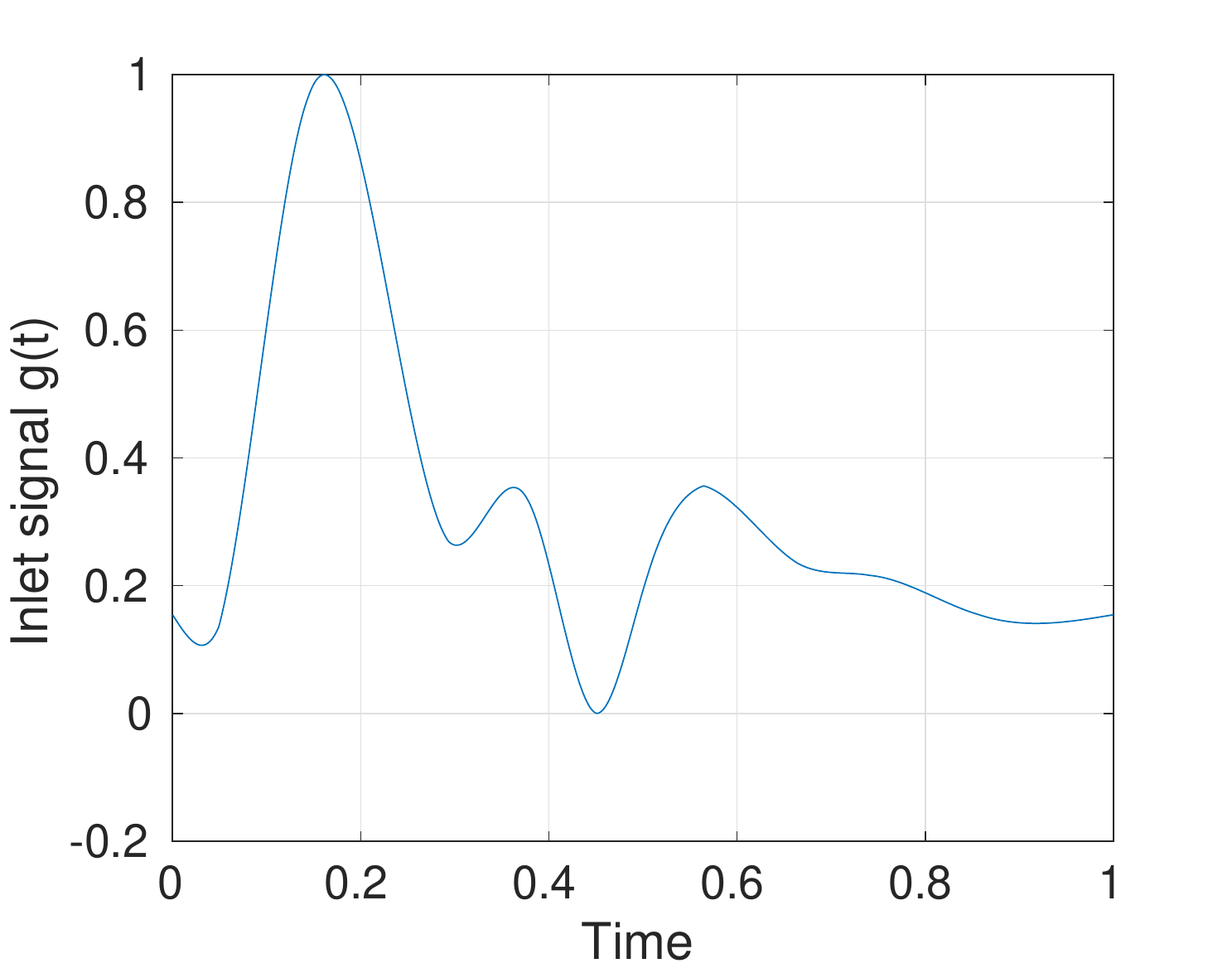}
  \caption{The function $g(t)$ for the inlet boundary condition.}
  \label{fig:q_inflow_common_carotid}
\end{figure}
\item For the outlet boundaries $\Gamma_\outlet^1$ and $\Gamma_\outlet^2$, we use a Windkessel model (see \cite{fqv2009} for a survey on 0-D models in haemodynamics), which gives the average pressure over each $\Gamma_\outlet^k$,
$$
\bar p_{o,k} = p_d^k + R_p^k \int_{\Gamma_\outlet^k} u \cdot n ~\ds ,\quad k=1,\,2
$$
where $p_d^k \in \mathbb{R}$ is called \emph{distal pressure} and is the solution to the ordinary differential equation:
\begin{equation}
\label{eq:pd}
\begin{cases}
C_d^k \dfrac{d p_d^k}{d t} + \dfrac{p_d^k}{R_d^k} &= \int_{\Gamma_\outlet^k} u \cdot n ~ \ds \\
p_d^k(t=0) &= p_{d, k} \text{ given.}
\end{cases}
\end{equation}

This model aims to represent the cardiovascular system behavior beyond the boundaries of
the working domain with a minimal increase in the computational cost. It is based on an analogy between flow and pressure with current and voltage in electricity. This is the reason why $C_d^k$ is called distant capacitance and $R_p^k$ and $R_d^k$ are respectively called proximal and distant resistances. These three parameters are positive real numbers.
\end{itemize}

\fgal{In the present work, we use a semi-implicit time discretization, and finite elements for the space discretization. This is described alongside with the ODE coupling for the distal pressures in \ref{appendixDisc}}

Now that the model has been introduced, let us define the set of solutions that we consider in our numerical experiments. We set,
\begin{equation}
\begin{aligned}
\rho & = 1 \; \text{g}/\text{cm}^3 \\
\mu &= 0.03 \; \text{Poise} \\
C_d^k &= 1.6 \times 10^{-5} \text{ for $k=1,2$} \\
R_{p}^k &= 7501.5  \text{ for $k=1,2$} \\
p_{d,k} &= 1.06 \times 10^5 \text{ for $k=1,2$} \\
R_d^k & = 60012 \\
\end{aligned}
\end{equation}
We introduce the ratio of the distal resistances for the Windkessel model at the outlets 
$$
\eta \coloneqq R_{d}^1 / R_d^{2} = 60012 / R_d^{2}.
$$
This parameter plays an important role since it impacts on how the blood flow splits between the two branches. When $\eta\to 0$ or $\infty$, one branch is obstructed and the blood tends to flow through the other branch. In the following, we call this situation an arterial blockage. When $\eta\approx 1$, the flow splits more or less equally and there is no blockage.

We define the heart rate as the number of cardiac cycles per minute, that is,
$$
\text{\HR} \coloneqq 60/T_c,
$$
where $T_c>0$ is the cardiac cycle duration expressed in seconds. We have $T_c = T_{sys}+T_{dia}$, where $T_{sys}$ and $T_{dia}$ are the duration of the systole and diastole respectively.

Our $\cM$ is generated by the variations of the following six parameters
\begin{equation}
\label{eq:params}
\begin{aligned}
t & \in [0, T] \\
\text{\HR} & \in [48, 120] \\
s & \in [0, 0.2]  \\
T_{sys} & \in [0.2863, 0.3182] \; \text{s.} \\
u_0 & \in [17, 20] \; \text{cm/s} \\
\eta & \in [0.5, 1.5] 
\end{aligned}
\end{equation}
We emphasize that time is seen as a parameter, \fg{and the simulation time $T$ for each solution depends on the heart rate, that is to say: $T = 60 / \text{\HR}$  . The parameter set is thus}
$$
\rY =  \{  (t, \HR, s, T_{sys}, u_0, \eta) \in \bR^6 \,:\,   t\in [0,T]\,   \text{\HR} \in [48, 120],\,   s\in  [0, 0.2],\, \dots   \} \subset \bR^6
$$
and the set of solutions is
$$
\cM \coloneqq \{  u(y) \in [H^1(\Omega)]^3\, :\, y \in \rY \}.
$$
The computation of reduced models involves a discrete training subset $\widetilde{\cM} \subset \cM$ which, in the experiments below, involves $\#\widetilde{\cM} = 21513 $ snapshots $u(y)$. The parameters are chosen from a uniform random distribution and we only save the solutions during the second cardiac cycle of each simulation. \dl{We considered a zero initial condition to start the simulation. Hence, the first cardiac cycle is less representative of the flow nature compared to the subsequent cycles.}

\subsection{Optimal partitioning and optimal dimension of $V_n$}
\label{subsec:optimal_n}

During ultrasound examination, we have access to the patient's heart rate $\HR$ and the time $t$ of the cardiac cycle. We can therefore decompose the vector $y$ of parameters as
$$
y = (y^\obs, y^\unobs), \quad y^\obs = (t, \HR), \quad y^\unobs=(s, T_{sys}, u_0, \eta).
$$
and use the piecewise reconstruction algorithm introduced in section \ref{sec:piecewise}. For this, we need to find an appropriate partition of $Y^\obs$ which will yield a partition of the whole parameter domain and a manifold decomposition as in equations \eqref{eq:partition-Y} and \eqref{eq:partition-M}.

The strategy that we have followed consists in computing first a training subset $\widetilde\cM$ of snapshots. We next consider a splitting of the time interval into $K\in \bN^*$ uniform subintervals
\fgal{
$$
[0, T] = \( \cup_{k=0}^{K-2} \tau_k \) \cup [(K-1) T/K, T], \quad \text{with }  \tau_k=[ k T/K, (k+1)T/K [
$$}
We proceed similarly for the heart rate's interval and split it into $K'\in \bN^*$ uniform subintervals,
\fgal{
$$
[48, 120] = \(\cup_{k'=0}^{K'-2} h_{k'}\) \cup [ 48 + 72(K'-1)/K', 120], \quad \text{with }  h_{k'}=[ 48 + 72k'/K', 48 + 72(k'+1)/K' [.
$$}
For fixed $(K,K')$, we have the partition in the parameter domain (see Figure \ref{fig:manifold_partitioning})
$$
Y^\obs = \bigcup_{(k,k') \in \{0,\dots, K-1\}\times \{0,\dots, K'-1\} } \tau_k \times h_{k'},
\qquad
Y = \bigcup_{(k,k') \in \{0,\dots, K-1\}\times \{0,\dots, K'-1\} } \tau_k \times h_{k'} \times Y^\unobs
$$
and the induced partition in the manifold
$$
\cM = \bigcup_{(k,k') \in \{0,\dots, K-1\}\times \{0,\dots, K'-1\} } \cM^{(k,k')}
$$
\begin{figure}[!htbp]
\centering
\includegraphics[height=5cm]{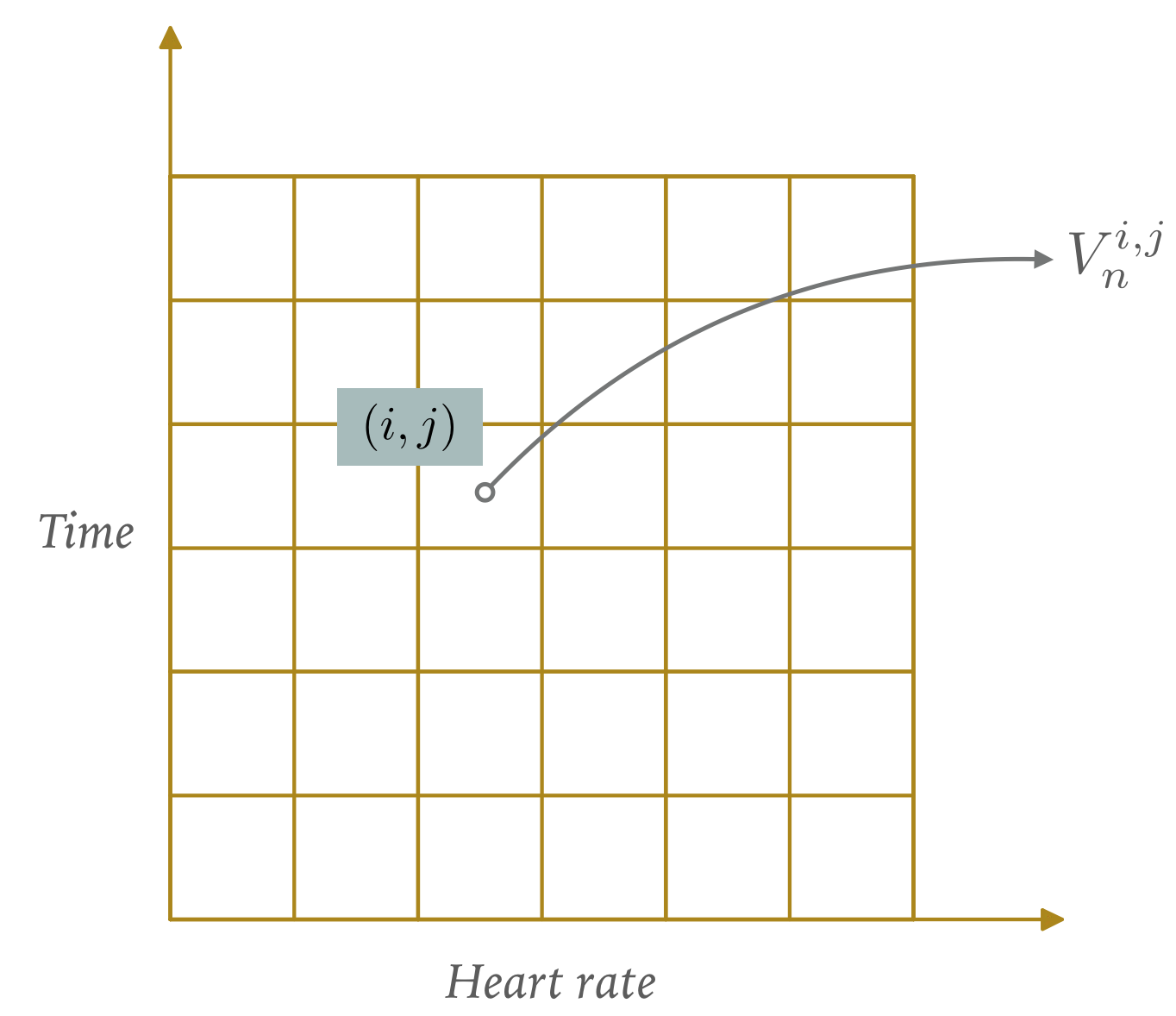}
\caption{Manifold splitting and reduced models $V_n^{(k,k')}$ on each partition.}
\label{fig:manifold_partitioning}
\end{figure}
 For each $\cM^{(k,k')}$, we can compute reduced models $V_n^{(k,k')}$. If we measure the reconstruction error in the worst case sense, we can estimate the reconstruction performance with this splitting by computing
\begin{equation}
e\(K,K'\) = \max_{(k,k') \in \{0,\dots, K-1\}\times \{0,\dots, K'-1\}} \min_{1\leq n \leq m} \max_{u\in \widetilde \cM_{(k,k')}} \frac{ \dist(u,V_n^{(k,k')})  }{\beta(V_n^{(k,k')}, W_m)}
\label{eq:size_criteria}
\end{equation}
 We then look for the optimal partition when $K$ and $K'$ range between $1$ and $7$, that is, we select
 $$
 (K_\opt, K_\opt') \in \argmin_{ (K, K') \in \{1,\dots,7\}\times \{1,\dots, 7\} } e\(K,K'\).
 $$
 
 When we consider only the velocity $u$ as a target quantity, we obtain a $5 \times 5$ partitioning of $Y^\obs$. 
 In Figure \ref{fig:size_opt_decay}, we show the behavior with $n$ of the stability constant $\beta(V_n^{(k,k')}, \Wm)$ and the error $\max_{u \in \widetilde \cM^{(k,k')}} \dist(u, V_n^{(k,k')})$ for each element of this optimal partition.
 
 We proceed similarly to derive the optimal partition for the couple velocity-pressure $(u,p)$ in $V=U\times P$. We also obtain a $5\times 5$ partition and Figure \ref{fig:size_opt_decay_up} shows the behavior of $\beta(V_n^{(k,k')}, \Wm)$ and the error $\max_{(u,p) \in \widetilde \cM^{(k,k')}} \dist((u, p), V_n^{(k,k')})$ for this case. Note that the value of the stability constant is very low, and this is due to the fact that our measurement space allows only to sense in the velocity.
 
 Once the optimal partition has been found, for each subset $\cM^{(k,k')}$, we select the optimal dimension $n^*$ as
\begin{equation}
n_{(k,k')}^* \in \argmin_{n=1,\ldots, m}{ \frac{ \dist(u,V_n^{(k,k')})  }{\beta(V_n^{(k,k')}, W_m)}}.
\label{eq:n_criteria}
\end{equation}
This procedure for the selection of $n_{(k,k')}^*$ \om{gives the best trade-off between accuracy and stability.} It is referred to as \textsl{the multi-space} approach in \cite{binev2017}.

\begin{figure}[!htbp]
\centering
\subfigure[Stability constant $\beta \left( V_n^{i,j}, W_m \right)$ for each manifold partition]{\includegraphics[height=5.5cm]{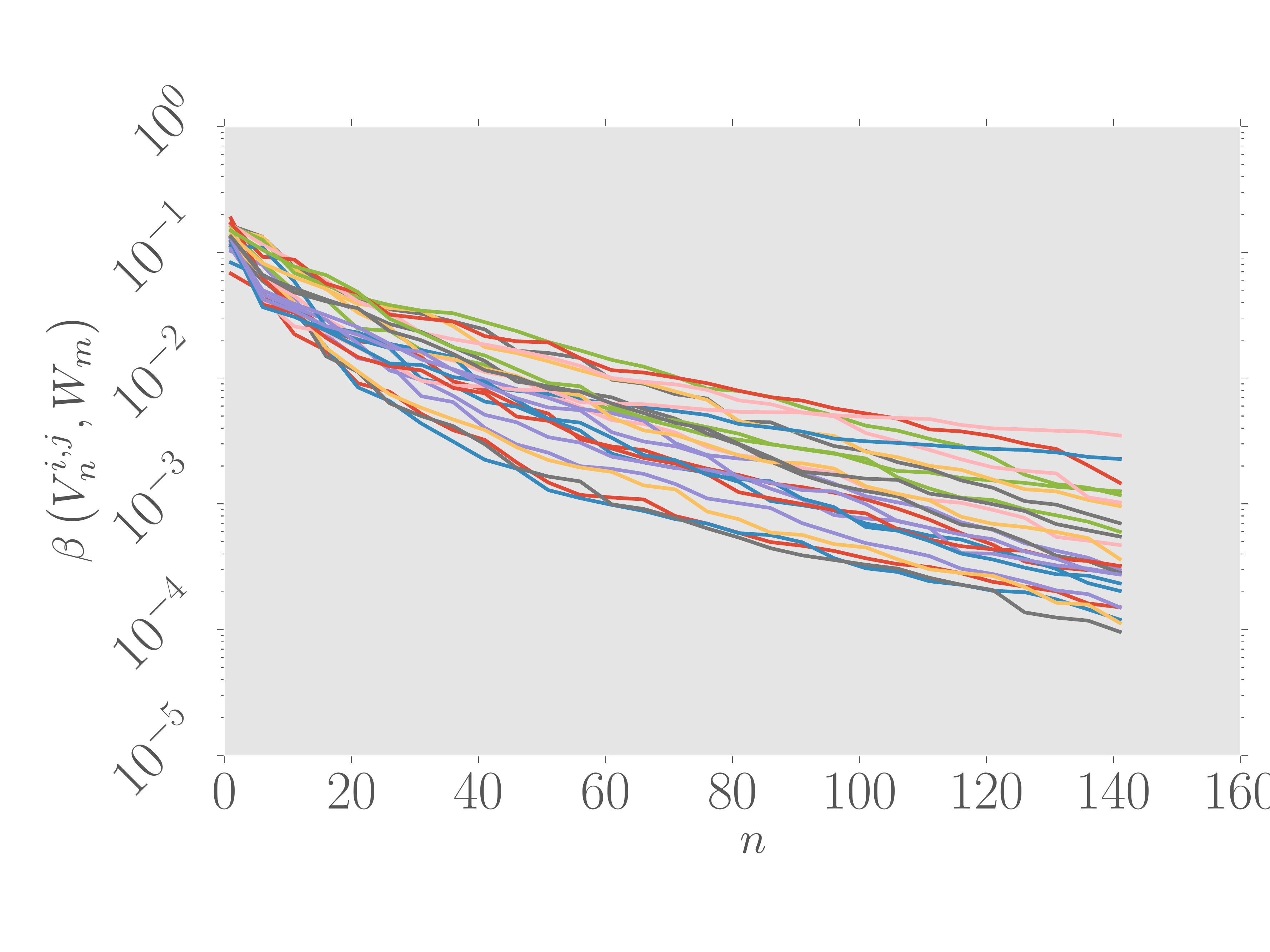}}
\subfigure[Model error $\dist \left( u, V_n^{i,j}   \right)$ for each manifold partition]{\includegraphics[height=5.5cm]{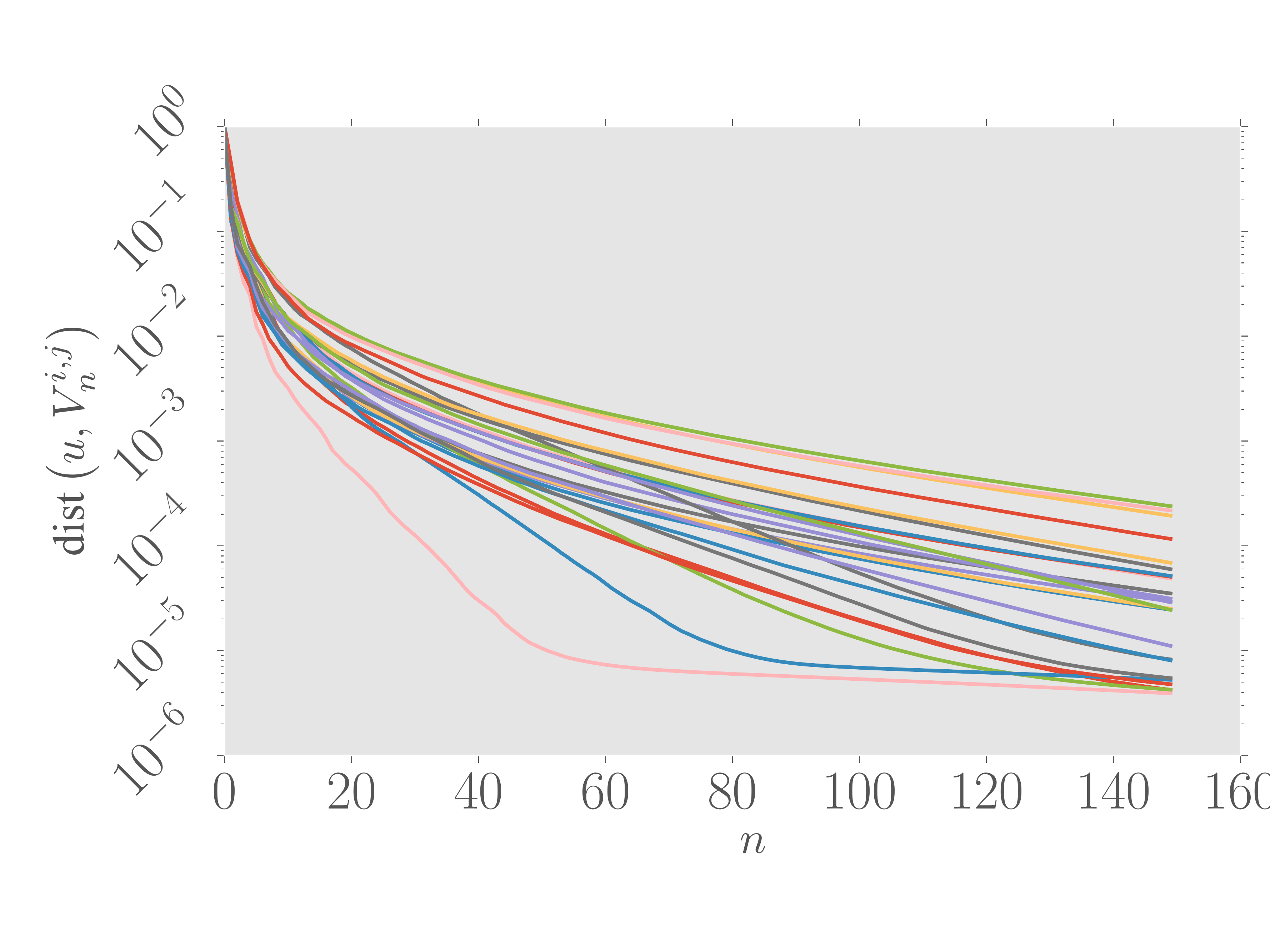}}
\caption{Behavior of stability constant and model error respect to the dimension of $V_n$ for reconstruction of velocity only. Optimal partitioning is ($K=K'=5$).}
\label{fig:size_opt_decay}
\end{figure}

\begin{figure}[!htbp]
\centering
\subfigure[Stability constant $\beta \left( V_n^{i,j}, W_m \right)$ for each manifold partition]{\includegraphics[height=5.5cm]{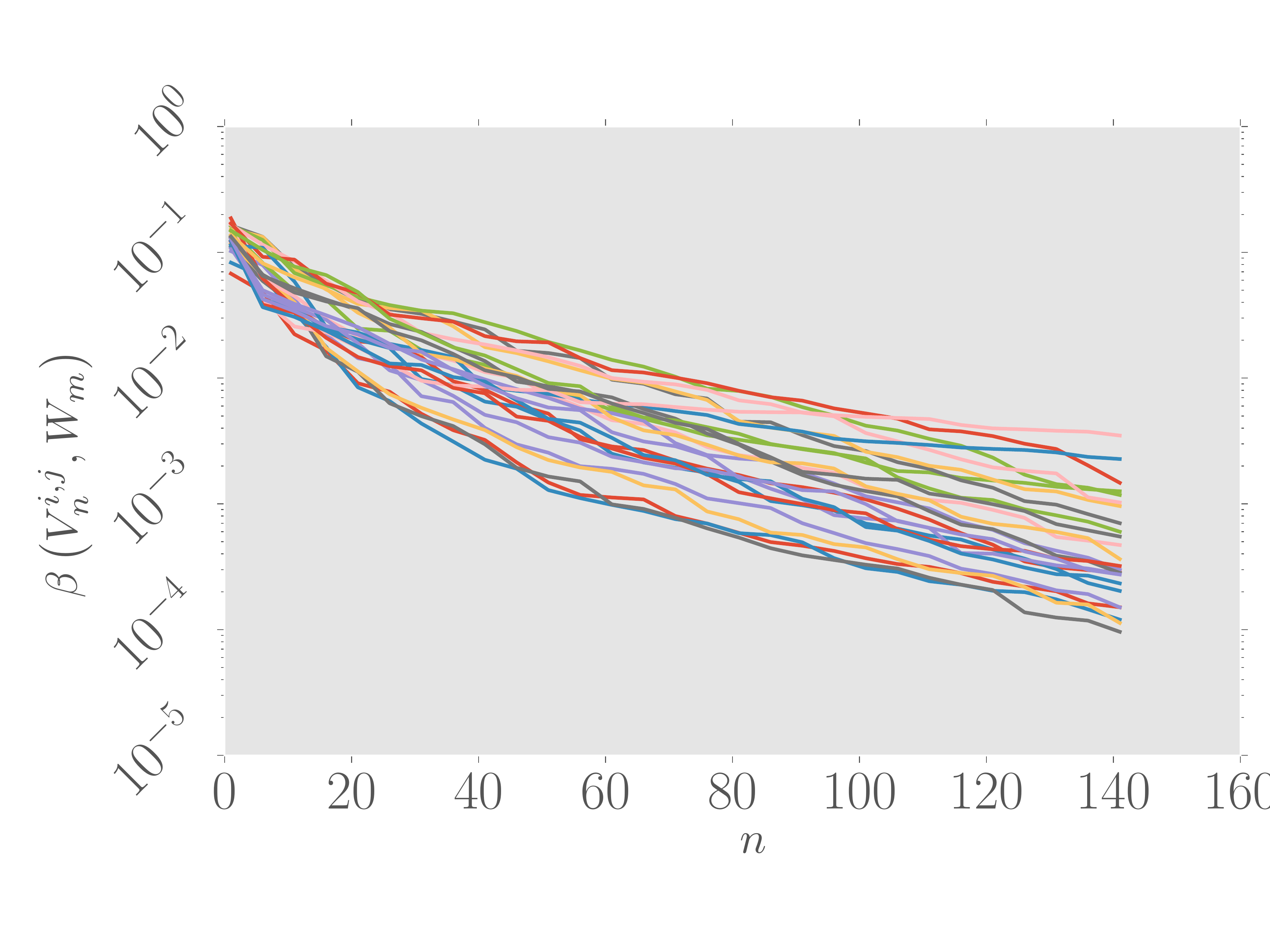}}
\subfigure[Model error $\dist \left( \(u,p\), V_n^{i,j} \right)$ for each manifold partition]{\includegraphics[height=5.5cm]{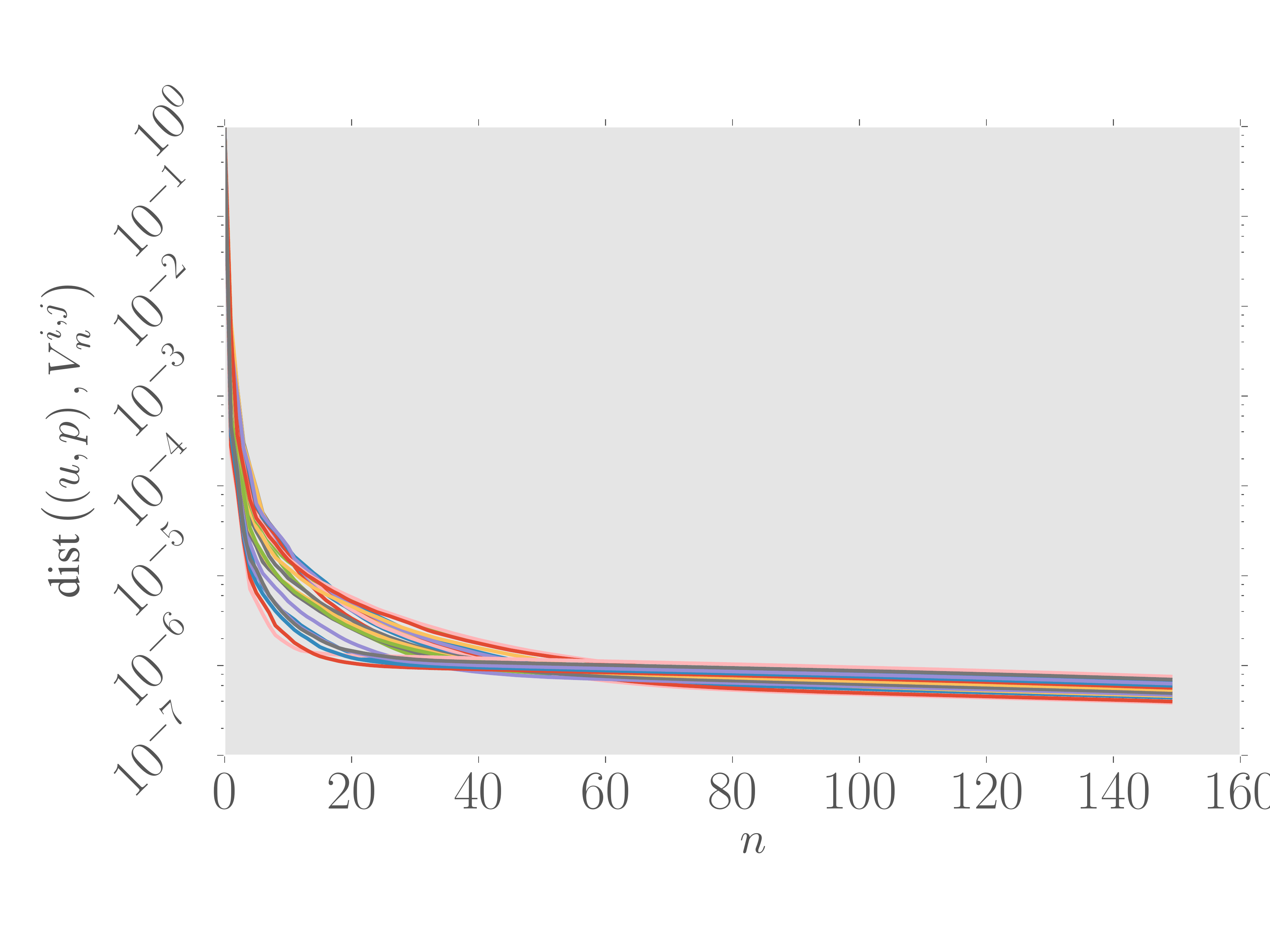}}
\caption{Behavior of stability constant and model error respect to the dimension of $V_n$ for joint reconstruction of section \ref{sec:joint_up}. Optimal partitioning is ($K=K'=5$).}
\label{fig:size_opt_decay_up}
\end{figure}

\subsection{Reconstruction results for velocity field and related quantities}
In the rest of the paper, we use the piecewise linear approach with the optimal splitting and the optimal dimension $n^*_{(k,k')}$ for the reduced models $V_n^{(k,k')}$. To simplify notation, we will write $V_n$ instead of $V_n^{(k,k')}$ when no confusion arises and $(u^*, p^*)$ instead of $(u^*_{n^*_{(k,k')}}, p^*_{n^*_{(k,k')}})$.  In addition, depending on the context, $V_n$ denotes either the reduced model for the velocity reconstruction or the reduced model for the joint reconstruction described in section \ref{sec:joint_up}.

Figure \ref{fig:error_u} shows the relative error in time in the velocity reconstruction
\begin{equation}
e(u(t))^2 = \frac{\Vert u(t) - u^*(t)\Vert_U^2}{\int \Vert u(t)\Vert_U^2 dt}.
\label{eq:err_u_time}
\end{equation}
in norm $U = [H^1(\Omega)]^3$. We observe that there is no field over 10\% error. One can further examine the error by studying separately the $L^2(\Omega)$ reconstruction error  of the velocity and its gradient, as shown in Figure \ref{fig:error_L2_grad}. The reconstruction plots show that there are small error peaks around the region where one time window ends and the next one begins. Strategies of \textsl{window overlapping} will be explored in future works in order to mitigate this behavior. \om{For the current cardiovascular application on carotids, we refer to \cite{GYK2009} for a previous work on forward reduced modelling involving window overlapping for the purpose of analysing turbulence in stenoses carotid arteries using POD. Another approach which makes the reduced space depend on time was explored in \cite{galarce2020}. It consists in building a data-driven reduced model based on the observation $\omega$. Since $\omega$ evolves in time, the resulting reduced space is automatically updated.}

\begin{figure}[!htbp]
  \centering
  \subfigure[$e\(u\(t\)\)$ for each \dl{time dependent solution in one cardiac cycle}.]{
    \includegraphics[height=5cm]{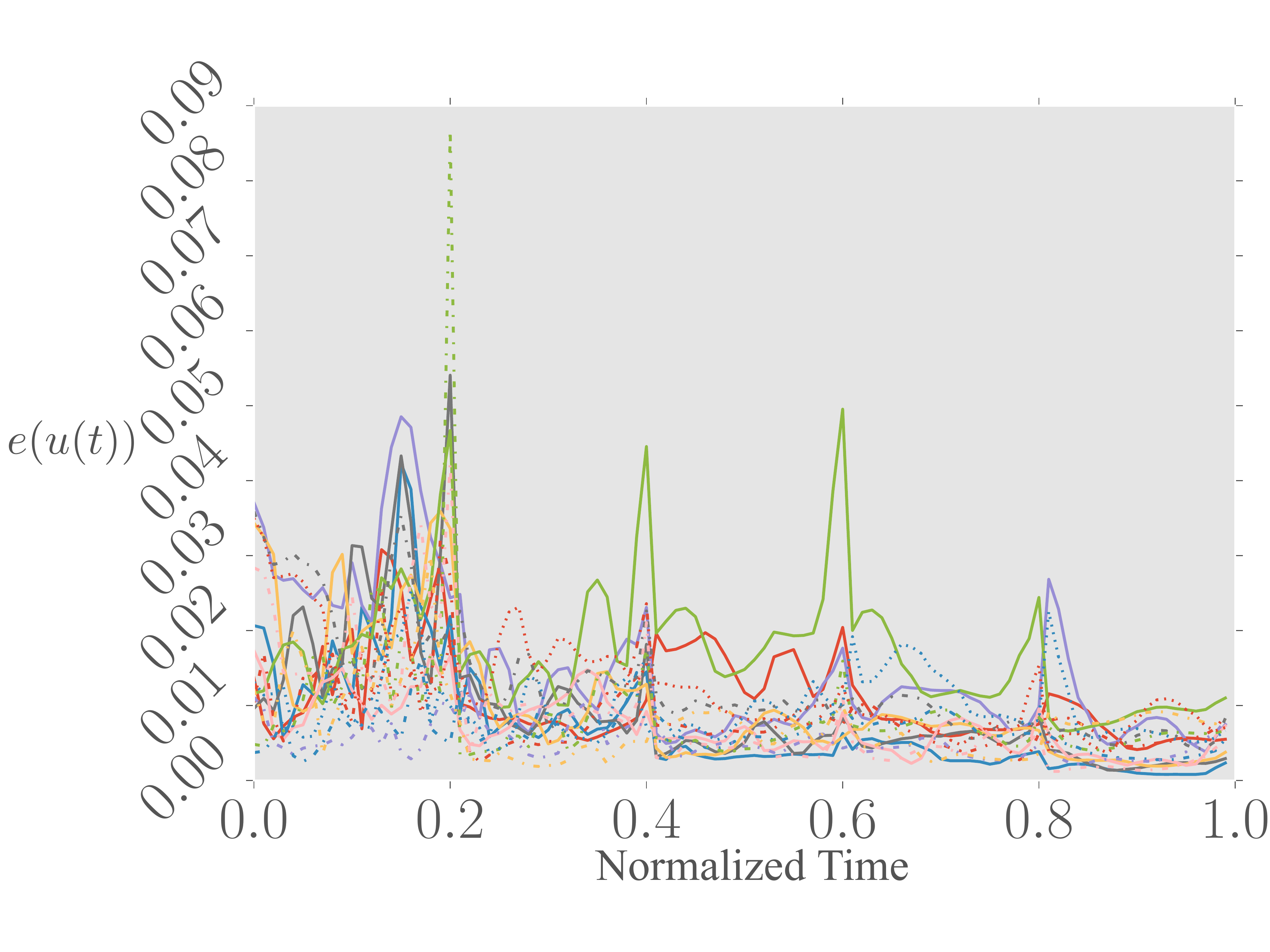}
  }
  \subfigure[$e\(u\(t\)\)$ averaged and max.]{
    \includegraphics[height=5cm]{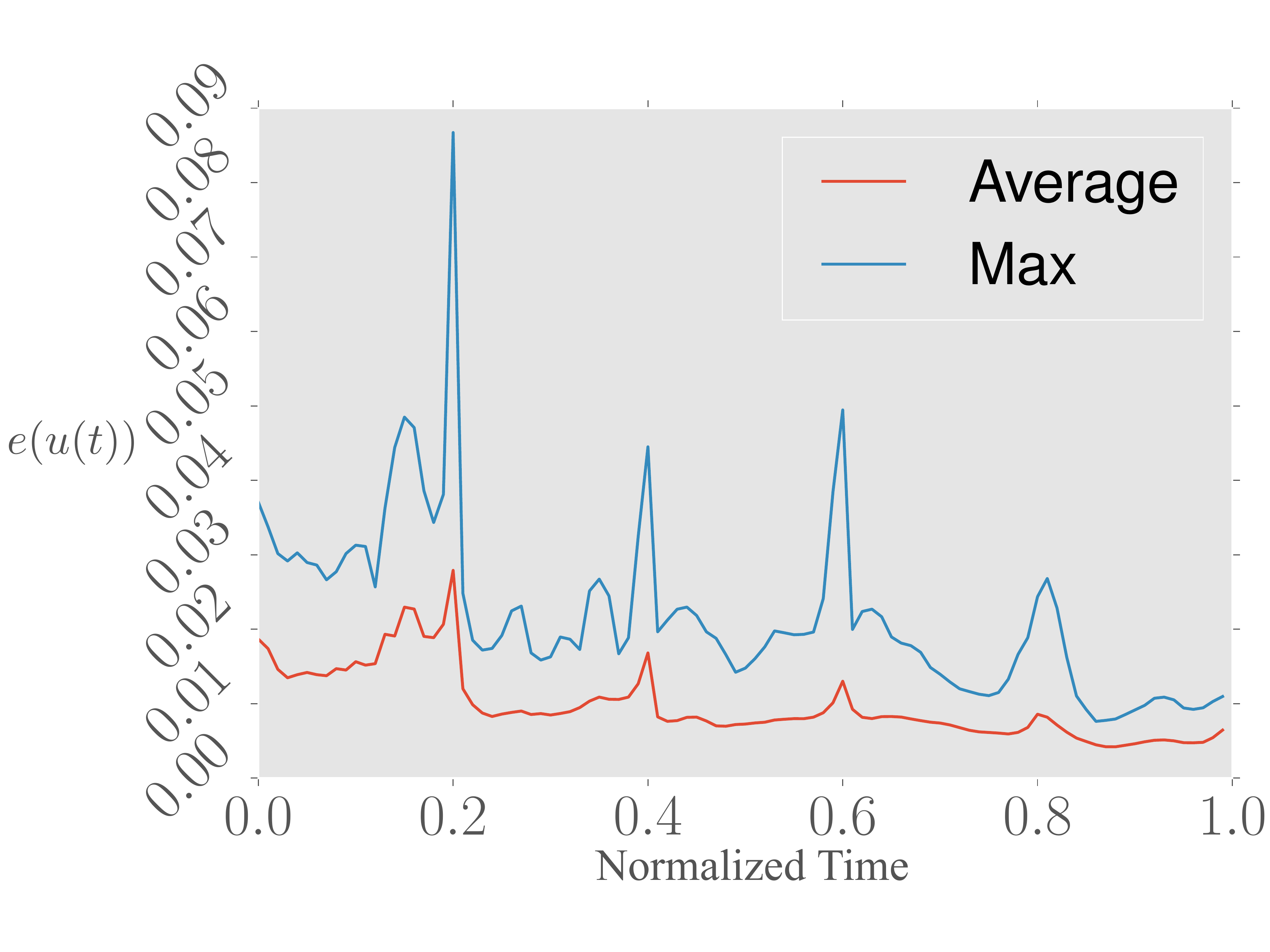}
  }

  \caption{Reconstruction error in $U = [H^1(\Omega)]^3$ of the velocity field for 16 \dl{different solutions}. Notice the small jumps at each time window interface. The vertical axis shows the error as expressed in \eqref{eq:err_u_time}. The horizontal axis shows the normalized time for one cardiac cycle.}
  \label{fig:error_u}
\end{figure}

\begin{figure}[!htbp]
\centering
\subfigure[$e((u(t))$ in norm $L^2$.]{\includegraphics[height=5cm]{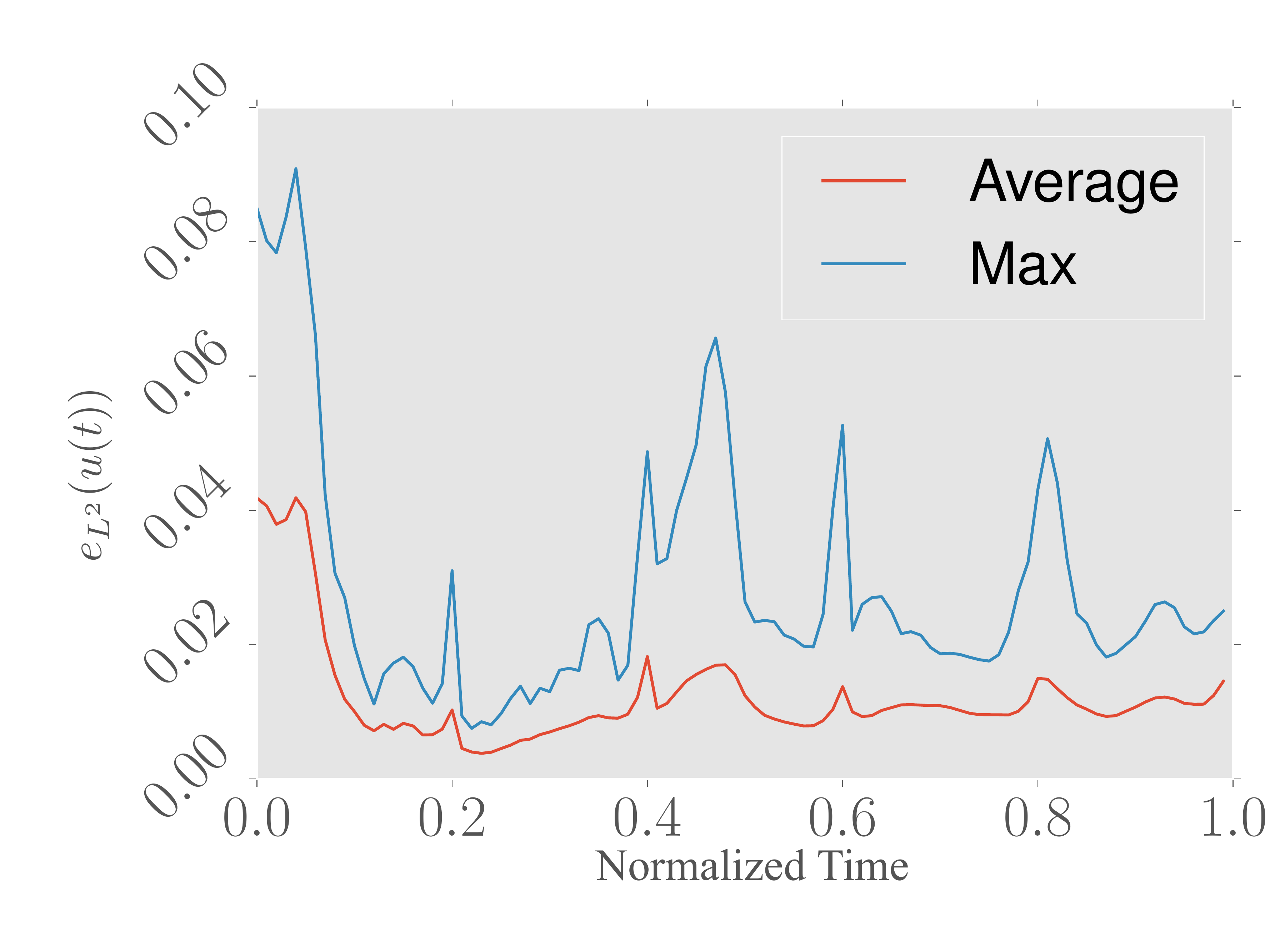}}
\subfigure[$e((\nabla u (t)))$ in norm $L^2$.]{\includegraphics[height=5cm]{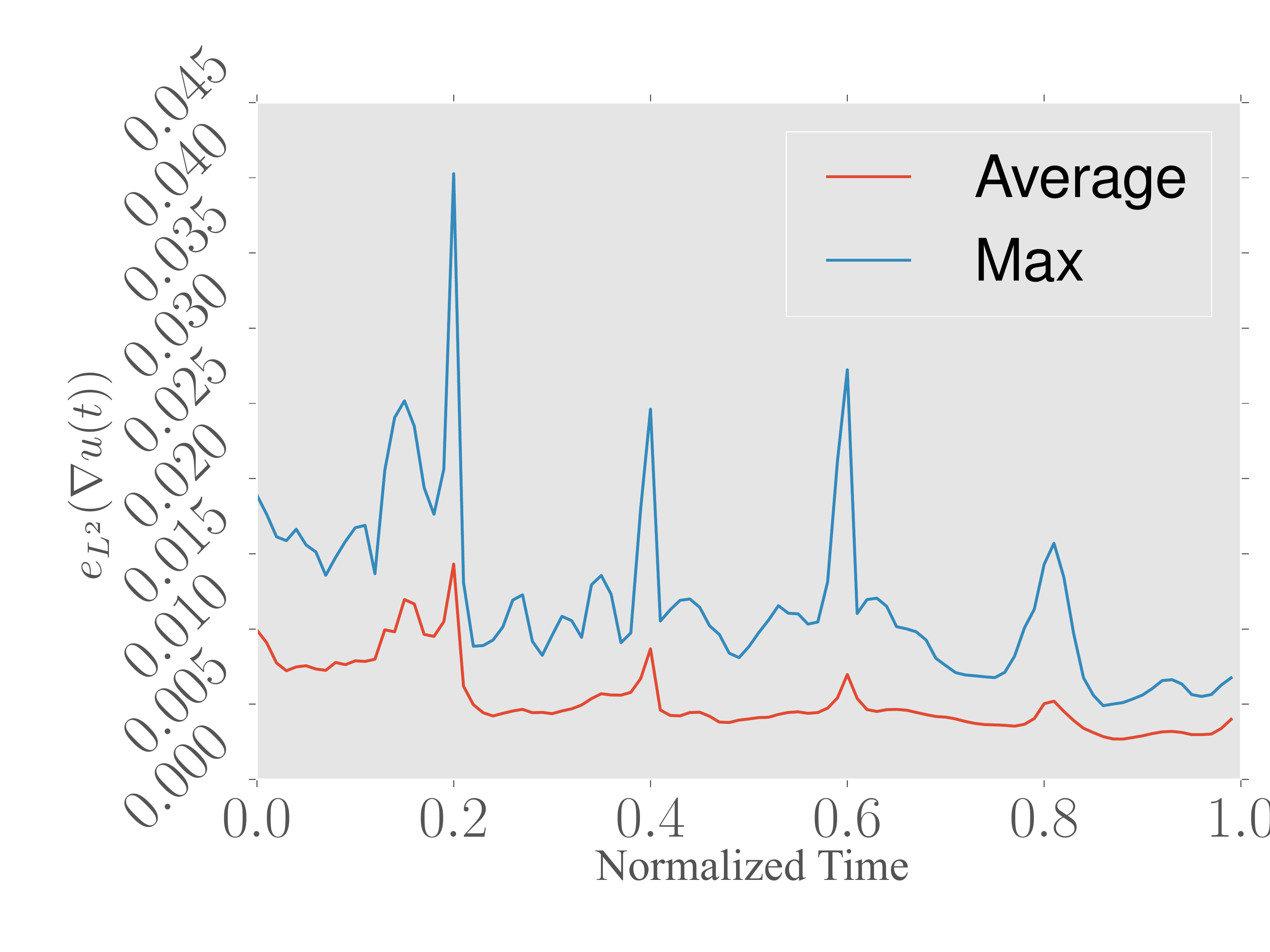}}
\caption{$L^2$ error in velocity reconstruction. We observe that the accuracy for both $u^*$ and $\nabla u^*$ stays in the same orders of magnitude. The vertical axis shows the error as expressed in \eqref{eq:err_u_time}. The horizontal axis shows the normalized time for one cardiac cycle.}
\label{fig:error_L2_grad}
\end{figure}

Figure \ref{fig:u_model_error} shows the approximation error due to the projection in $V_n$,
$$
e_{V_n}(u(t))^2 = \frac{\norm{ u(t) - P_{V_n} u(t)}^2}{\int \norm{u(t)}^2 dt}.
$$
We see that this error does not interfere with the reconstruction quality in the sense that it stays much lower than \eqref{eq:err_u_time}.

An example of the velocity reconstruction during the early systole period is shown in Figure \ref{fig:u}, where we observe that the larger errors occur in the stenosis zone. This behavior is observed during the whole cardiac cycle.

Concerning the WSS and vorticity, we can see the time evolution of the errors \eqref{eq:error_wss} and \eqref{eq:err_vorticity_time} in Figure \ref{fig:vort_wss_rec_error}. As an illustration, Figure \ref{fig:vort_rec_example} and Figure \ref{fig:wss_rec_example} shows a three-dimensional reconstructed vorticity and wall shear stress fields, respectively. The error in the stenosis area tends to be \textsl{propagated} from the velocity reconstruction, as can be expected.

\begin{figure}[!htbp]
\centering
  \includegraphics[height=5.5cm]{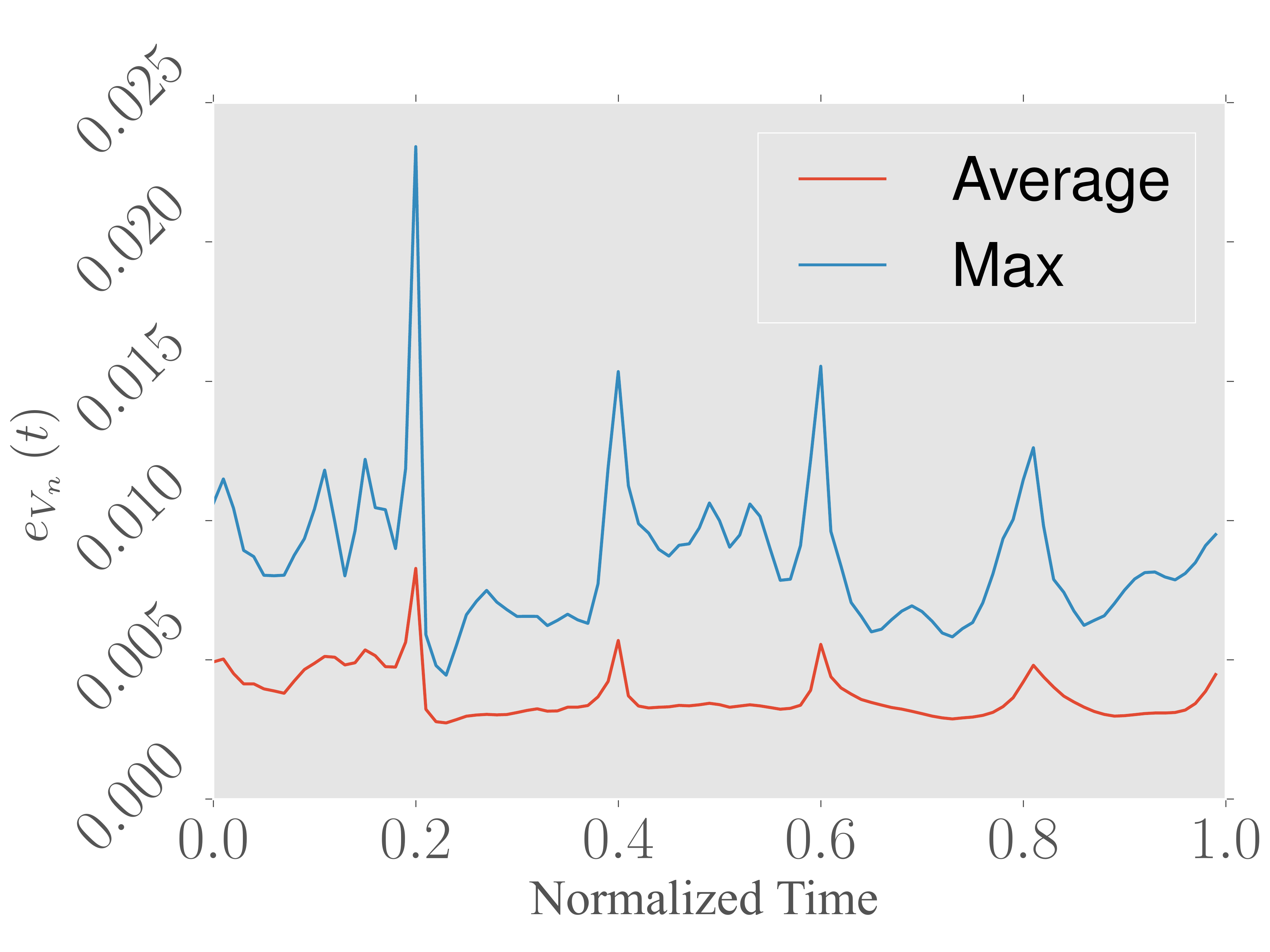}
  \caption{Average and worst case scenario model error $e_{V_n}(u(t))$.}
  \label{fig:u_model_error}
\end{figure}

\begin{figure}[!htbp]
\centering
\subfigure[$u$.]{
\includegraphics[height=8cm]{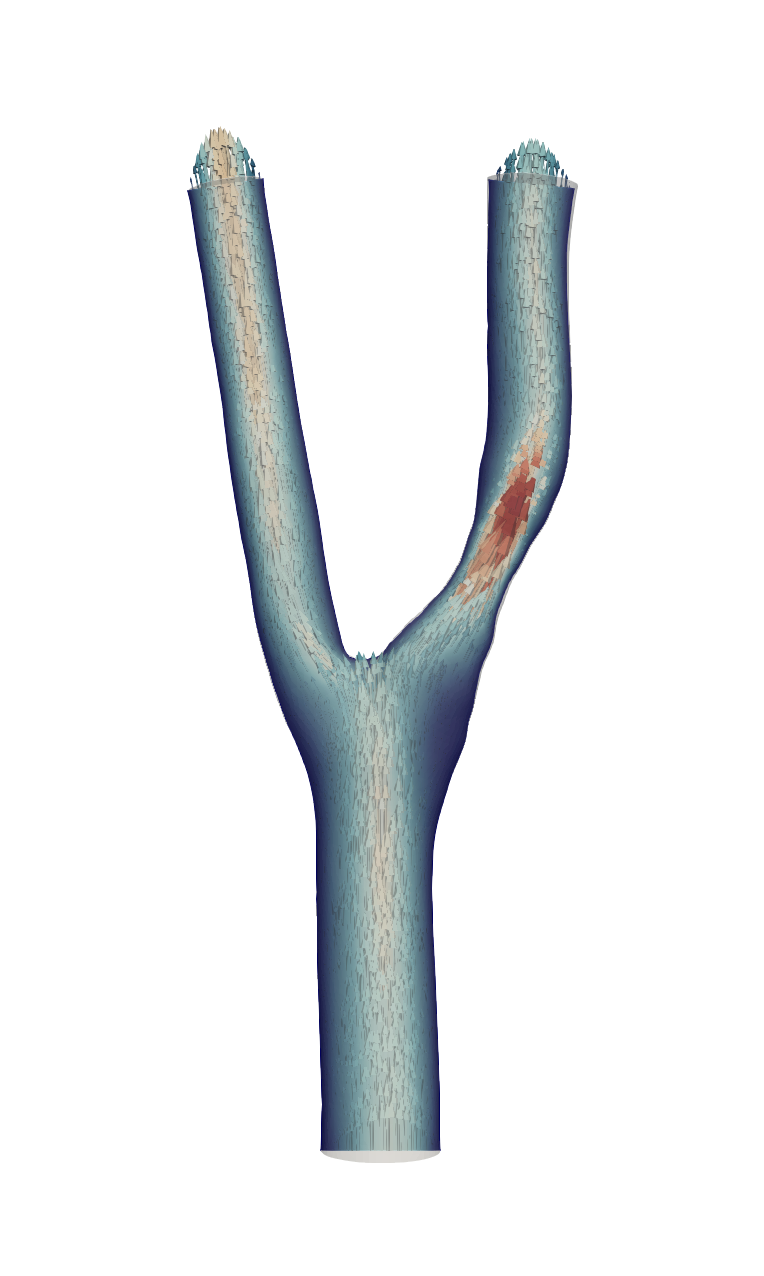}}
\subfigure[$u^*$.]{
\includegraphics[height=8cm]{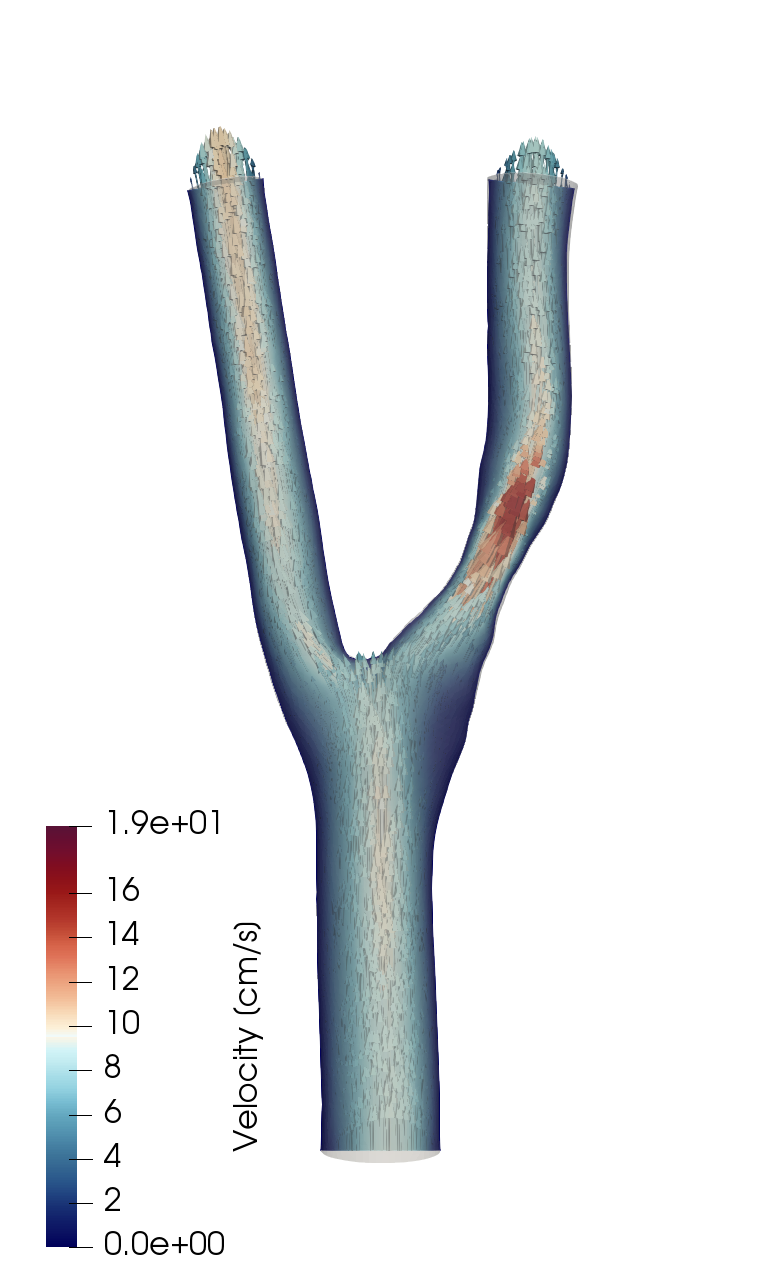}}
\subfigure[$u - u^*.$]{
\includegraphics[height=8cm]{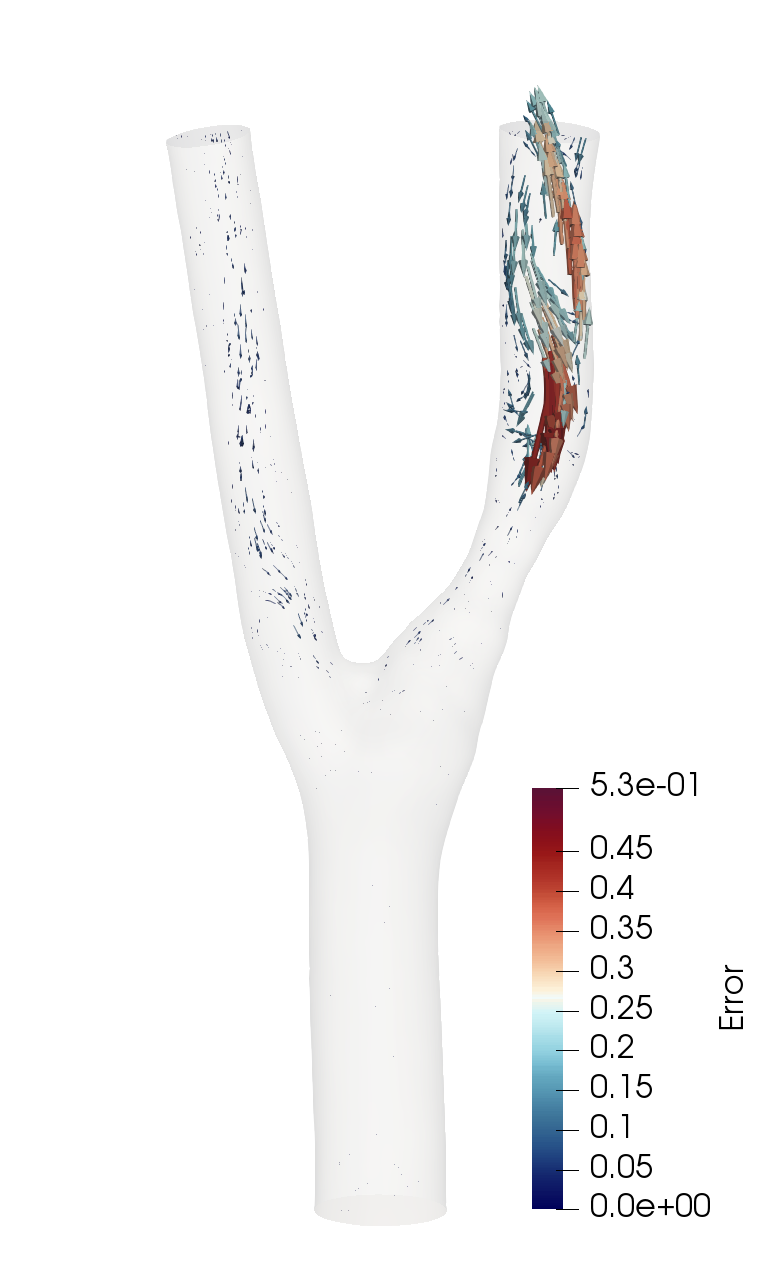}}
\caption{Example of reconstruction of the velocity during the early systole period. We observe a zone of high error close to the stenosis.}
\label{fig:u}
\end{figure}


\begin{figure}[!htbp]
\centering
\subfigure[WSS error.]{\includegraphics[height=5cm]{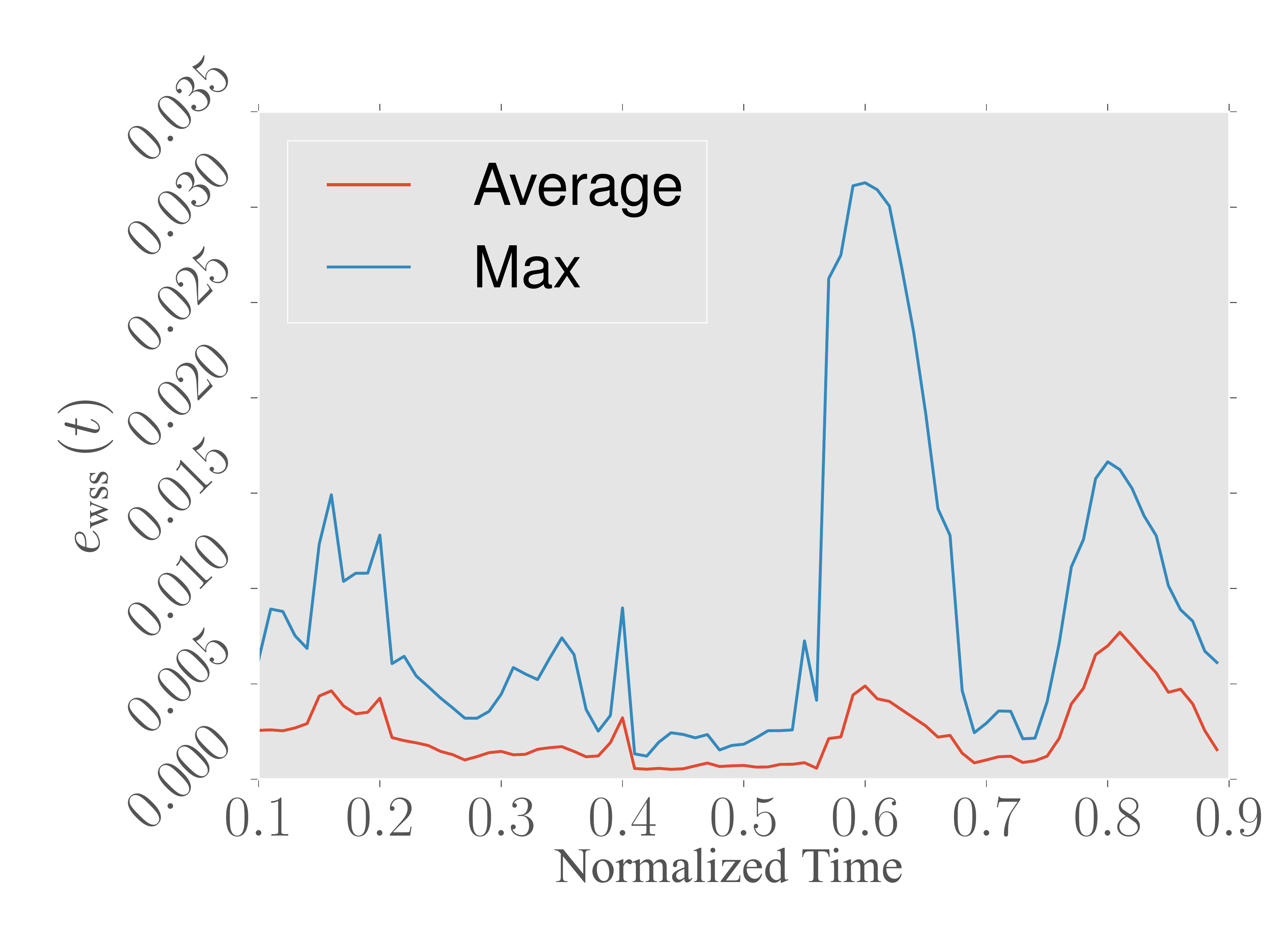}}
\subfigure[Vorticity error.]{\includegraphics[height=5cm]{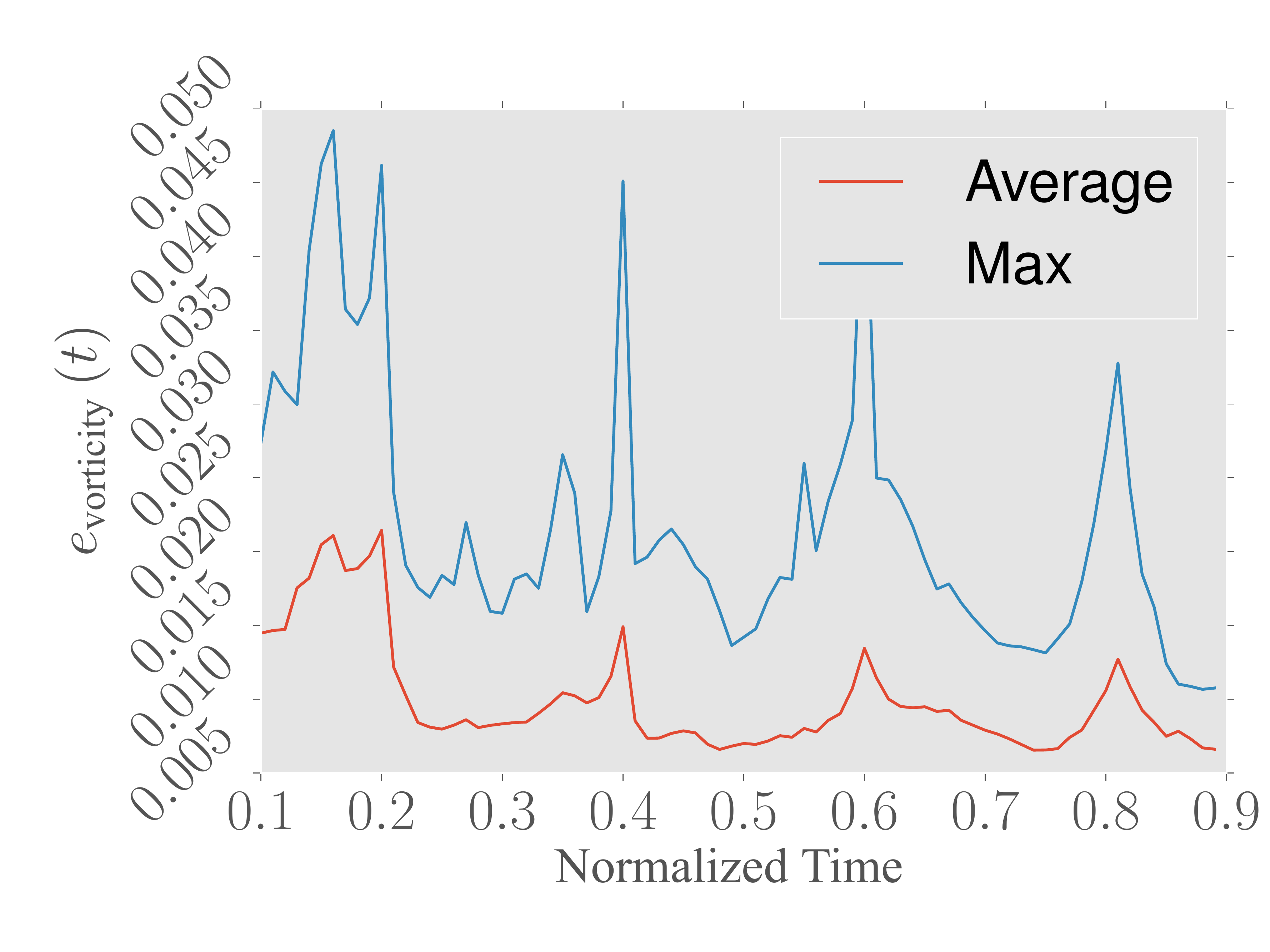}}
\caption{Relative errors \eqref{eq:err_vorticity_time} and \eqref{eq:error_wss} in time (horizontal axis in seconds).}
\label{fig:vort_wss_rec_error}
\end{figure}

\begin{figure}[!htbp]
\centering
\subfigure[$\vorticity$.]{
\includegraphics[height=8cm]{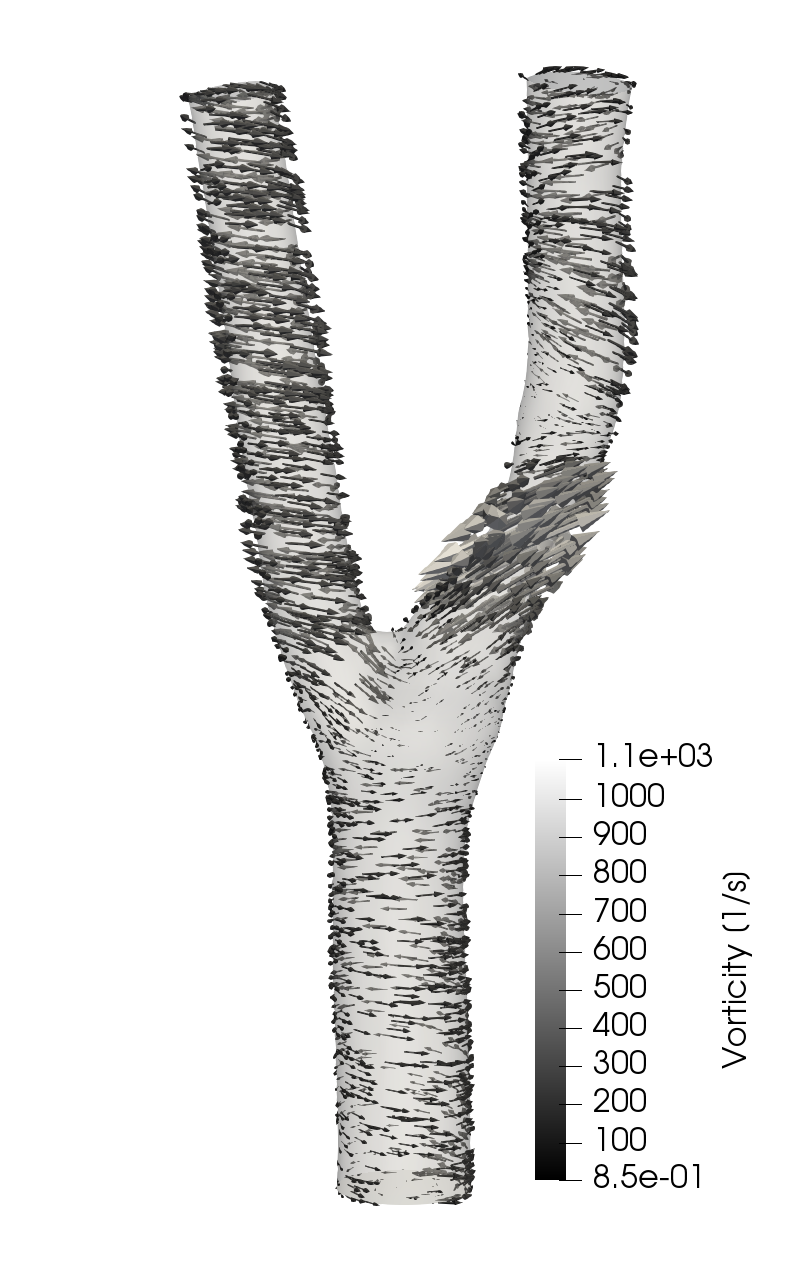}}
\subfigure[$\vorticity^*$.]{
\includegraphics[height=8cm]{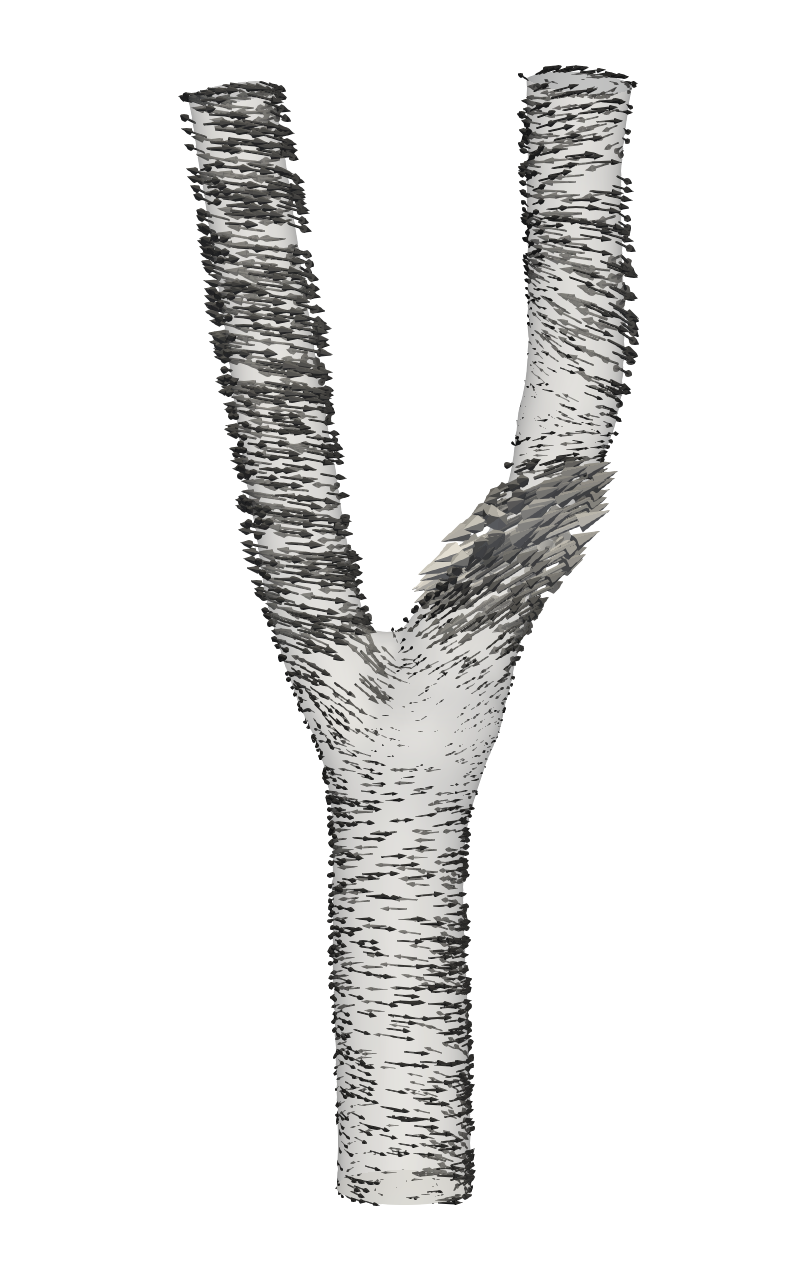}}
\subfigure[$\vorticity - \vorticity^*$.]{
\includegraphics[height=8cm]{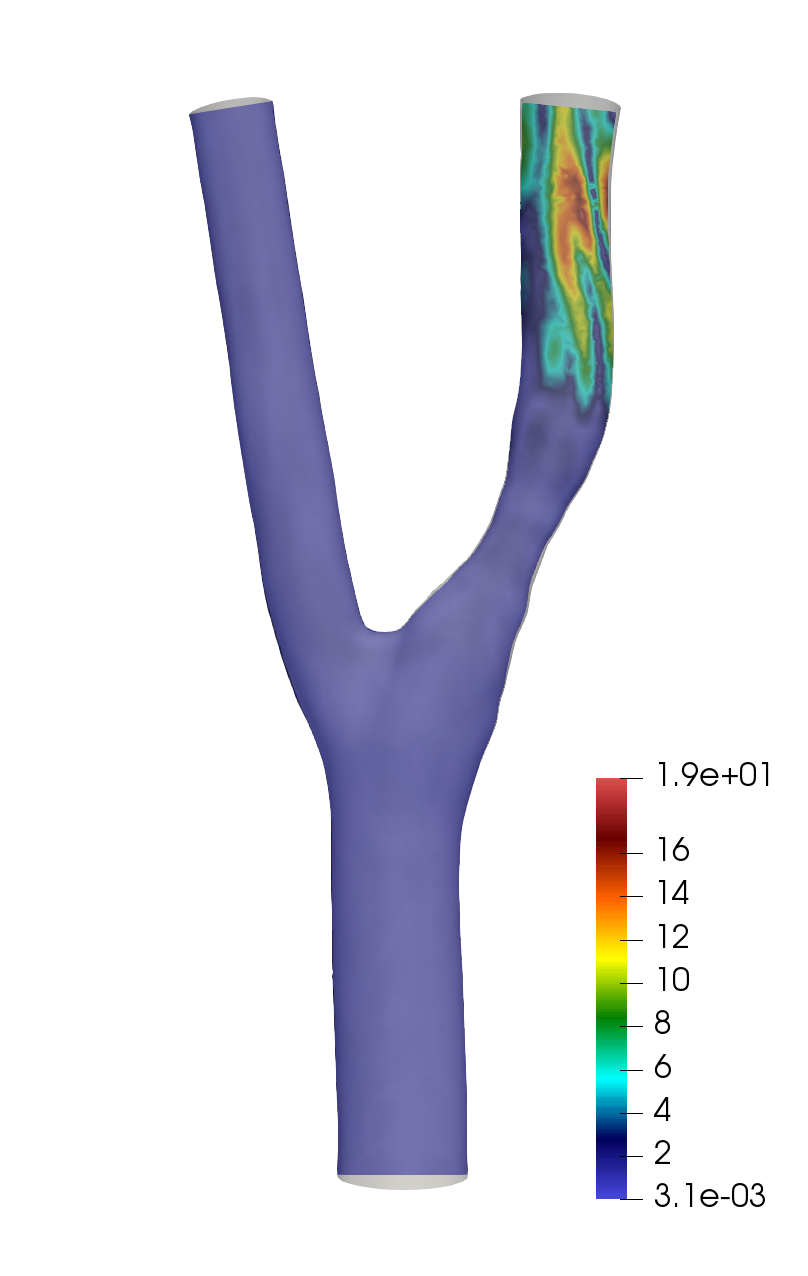}}
\caption{Example of reconstruction of the vorticity during the early systole period. We observe how the error from the velocity reconstruction is reproduced close to the stenosis.}
\label{fig:vort_rec_example}
\end{figure}

\begin{figure}[!htbp]
\centering
\subfigure[$\wss$.]{
\includegraphics[height=8cm]{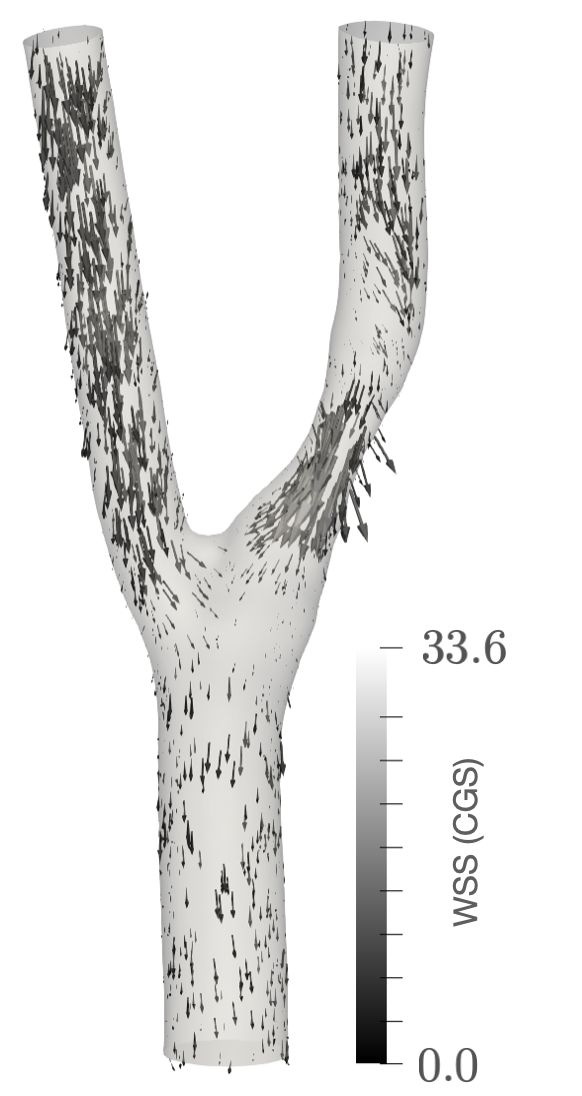}}
\subfigure[$\wss^*$.]{
\includegraphics[height=8cm]{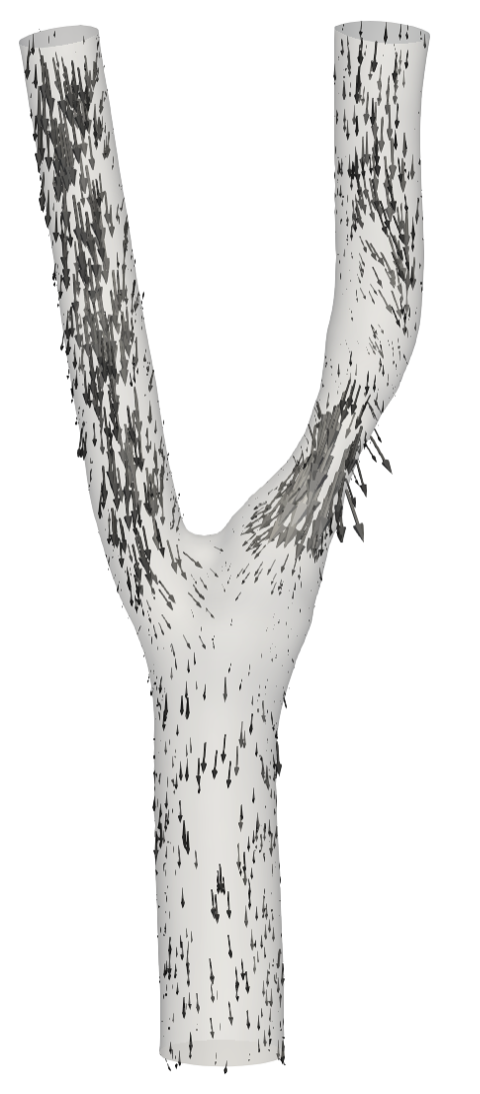}}
\subfigure[$\wss - \wss^*$.]{
\includegraphics[height=8cm]{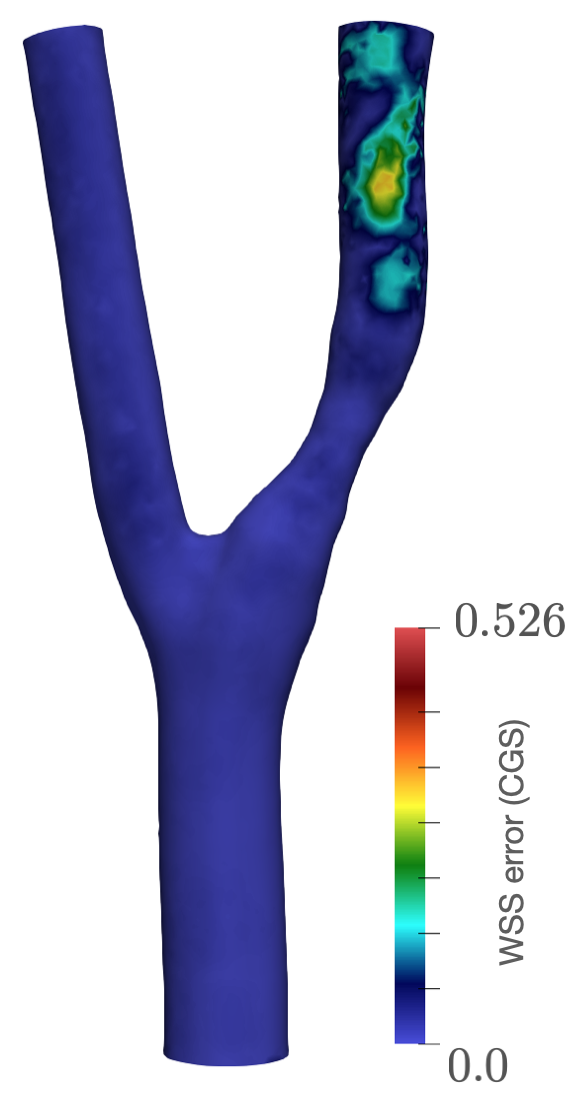}}
\caption{Example of reconstruction of the wall shear stress during the early systole period. We observe how the error from the velocity reconstruction is reproduced close to the stenosis.}
\label{fig:wss_rec_example}
\end{figure}

\subsubsection{\om{Estimation of computational costs}}
\label{ref:comp_costs}
Let us briefly discuss the computational cost of the method in terms of computational complexity. Since the forward simulations are produced using a finite element method over a mesh of $N$ vertices with $\bP_1$ elements for both velocity and pressure, the construction of a discrete manifold $\widetilde \cM$ scales like $\cO\( (4N)^2 \#\widetilde{\cM} \)$ using an iterative solver such as GMRES. If we partition $\widetilde \cM$ into $K$ subsets as explained in section \ref{sec:piecewise}, we need to compute a singular value decomposition for each subset $\widetilde \cM_k$, $k=1,\ldots,K$. Each SVD computation follows $\cO(3N (\# \widetilde \cM_k)^2 )$ for the velocity basis and $\cO(4N (\# \widetilde \cM_k)^2 )$ for the joint basis of velocity and pressure. 

To speed-up the online reconstruction of the least-squares term $v_{m,n}^*$, we precompute a QR decomposition in the offline stage at a cost $\cO \( m n^2 \)$. Once this is done, in the online phase we only need to solve a system of $n \times n$ equations to find the coefficients in the low dimensional space and then combine them in dimension $3N$ if we are doing velocity reconstruction, or $4N$ if we are doing a joint reconstruction. Since we have stored the QR decomposition, we only have a simple backward substitution for the online phase and the basis combination. Accordingly, without the pre-computation step, the velocity reconstruction costs $\cO \( n^2 + 3 n N \)$ and the joint reconstruction follows $\cO \( n^2 + 4 n N \)$.

In addition we need to compute the $m$ Riesz representers $\omega_i$. Since $W$ does not change in our application, this is done offline and we require $\cO \( (3 N)^2 m \)$ or $\cO \( (4 N)^2 m \)$ operations for the velocity reconstruction and the joint reconstruction respectively.

\subsection{Pressure drop estimation results}

This section is devoted to comparing the two reconstruction methods for the pressure drop introduced in section \ref{sec:pdrop}. 

In the first method, we first compute the joint reconstruction of velocity-pressure with the piecewise linear algorithm, and then compute the pressure drop with formula \eqref{eq:pdrop}. Figure \ref{fig:p_drop_pbdw_mixed} shows the evolution of the estimated pressure drop in 4 simulations and compares it with the evolution of the exact pressure drop. The figure shows that the methods delivers a very high accuracy.

\begin{figure}[!htbp]
\centering
\includegraphics[height=4.5cm]{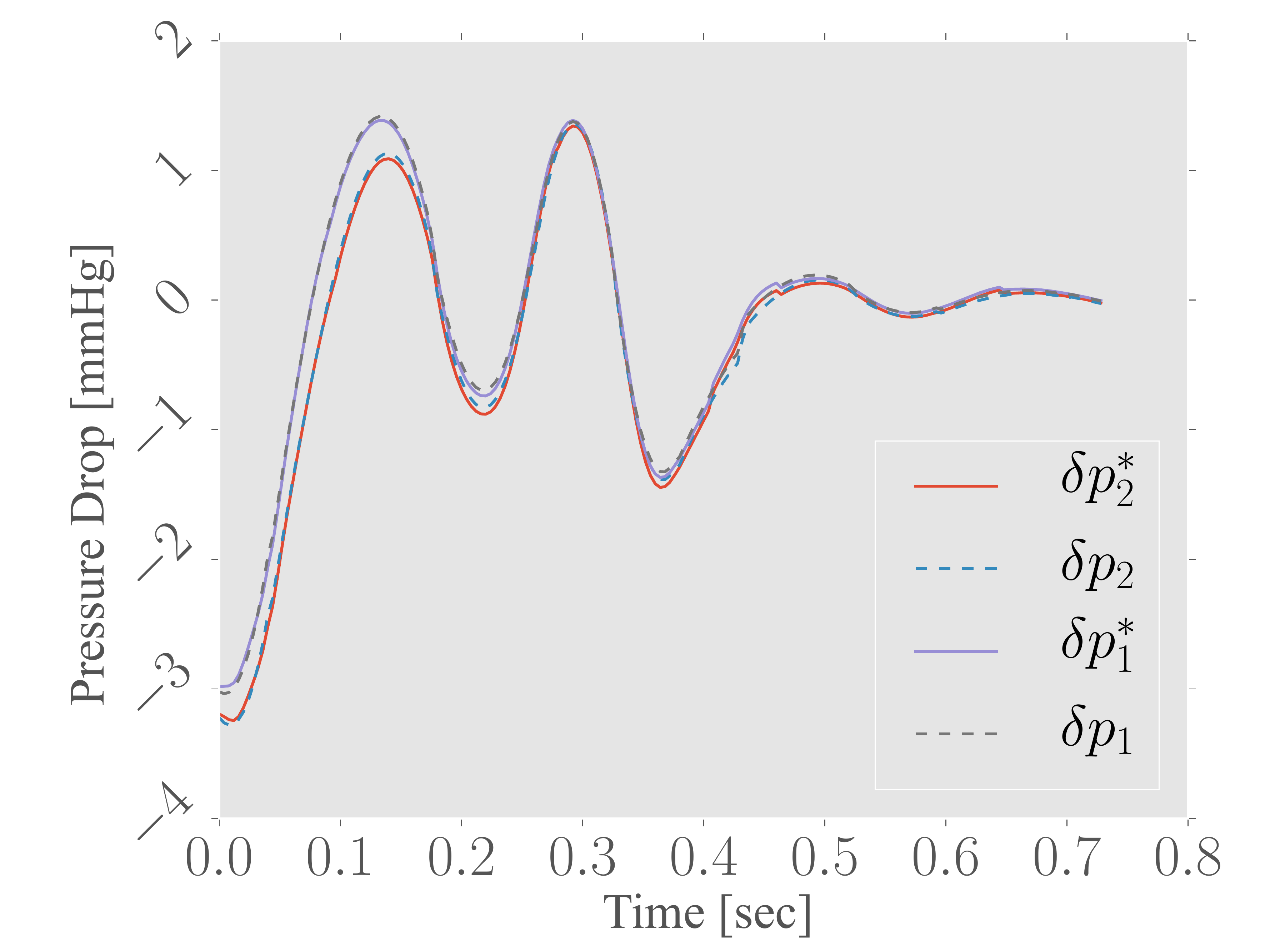}
\includegraphics[height=4.5cm]{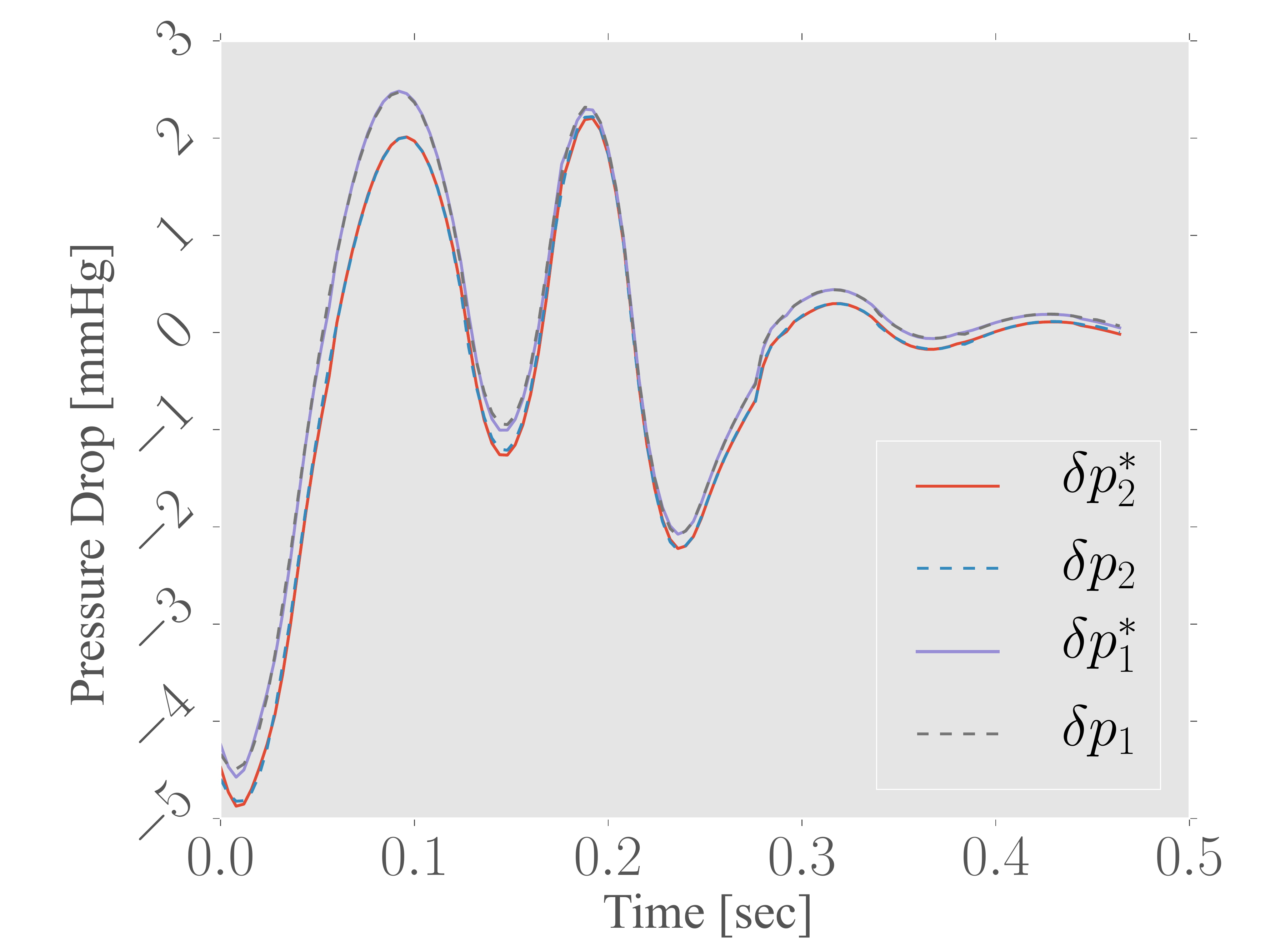}
\includegraphics[height=4.5cm]{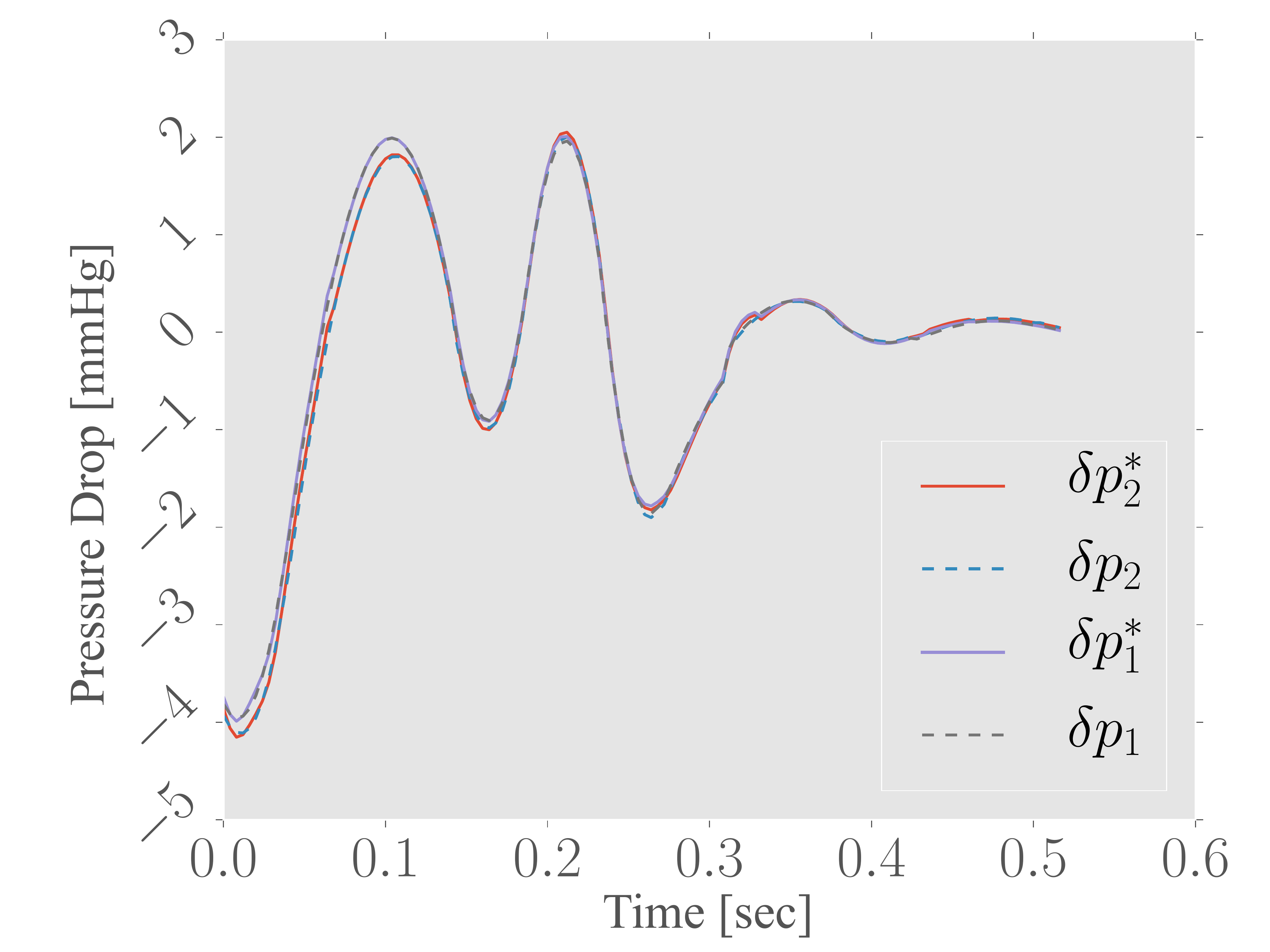}
\includegraphics[height=4.5cm]{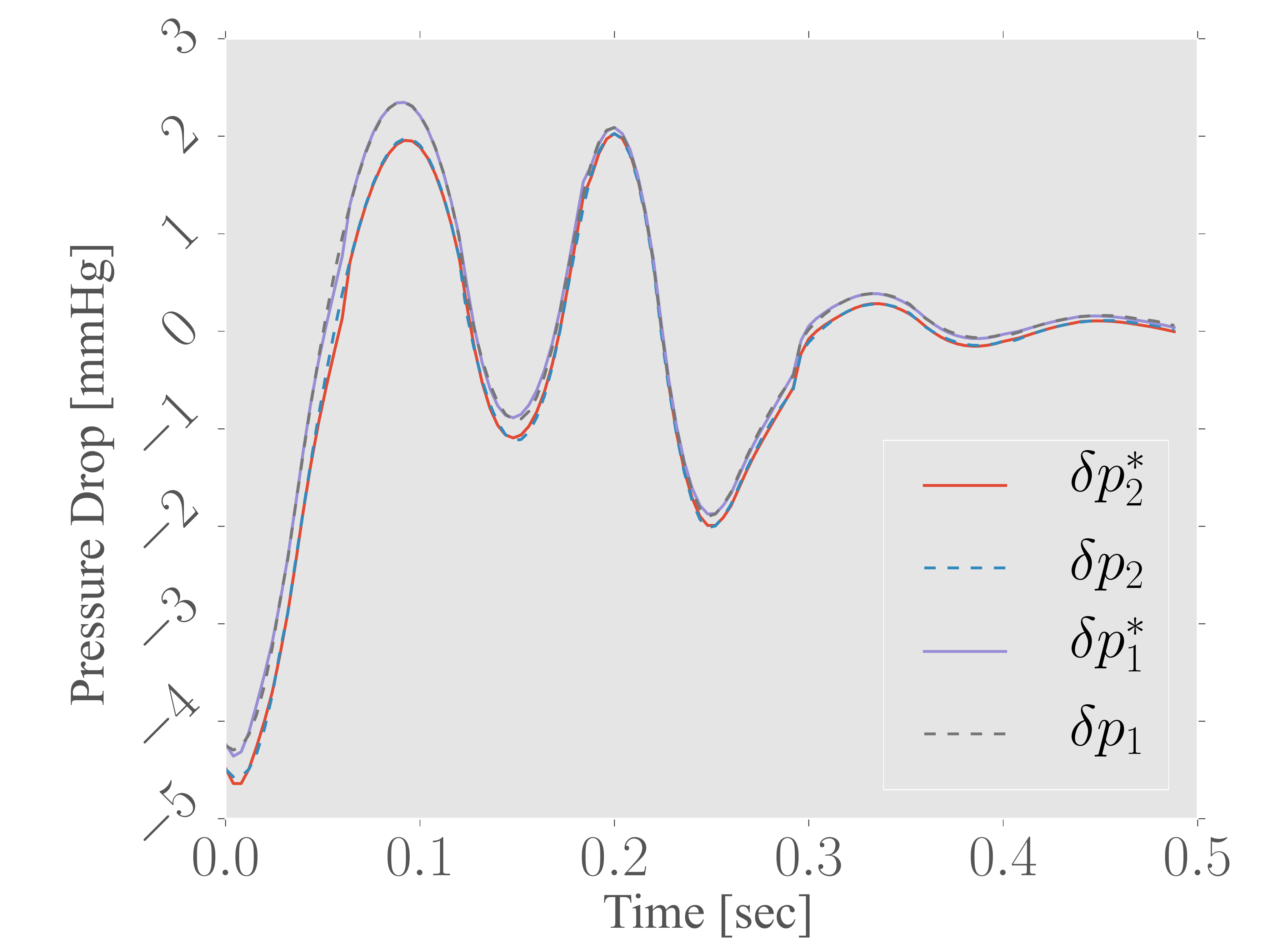}
\caption{Pressure drops $\delta p_1^*$ and $\delta p_2^*$ in four simulations using the joint reconstruction method of section \ref{sec:joint_up} with the piecewise linear algorithm. Dashed lines shows the ground truth $\delta p_1$ and $\delta p_2$. The vertical axis shows the pressure drop in [mmHg] and the horizontal axis the time in seconds.}
\label{fig:p_drop_pbdw_mixed}
\end{figure}

We can justify the obtained high accuracy by estimating the value of the stability parameter $\kappa_{m,n}$ defined in \eqref{eq:kappa_bound_pdrop}. For this, we approximate the space $W^\perp$ with
$$
\widetilde W^\perp = \vspan \{ \Psi_1, \ldots, \Psi_N \} \subset W^\perp,
$$
where $\{ \Psi_1,\ldots,\Psi_N \}$ is an orthonormal set of functions of $W^\perp$. These functions are obtained, for example, by first computing a singular value decomposition with  $N \gg n$ functions $\phi_i=(u_i,p_i)$ from the manifold $\cM$. In our case we set $N = 250$. We can then orthonormalize them with respect to $W_m = \vspan \{ \omega_1, \ldots, \omega_m \}$, which yields the desired
\begin{equation}
\Psi_i = \phi_i - P_{\Wm} \phi_i.
\end{equation}
We can next expand any function $\eta \in \widetilde W^\perp$ as
\begin{equation}
\eta = \sum_{i=1}^N \eta_i \Psi_i,\qquad \text{with }\eta_i = \left< \eta, \Psi_i \right>
\end{equation}

\begin{figure}[!htbp]
\centering
\subfigure[$\kappa_{m,n}$ as a function of $n$.]{
\includegraphics[height=6cm]{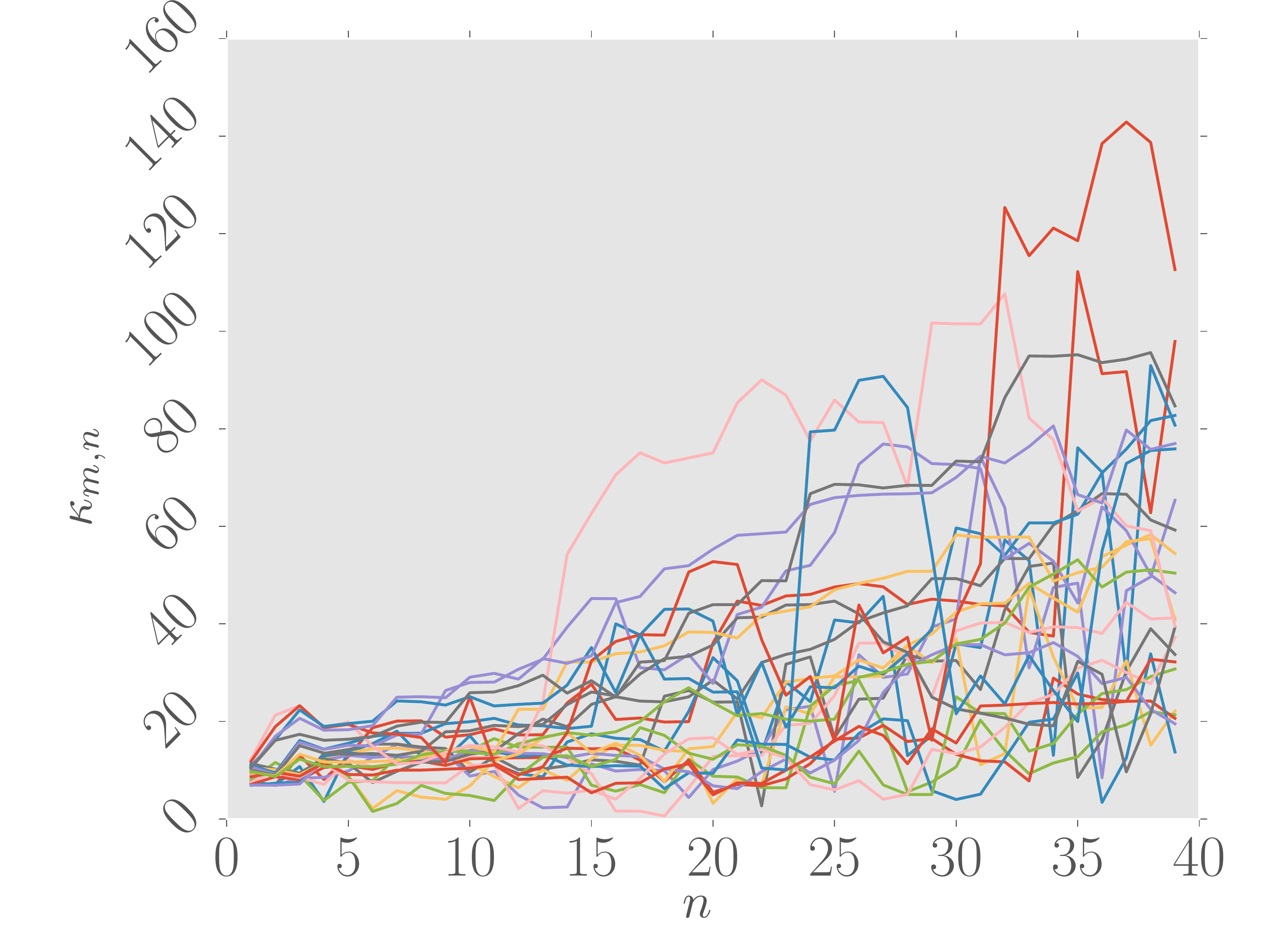}
\label{fig:mu_pbdw_joint_a}
}
\subfigure[Upper bound of pressure drop reconstruction error as a function of $n$.]{
\includegraphics[height=6.1cm]{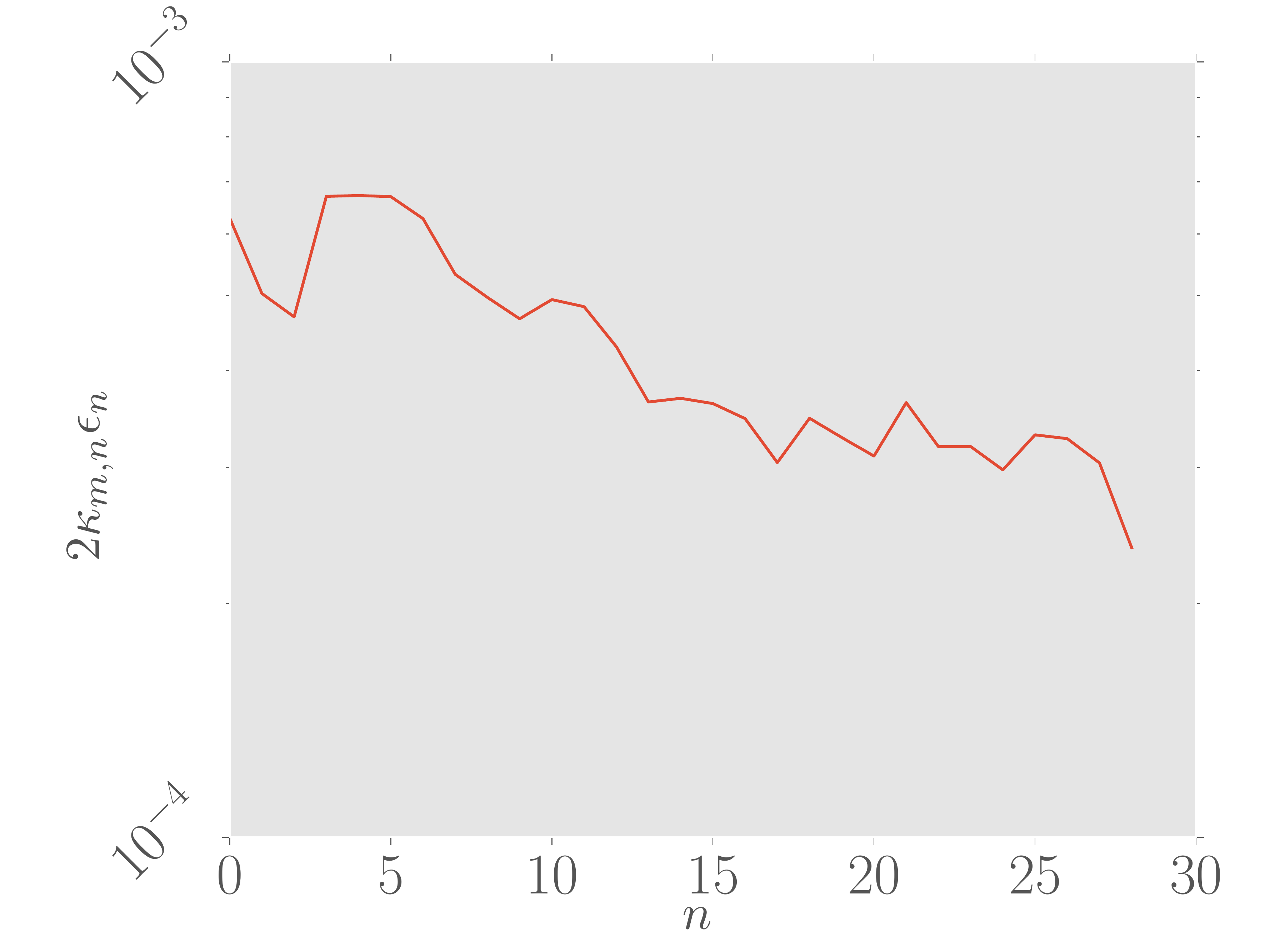}
\label{fig:mu_pbdw_joint_b}
}
\caption{Behavior of $\kappa_{m,n}$ respect to dimension of $V_n$ for each manifold partition in the method for joint reconstruction of section \ref{sec:joint_up}. In the right side we see the plot of $\max_{i} \{ 2 \kappa_{m,n} \epsilon_n \}$ for $i=1,\ldots,25$ denoting the worst among the 25 windows in the piecewise linear approach. this quantity is an upper bound of the pressure drop reconstruction error (see inequality \eqref{eq:err-p-drop}).
}
\label{fig:mu_pbdw_joint}
\end{figure}
The discrete version of equation \eqref{eq:kappa_bound_pdrop} leads to the optimization problem 
$$
\begin{aligned}
\max_{\eta \in \bR^N} \text{ }& \eta^T Q \eta \\
\text{s.t. } & \eta^T M \eta = 1,
\end{aligned}
$$
where
$$
M_{ij} = \innerp{\Psi_i - P_{V_n} \Psi_i}{ \Psi_j - P_{V_n} \Psi_j},
$$
and
$$
Q_{ij} = Q(\Psi_i) Q(\Psi_j).
$$
This problem is equivalent \fg{to} the generalized eigenvalue problem of finding $\eta \in \bR^N$ and the largest $\lambda \in \bR_+$ such that
\begin{equation}
Q \eta = \lambda M \eta
\label{eq:gevp}
\end{equation}
As a result, we can estimate the value of $\kappa_{m,n}$ with the largest eigenvalue $\lambda$ of problem \eqref{eq:gevp}. Figure \ref{fig:mu_pbdw_joint_a} shows the estimated value of $\kappa_{m,n}$ as a function of the dimension $n$ of the reduced model $V_n$. Since we use a $5\times 5$ partition of the manifold, we plot the 25 associated curves. From the figure, we deduce that $\kappa_{m,n} \leq 160$ for all partitions. As we see from Figure \ref{fig:mu_pbdw_joint_b}, the product $2\epsilon_n \kappa_{m,n}$ stays lower than $10^{-3}$ for any dimension $n$. As proven in \eqref{eq:err-p-drop}, this quantity is an upper bound of the reconstruction error and rigorously confirms the quality of the approach.

We can next examine in Figure \ref{fig:p_drop_virtual_work} the performance of the second method involving virtual works discussed in section \ref{sec:pdrop}. Although we observe good results like for the previous method, a mismatch is observed for one of the common carotid branches in the systolic phase. This error is not observed with the joint reconstruction, and it is probably due to the fact that the method involves less assumptions on the nature of the flow and pressure.

\begin{figure}[!htbp]
\centering
\includegraphics[height=5cm]{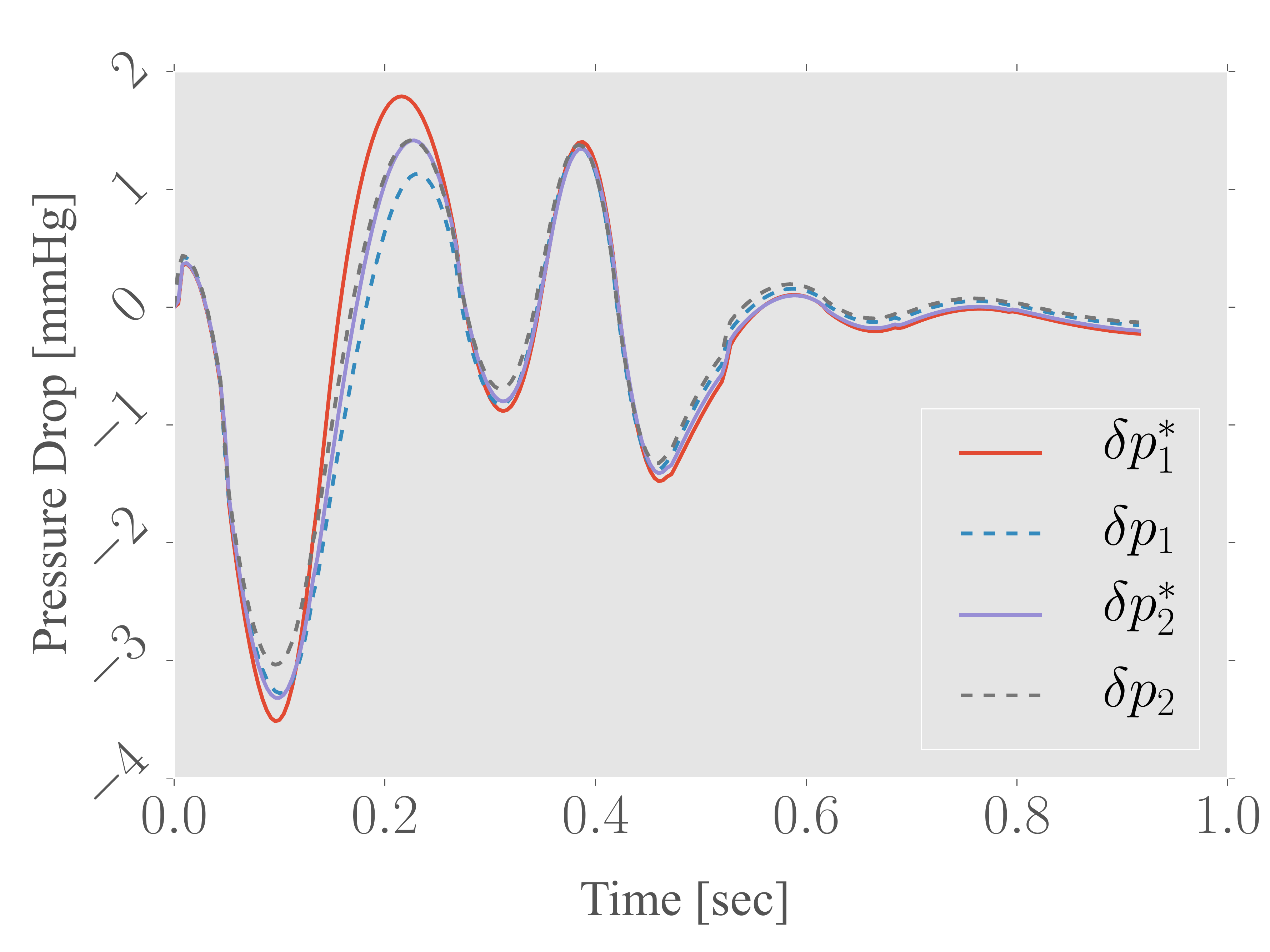}
\includegraphics[height=5cm]{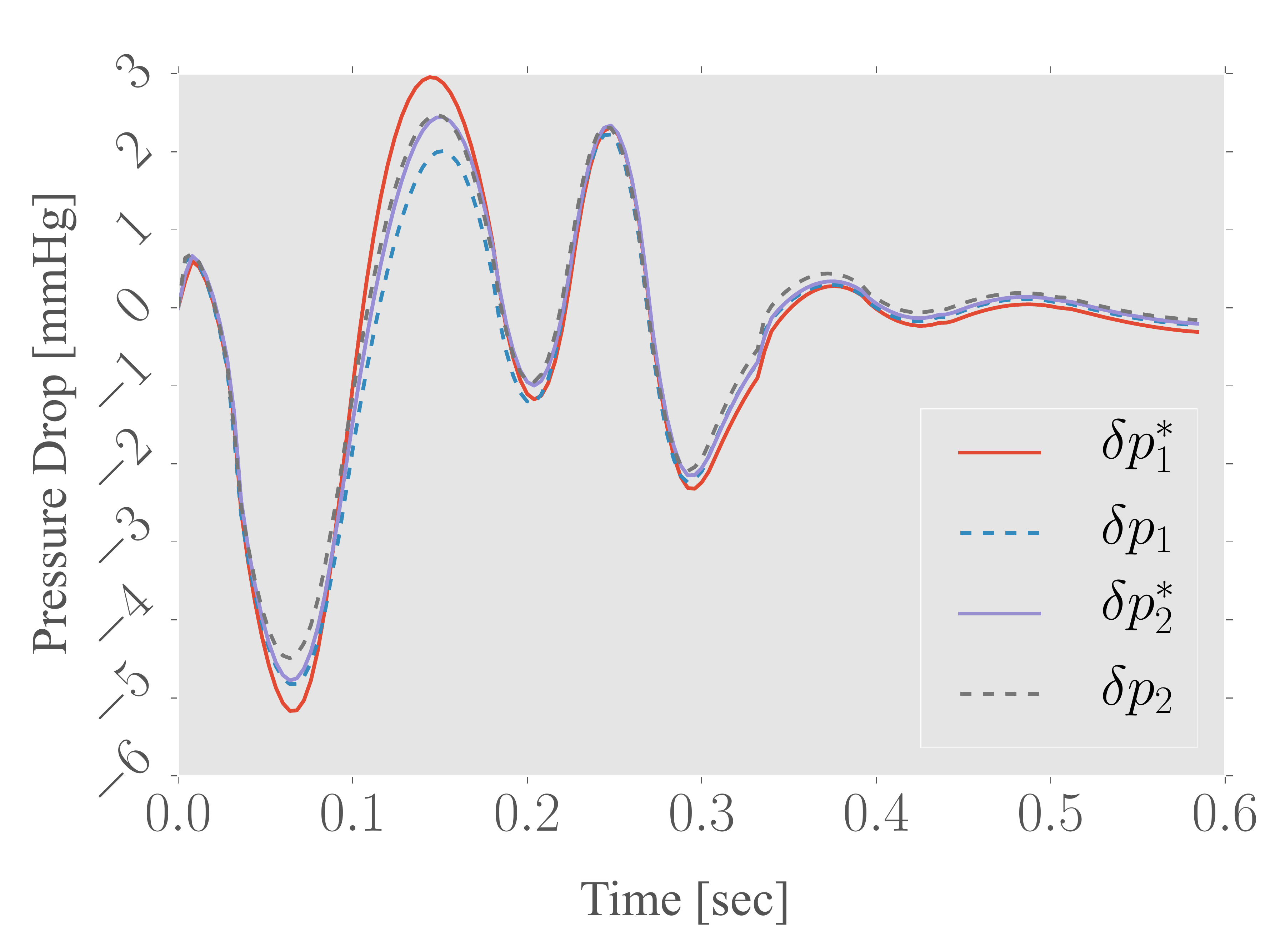}
\includegraphics[height=5cm]{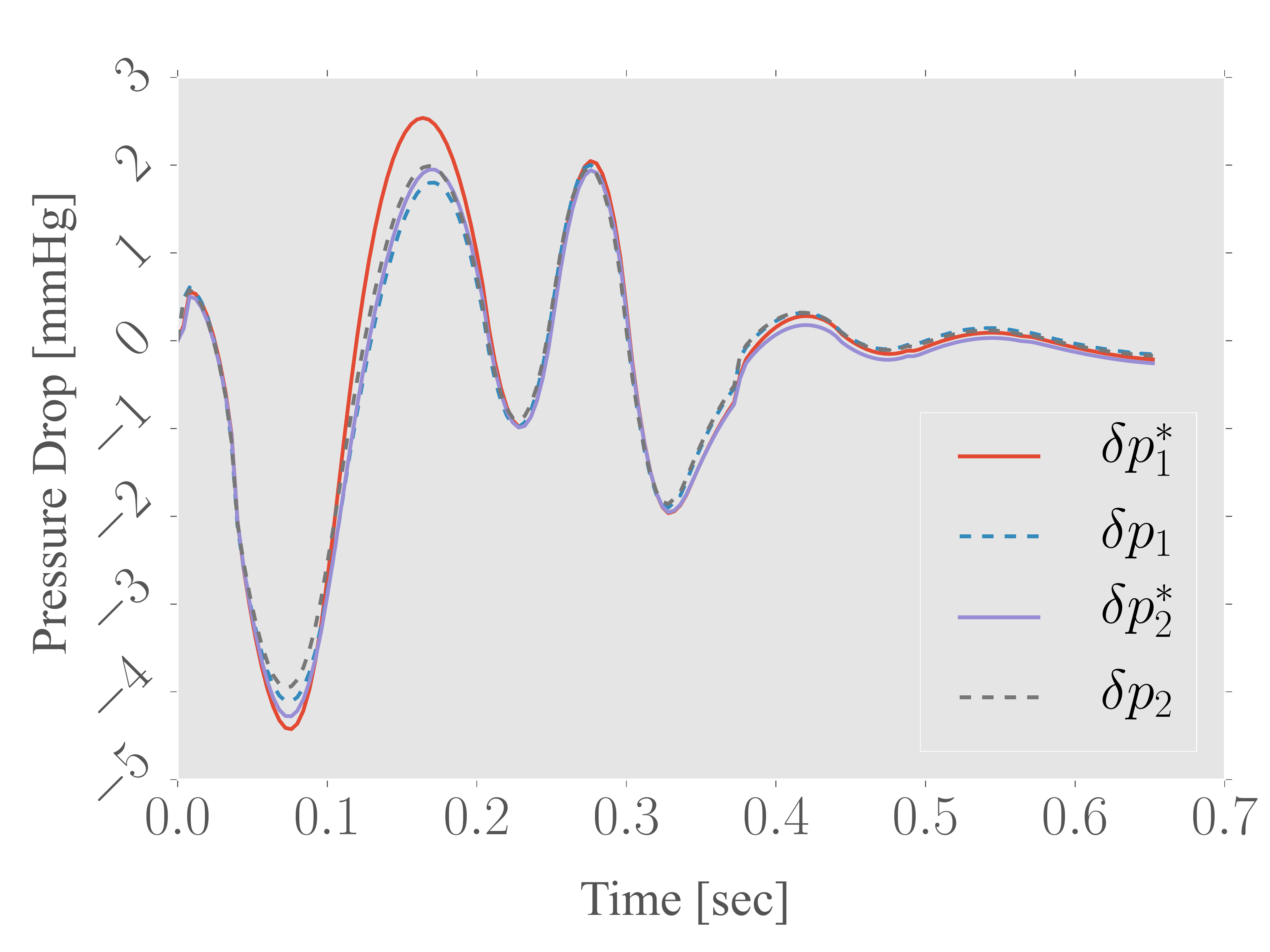}
\includegraphics[height=5cm]{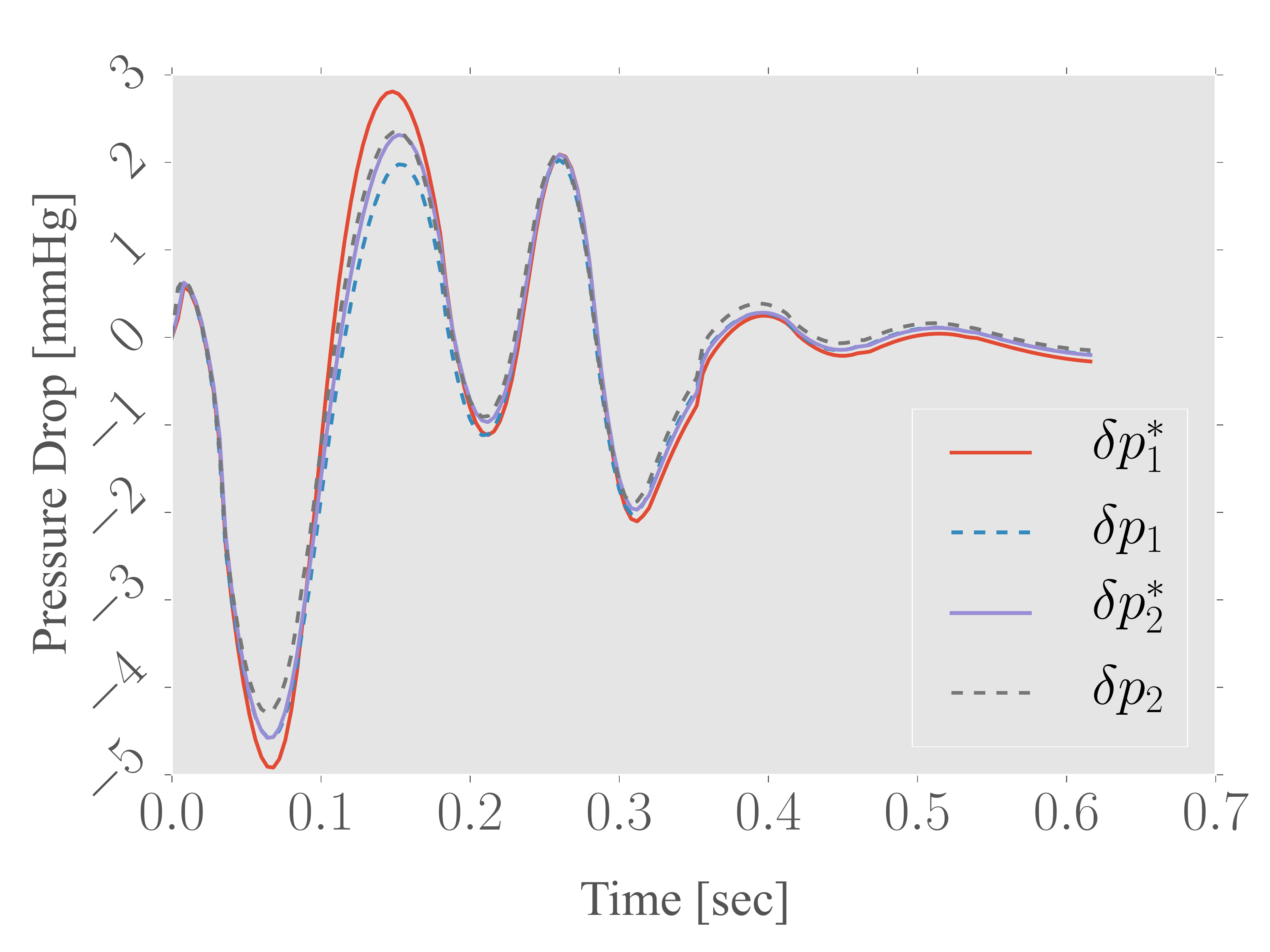}
\caption{Pressure drops $\delta p_1^*$ and $\delta p_2^*$ with virtual work principle for noise-free measures. We observe the time evolution for 4 test cases in the two carotid branches, and a comparison with the ground truths $\delta p_1$ and $\delta p_2$. Vertical axis shows the drop in [mmHg], whereas the horizontal axis shows the time evolution in seconds.}
\label{fig:p_drop_virtual_work}
\end{figure}

\section{Noisy measurements}
\label{sec:cls_test}
There are multiple sources of noise in CFI images, e.g., fake echo, reverberation, speckle, side lobes, ghosting, which result in a complicated space-time structure of the noise (see \cite{ledoux1997,bjaerum2002,demene2015}). Here we study the effect of noise in the admittedly simple case where we assume a gaussian perturbation of our observations of the form
\begin{equation}
z_i = \ell_i(u) + \eta_i
\label{eq:l_i_noise}
\end{equation}
where $\eta_i \sim \cN(0, \sigma^2)$. The noise is then independent at each voxel, averaged on the perfect measures. The standard deviation is chosen relative to the synthetic measures as
$$
\sigma = \frac{ \max_{t} \max_{i=1,\ldots,m} \ell_i(u(t)) }{ \alpha}
$$
where $\alpha>0$ is a parameter that steers the \textsl{noise level}.

As already observed in previous works, a naive reconstruction with the PBDW method with the noisy measurements $(z_i)_{i=1}^m$ is not asymptotically robust in the sense that when number $m$ of observations increases, the error bounds degrade essentially like $\sqrt{m}\sigma$. This has motivated the search for more stable formulations, and several approaches based on different types of regularization and thresholding have been proposed (see \cite{Taddei2017, ABGMM2017, GMMT2019}). Here we consider a simple variant based on a thresholding technique for $v_{m,n}^*$ (see equation \eqref{eq:vStarExplicit}) in the spirit of \cite{ABGMM2017}. To explain it, we first need to recall that in the noiseless case, $v_{m,n}^*$ is the unique minimizer of (see equations \eqref{eq:vn} and \eqref{eq:explicit-v} of the appendix)
\begin{equation}
\label{eq:vnmin}
v_{m,n}^* = \argmin_{v\in V_n} \frac 1 2 \Vert P_{W_m}u - P_{W_m} v \Vert^2.
\end{equation}
A slightly different approach is to find $\tilde v_{m,n}^*\in V_n$ as
\begin{equation}
\label{eq:tilde-vnmin}
\tilde v_{m,n}^* =
\argmin_{v \in V_n} \frac 1 2 \sum_{i=1}^m \vert \ell_i(u) - \ell_i(v) \vert_{\ell(\bR^m)}^2.
\end{equation}
We may note that, in general, $\tilde v_{m,n}^* \neq v_{m,n}^*$ except if $\{\omega_i\}_{i=1}^m$ is an orthonormal family in $V$. In presence of noise, we measure $z_i$ and not $\ell_i(u)$ so the minimization becomes
\begin{equation}
\label{eq:tilde-vnmin-noise}
\min_{v \in V_n} \frac 1 2 \sum_{i=1}^m \vert z_i - \ell_i(v) \vert_{\ell(\bR^m)}^2.
\end{equation}
To make this reconstruction more robust againt noise, instead of minimizing over the whole  space $V_n$, we can use the structure of the PDE solution manifold $\cM$ and minimize over its ``footprint'' on $V_n$, that is,
$$
\cK_n = P_{V_n} \cM \coloneqq \{ P_{V_n} u \, :\, u \in \cM \}.
$$
The resulting minimization reads
\begin{equation}
\label{eq:v-Kn}
\hat v^*_{m,n} = \argmin_{v \in \cK_n} \frac 1 2 \sum_{i=1}^m \vert z_i - \ell_i(v) \vert_{\ell(\bR^m)}^2.
\end{equation}
In practice, if $\{ v_i \}_{i=1}^n$ is an orthonormal basis of the space $V_n$, we can compute the coefficients $\textbf{c}^*\in \bR^n$ of $\hat v^*_{m,n}$ in this basis by solving the constrained least-squares problem
\begin{equation}
\begin{aligned}
\min_{\textbf{c} \in \bR^n} & \frac{1}{2} \sum_{i=1}^m \abs{ z_i -  \sum_{j=1}^n c_j \ell_i(v_j) }^2 \\
\text{s.t. }& \abs{c_j} \leq \max_{u \in \cM} \abs{\innerp{u}{v_j}},\quad j=1,\dots n.
\end{aligned}
\label{eq:ls_pod_discrete_constraints}
\end{equation}

We next study the reconstruction error with the unconstrained and constrained approaches \eqref{eq:tilde-vnmin-noise} and \eqref{eq:ls_pod_discrete_constraints}. Figures \ref{fig:cls_noise_u} and \ref{fig:cls_noise_up}  show respectively the error against the dimension $n$ of $V_n$ for the velocity reconstruction and velocity-pressure reconstruction. For each value of $n$, we compute the average error over 100 realizations of the noisy measurements for different levels $\alpha$ of the noise. The noiseless case is labeled $\alpha=\infty$ in the plots. The test case is focused on the first time partition, during the systolic phase of the cardiac cycle, and for snapshots in the lower heart rate partition. As expected, the quality of the reconstruction degrades when the level of the noise increases ($\alpha$ decreases). We observe that both constrained and unconstrained methods behave very similarly for a low number of modes. For the ambient space $V = U$, the constrained approach is able to grant a better reconstruction as we increase the dimension of the space $V_n$. However, for the ambient space $V = U\times P$, the constrained approach does not bring any improvement with respect to the unconstrained one.

\begin{figure}[!htbp]
\centering
\label{fig:cls_noise}
\subfigure[$V=U$.]{
\includegraphics[height=5.5cm]{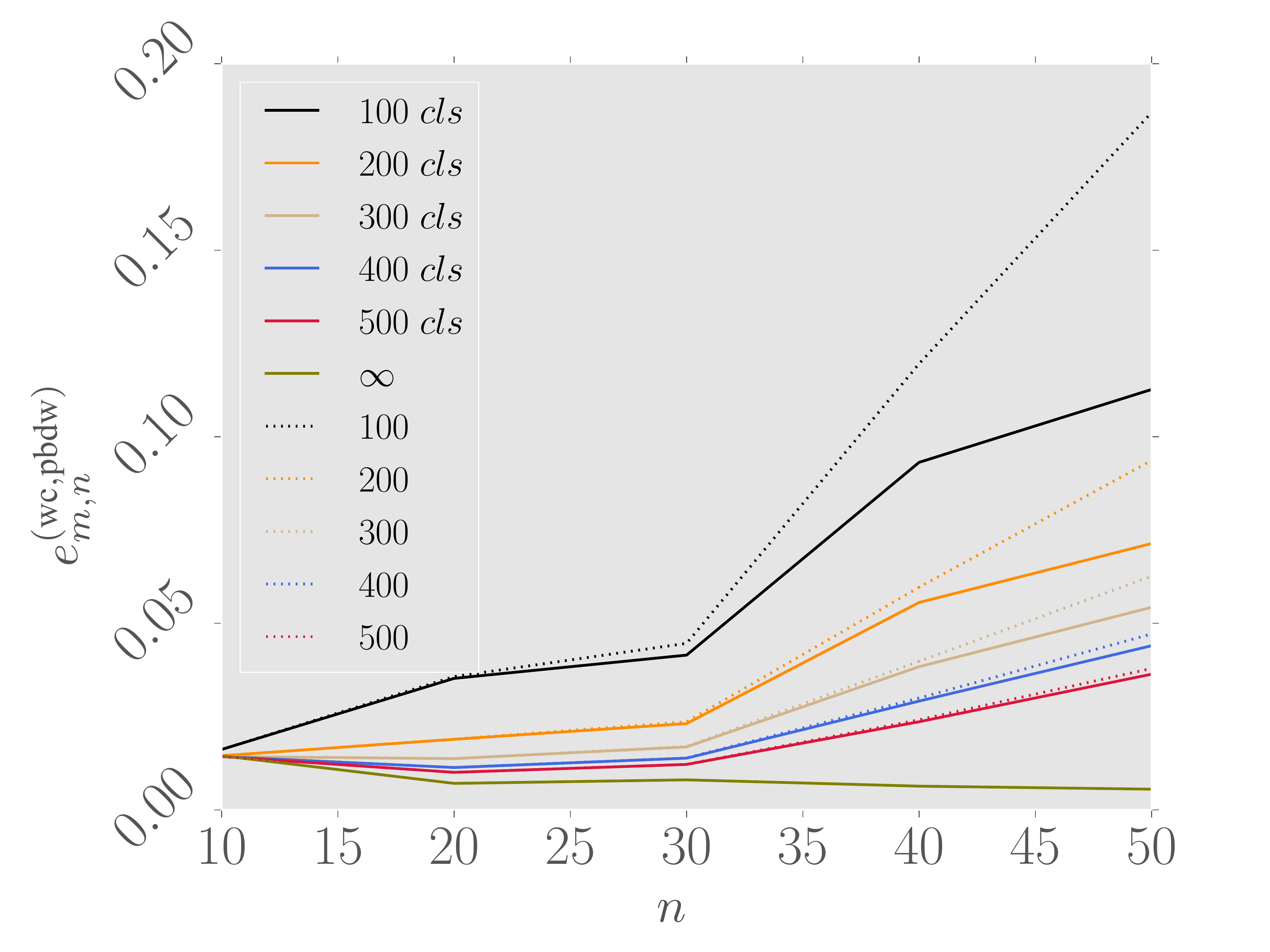}
\label{fig:cls_noise_u}
}
\subfigure[$V=U \times P$.]{
\includegraphics[height=5.5cm]{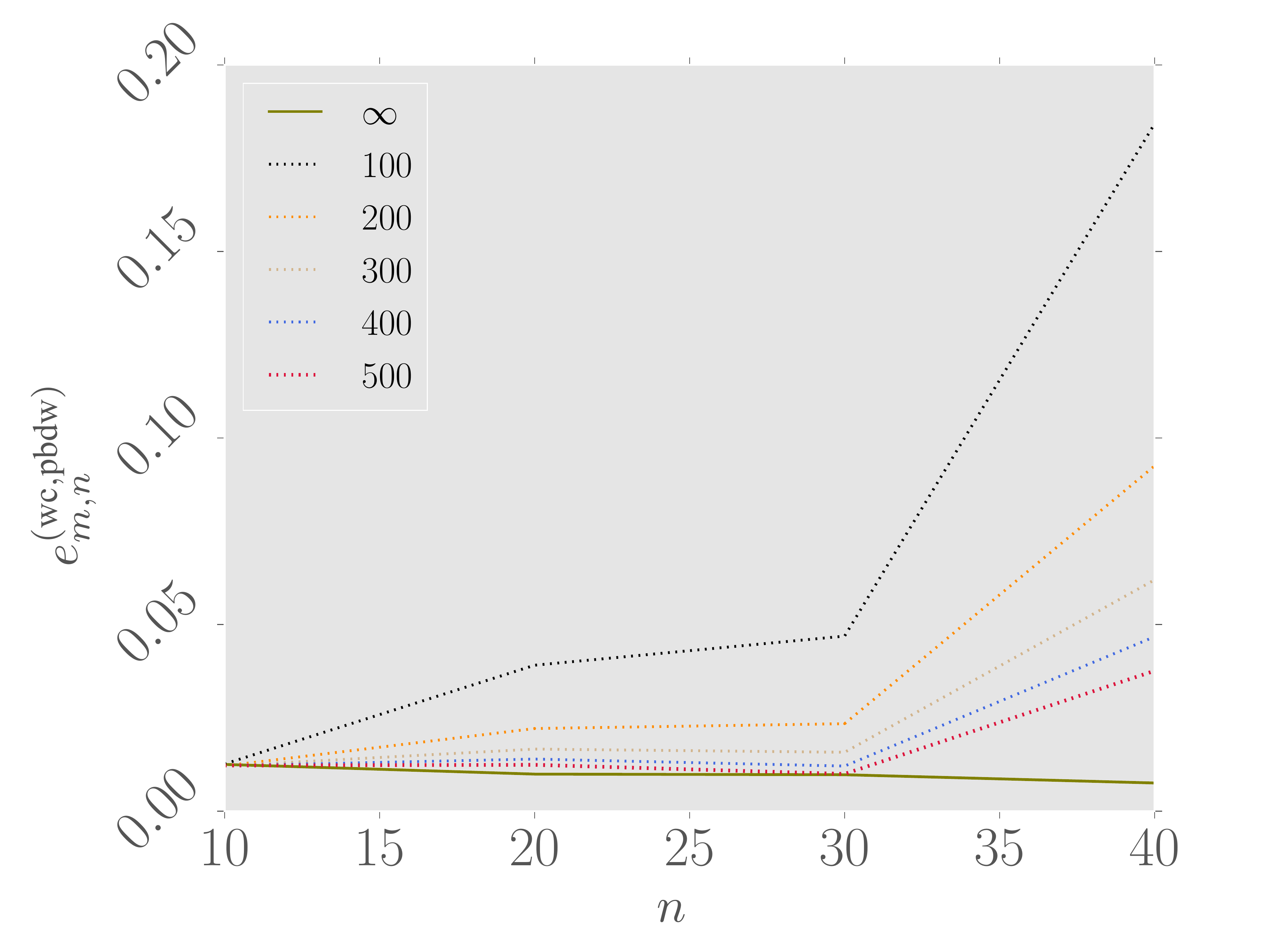}
\label{fig:cls_noise_up}
}
\caption{Reconstruction error \eqref{eq:err-wc-pbdw} in one of the manifold partitions. Dots: unconstrained approach \eqref{eq:tilde-vnmin-noise}. Full line: constrained approach \eqref{eq:ls_pod_discrete_constraints}.  Curves for the constrained and unconstrained approach overlap for the joint reconstruction in $V = U \times P$, showing that the constraints do not bring any improvement.}
\end{figure}
\begin{figure}[!htbp]
\centering
\subfigure[Solution 1, velocity.]{\includegraphics[height=5cm]{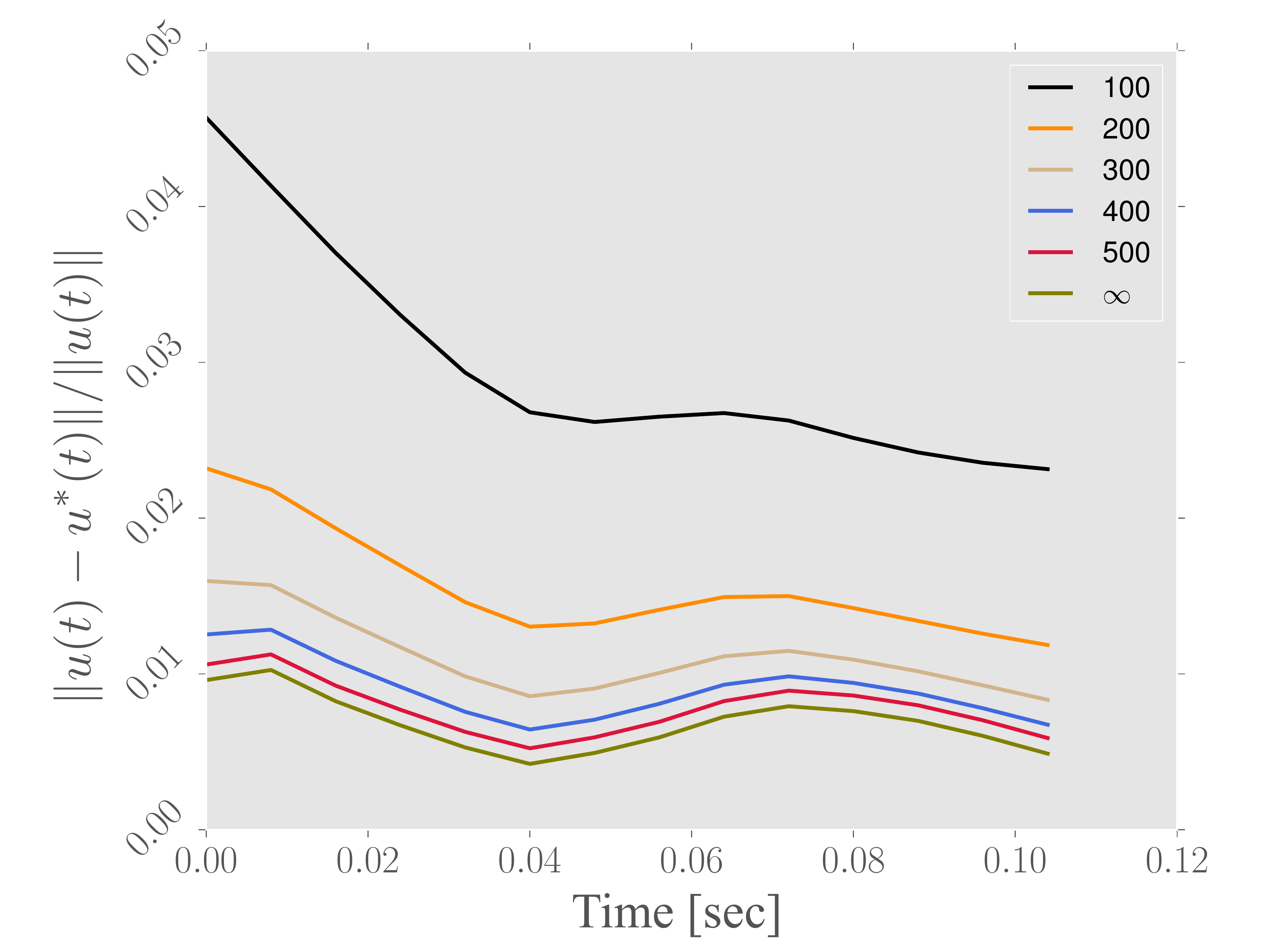}}
\subfigure[Solution 1, pressure.]{\includegraphics[height=5cm]{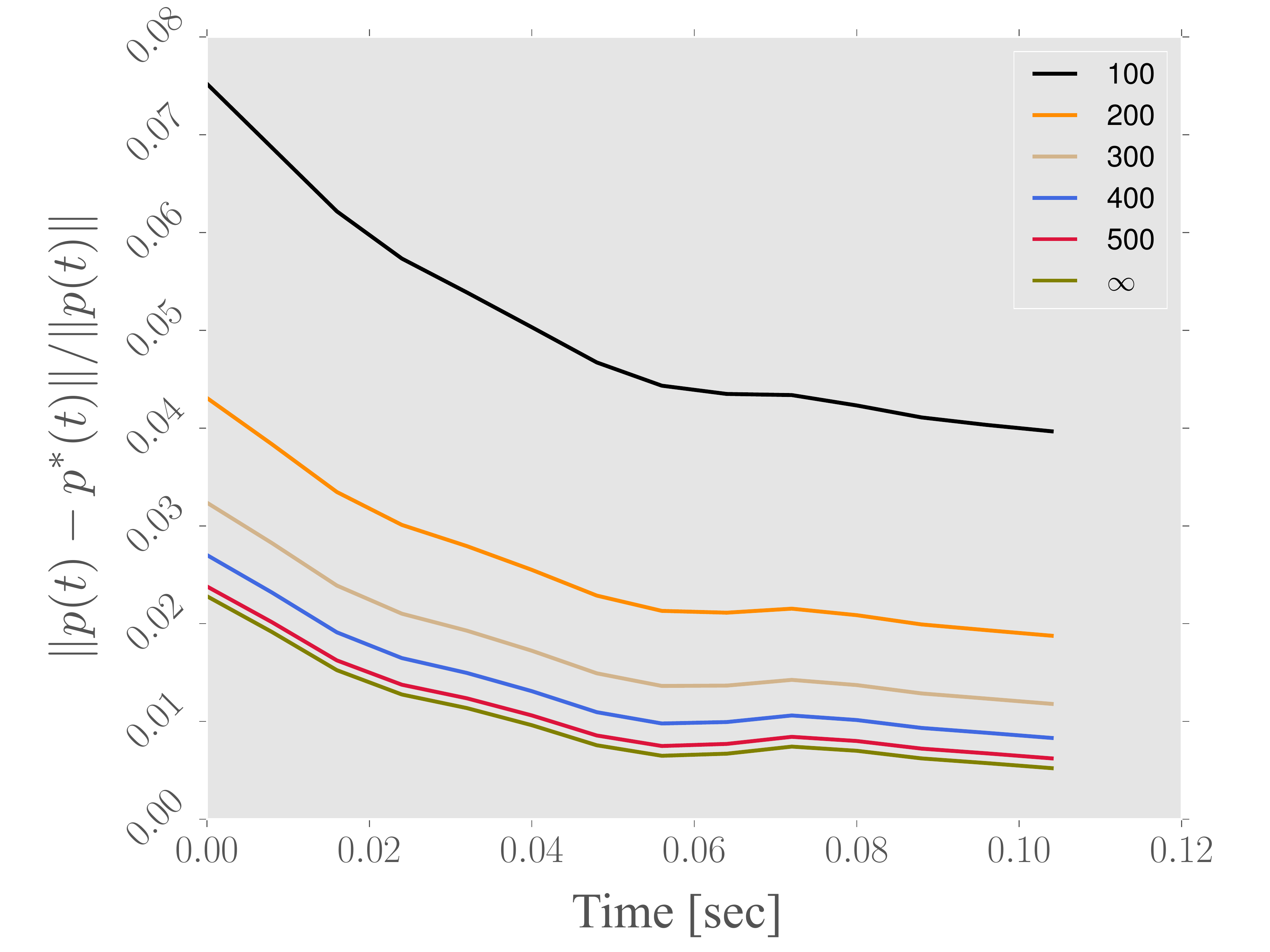}}
\subfigure[Solution 13, velocity.]{\includegraphics[height=5cm]{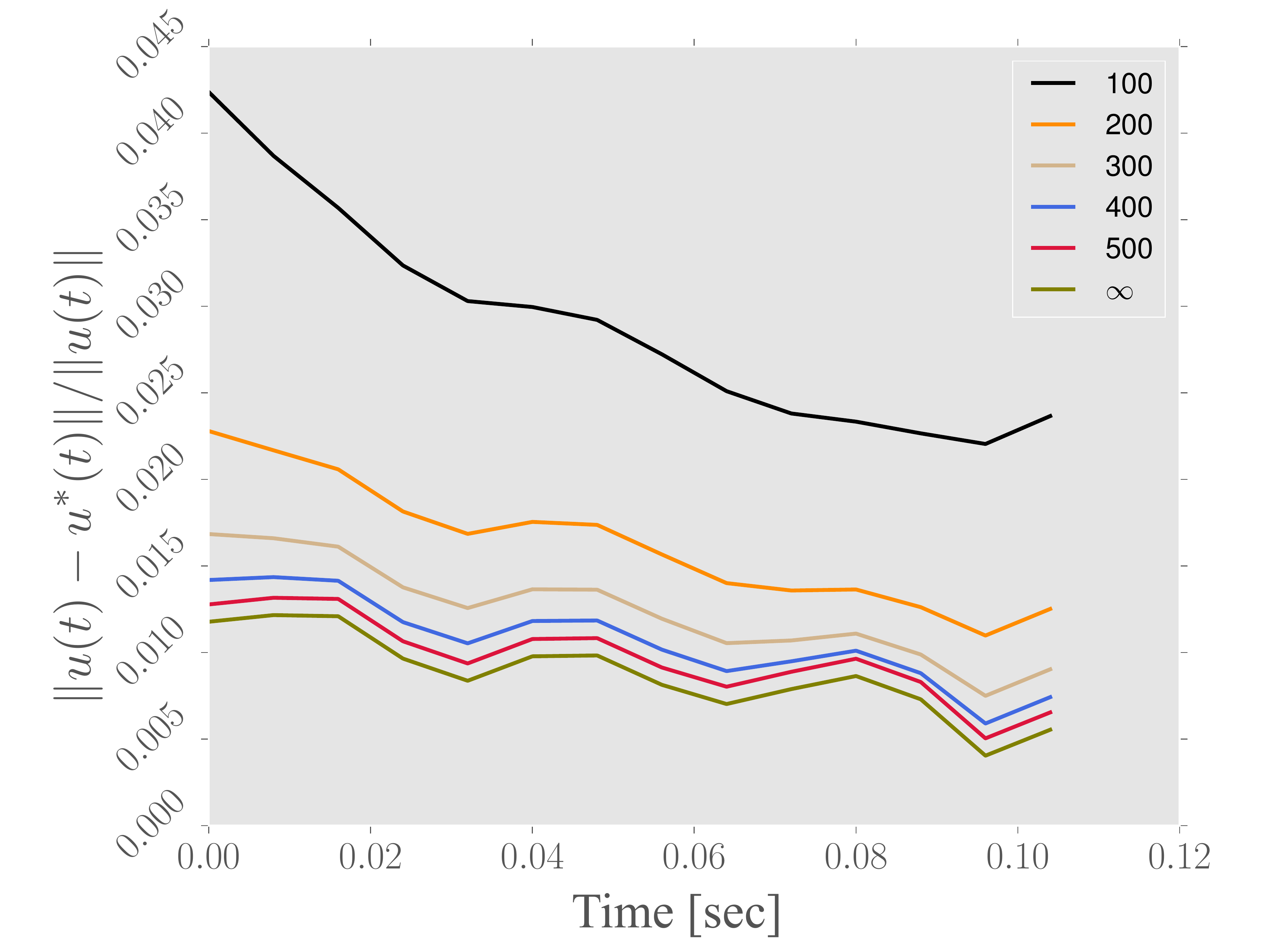}}
\subfigure[Solution 13, pressure.]{\includegraphics[height=5cm]{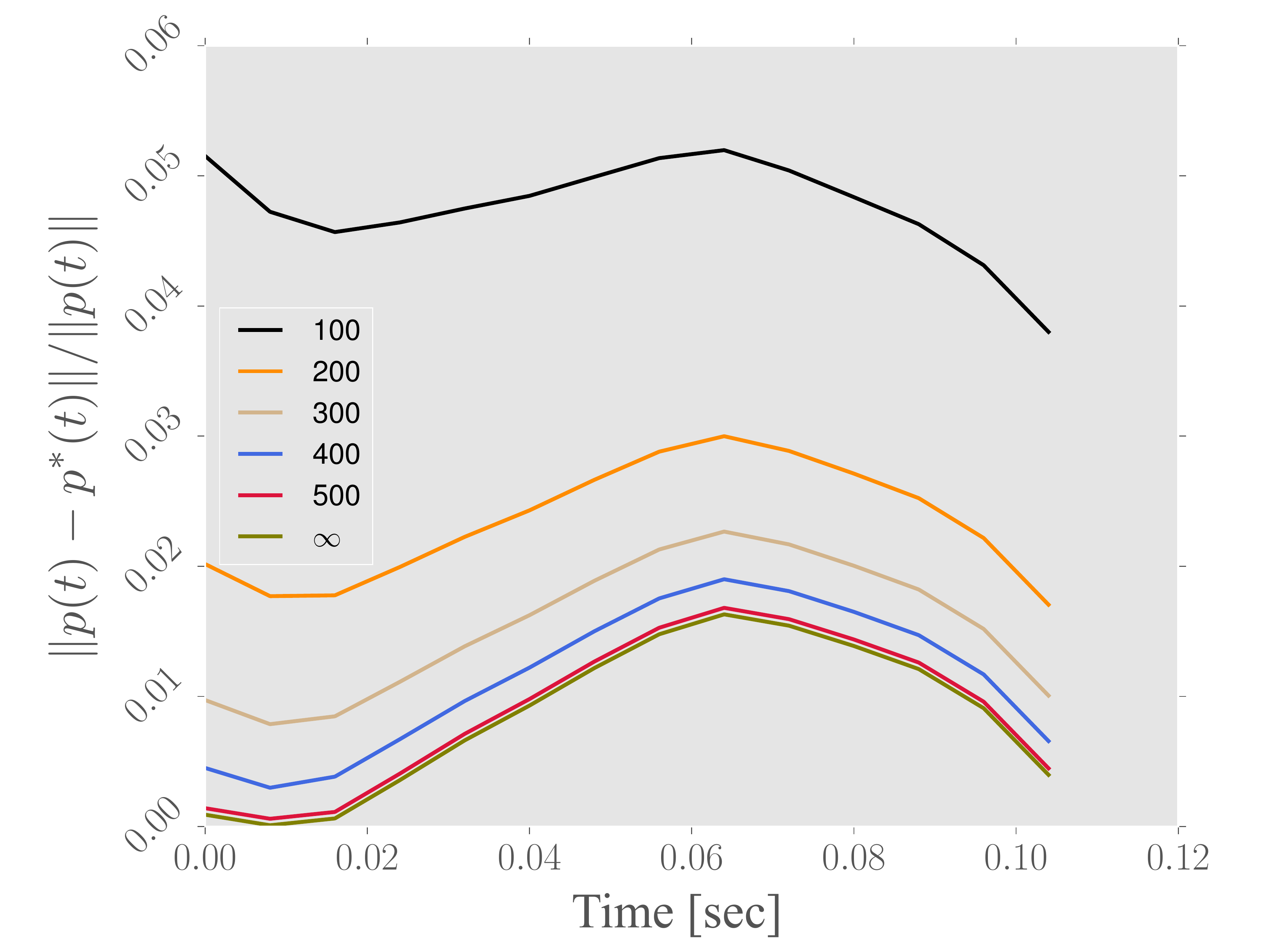}}
\caption{\fg{Relative errors in joint reconstruction for two time dependent solutions. Figure \ref{fig:cls_noise_up} suggests us to use 30 modes in the reconstruction.}}
\label{fig:u_p_noise_30}
\end{figure}

Figures \ref{fig:u_p_noise_30} and \ref{fig:pbdw_mixed_pdrop} show the reconstruction error for velocity and pressure fields, and the pressure drop computed from the joint reconstruction in $V = U \times P$. As in the noise-free numerical experiment, the reconstruction output is very satisfactory for all quantities. We observe that the pressure drop reconstruction is more robust to noise than the reconstruction of the full 3D pressure field although the reconstructed pressure drop is derived from the reconstructed 3D field. For all the cases, the reconstruction was done with a dimension $n=30$ for $V_n$.
\begin{figure}[!htbp]
\centering
\subfigure[\dl{Solution} 1, $\Gamma_o^1$.]{\includegraphics[height=6cm]{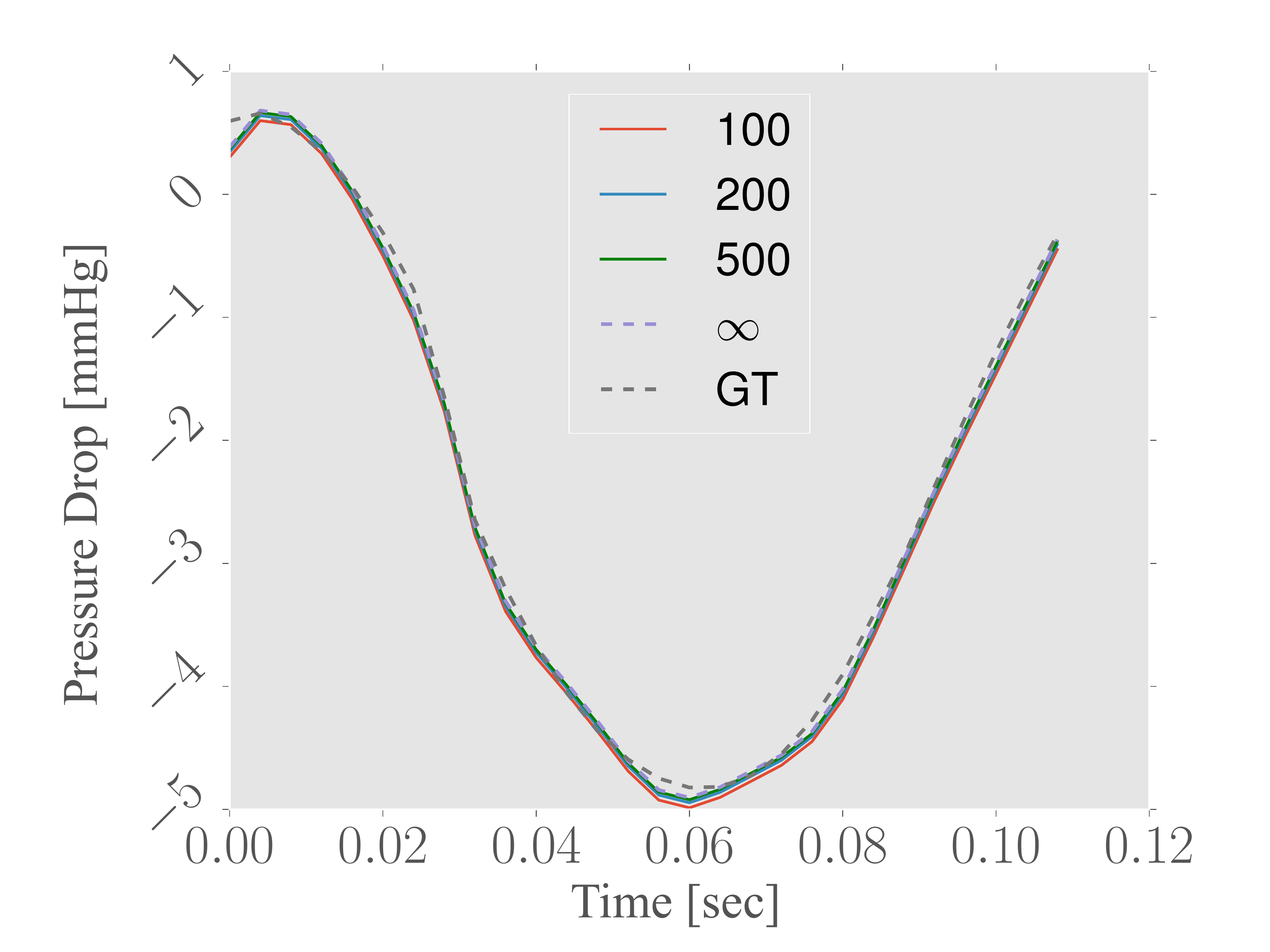}}
\subfigure[\dl{Solution} 1, $\Gamma_o^2$.]{\includegraphics[height=6cm]{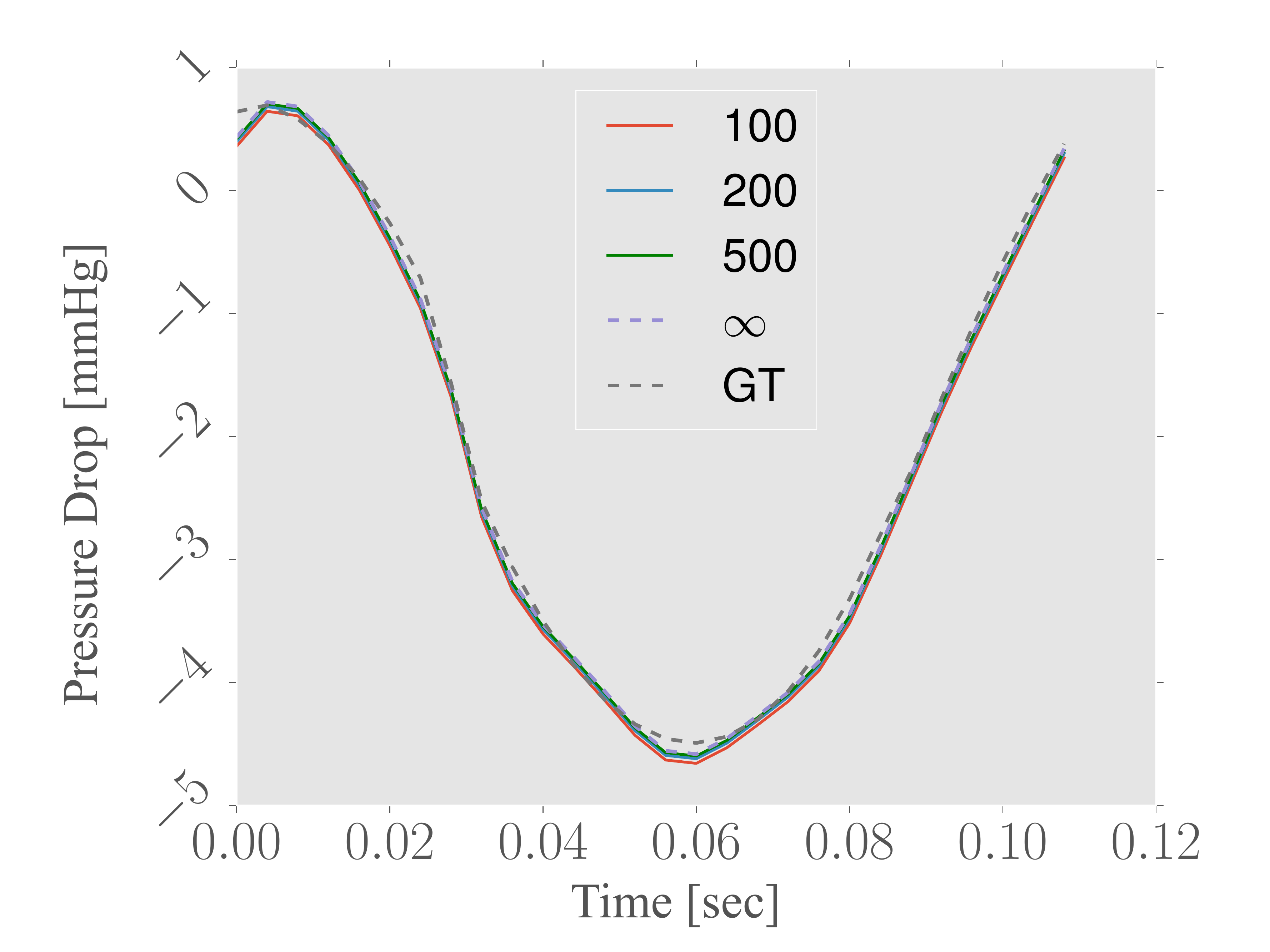}}
\subfigure[\dl{Solution} 13, $\Gamma_o^1$.]{\includegraphics[height=6cm]{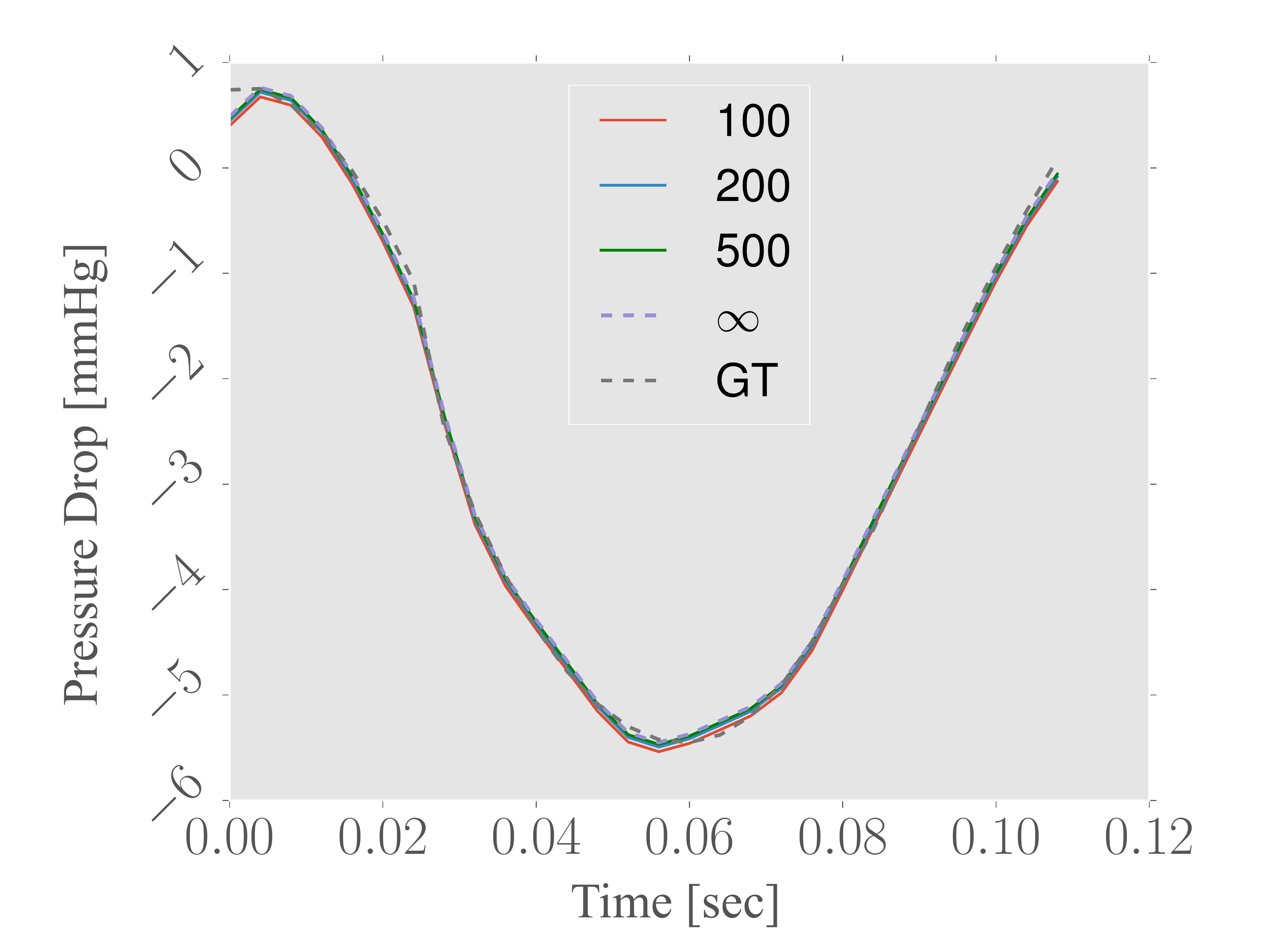}}
\subfigure[\dl{Solution} 13, $\Gamma_o^2$.]{\includegraphics[height=6cm]{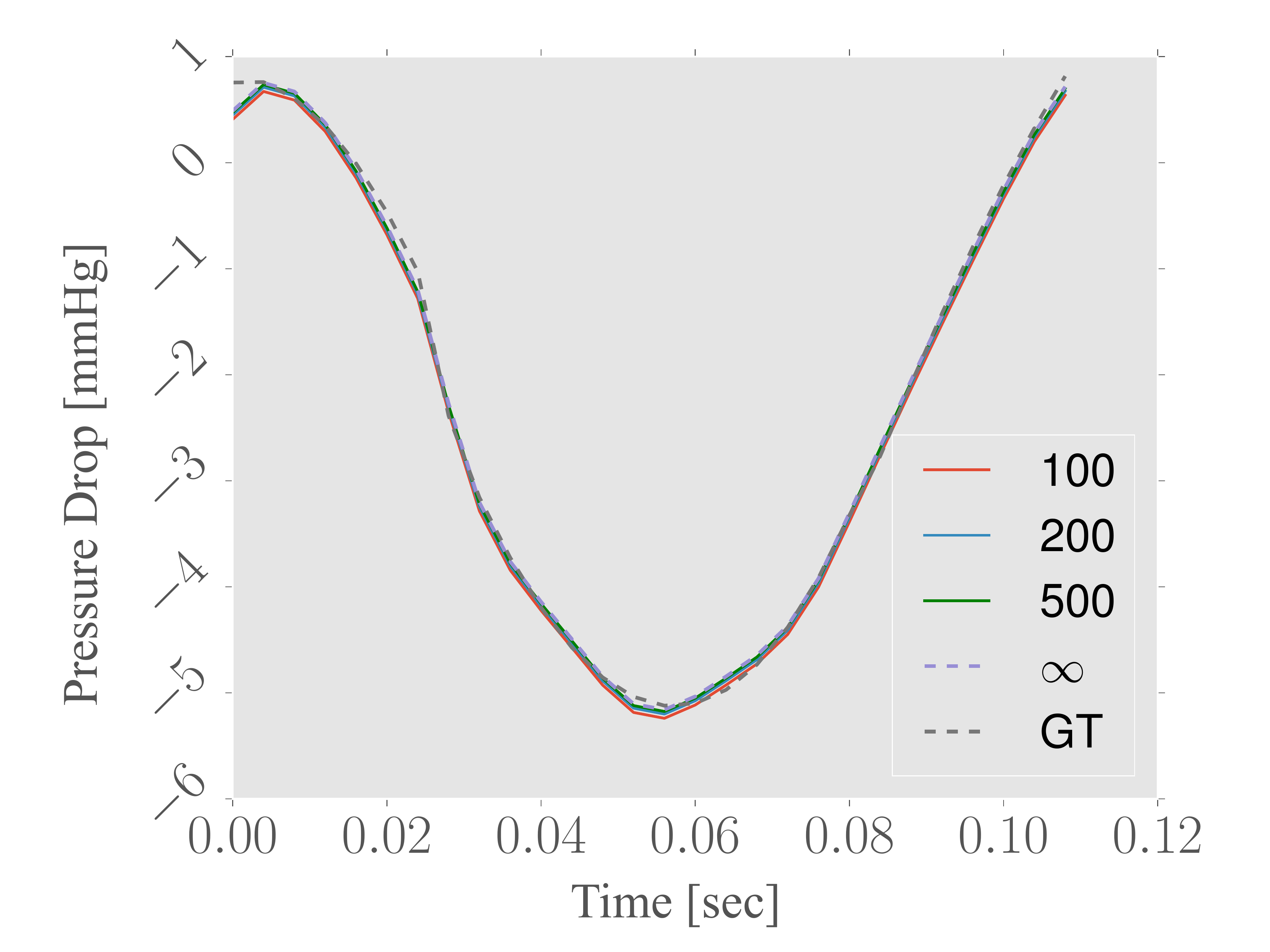}}
\caption{Reconstruction of pressure drop for 3 noise levels in two \dl{different solutions}. The results are presented for the early systole phase. Vertical axis shows the pressure drop in [mmHg] and the horizontal axis the time in seconds.}
\label{fig:pbdw_mixed_pdrop}
\end{figure}

\section{Conclusions and perspectives}

We conclude this paper by summarizing the main topic and contributions. We have proposed a systematic methodology involving reduced modelling to give quick and reliable estimations of QoI in biological fluid flows. We have assessed the feasibility of the approach in non trivial numerical examples involving the carotid artery. The numerical examples include:
\begin{itemize}
\item The reconstruction of velocity related quantities such as vorticity and wall shear stress from Doppler data, both holding average errors below the $5 \%$ in an $H^1$ sense. \fg{At worst, we have observed a maximal error during the cardiac cycle for one of the test cases below the $10 \%$.}
\item The reconstruction of unobserved QoI from Doppler data such as pressure fields and pressure drops, with comparisons with other state-of-the-art techniques \cite{bertoglio_pdrop}.
\item The simulation of semi-realistic measures by considering white noise in the input signals.
\item A theoretical study of the reconstruction error in all QoI. In particular, the numerical results confirm that the bound for the pressure drop estimation is rather sharp.
\end{itemize}
Although the present results are promising, they remain a proof-of-concept since the Doppler images and the flows serving as the ground truth are synthetically generated. 
To go further \dl{we need to include in the current methodology the uncertainty and the anatomic variability of the vascular geometry.  
For a given patient, the geometry is only approximately known (acquired by medical imaging) and this translates into an uncertainty in the predictions. Also, when we consider multiple patients, they all have a different geometric configurations for the same artery. To take into account these geometrical variations and uncertainties, we are currently extending the current methodolgy which, so far, heavily relies on the assumption that the spatial domain is fixed in the offline and online phases.} In addition, the methodology needs to be validated with real flows and real ultrasound images. This step poses however a certain number of challenges that we will address in a collaboration involving medical doctors and experts in 3D printing. The main roadmap is: (i) to manufacture arteries with similar mechanical properties as the biological ones, and favorable optical properties to collect ultrasound and PIV measurements. (ii) Once this is done, we will collect the ultrasound images and feed our reconstruction algorithms. We will compare the reconstructions with PIV images \dl{for what concerns velocity and with catheter measurements for what concerns pressure.}

\section*{Acknowledgements}

This research was supported by the EMERGENCES grant of the Paris City Council “Models and Measures”. Also, authors thankfully acknowledge the financial support of the ANID Ph. D. Scholarship 72180473.

\appendix
\section{Explicit expression and algebraic formulation of $u_{m,n}^*(\omega)$, the function given by the linear PBDW algorithm}
\label{appendix:linear-pbdw}
Let $X$ and $Y$ be two finite dimensional subspaces of $V$ and let
\begin{align*}
P_{X | Y} : Y &\to X \\
y &\mapsto P_{X | Y} (y)
\end{align*}
be the orthogonal projection into $X$ restricted to $Y$. That is, for any $y\in Y$, $P_{X | Y}(y)$ is the unique element $x\in X$ such that 
$$
\left< y - x, \tilde x\right>= 0,\quad  \forall \tilde x \in X.
$$

\begin{lemma}
Let $\Wm$ and $\Vn$ be an observation space and a reduced basis of dimension $n\leq m$ such that $\beta(\Vn, \Wm)>0$. Then the linear PBDW algorithm is given by
\begin{equation}
\label{eq:explicit-u}
u_{m,n}^*(\omega) = \omega + v^*_{m,n} - P_W v^*_{m,n},
\end{equation}
with
\begin{equation}
\label{eq:explicit-v}
v^*_{m,n} = \( P_{V_n | \Wm} P_{\Wm | V_n} \)^{-1} P_{V_n | \Wm} (\omega).
\end{equation}
\end{lemma}
\begin{proof}
By formula \eqref{eq:uStar}, $u^*_{m,n}(\omega)$ is a minimizer of
\begin{align}
\min_{u \in \omega + \Wm^\perp } \dist(u, V_n)^2
&= \min_{u \in \omega + \Wm^\perp } \min_{v\in V_n} \Vert u - v\Vert^2 \\
&= \min_{v\in V_n}  \min_{\eta \in \Wm^\perp} \Vert \omega + \eta - v \Vert^2 \\
&= \min_{v\in V_n} \Vert \omega -v - P_{\Wm^\perp}(\omega -v) \Vert^2 \\
&= \min_{v\in V_n} \Vert \omega -v + P_{\Wm^\perp}(v) \Vert^2 \\
&= \min_{v\in V_n} \Vert \omega - P_{\Wm}(v) \Vert^2 \label{eq:vn}
\end{align}
The last minimization problem is a classical least squares optimization. Any minimizer $v^*_{m,n}\in V_n$ satisfies the normal equations
$$
P^*_{\Wm|V_n}  P_{\Wm|V_n} v^*_{m,n} = P^*_{\Wm|V_n} \omega,
$$
where $P^*_{\Wm|V_n} : V_n \to \Wm$ is the adjoint operator of $P_{\Wm|V_n}$. Note that $P^*_{\Wm|V_n}$ is well defined since $\beta(V_n, \Wm)= \min_{v\in V_n} \Vert P_{\Wm|V_n} v \Vert / \Vert v \Vert >0 $, which implies that $P_{\Wm|V_n}$ is injective and thus admits an adjoint. Furthermore, since for any $\omega \in \Wm$ and $v\in V_n$, $\< v, \omega \>=\< P_{\Wm|V_n} v , \omega \> = \< v , P_{V_n|\Wm} \omega \>$, it follows that $P^*_{\Wm|V_n} = P_{V_n|\Wm}$, which finally yields that the unique solution of the least squares problem is
$$
v^*_{m,n} = \( P_{V_n | \Wm} P_{\Wm | V_n} \)^{-1} P_{V_n | \Wm} \omega .
$$
Therefore $u^*_{m,n} = \omega + \eta^*_{m,n} = \omega + v^*_{m,n} - P_\Wm v^*_{m,n}$.
\end{proof}

\textbf{Algebraic formulation:} The explicit expression \eqref{eq:explicit-v} for $v^*_n$ allows to easily derive its algebraic formulation. Let $F$ and $H$ be two finite-dimensional \dl{sub}spaces of $V$ of dimensions $n$ and $m$ respectively in the Hilbert space $V$ and let $\cF=\{f_i\}_{i=1}^n$ and $\cH=\{h_i\}_{i=1}^m$ be a basis for each subspace respectively. The Gram matrix associated to $\cF$ and $\cH$ is
$$
\bG(\cF, \cH) = \left(  \left< f_i, h_j\right> \right)_{\substack{1\leq i \leq n \\ 1\leq j \leq m}}.
$$
These matrices are useful to express the orthogonal projection
$P_{ F | H}: H\mapsto F$ in the bases $\cF$ and $\cH$ in terms of the matrix
$$
\bP_{F | H} = \bG(\cF, \cF)^{-1} \bG(\cF, \cH).
$$
As a consequence, if $\cV_n = \{ v_i \}_{i=1}^n$ is a basis of the space $V_n$ and $\cW_m = \{\omega_i\}_{i=1}^m$ is the basis of $W_m$ formed by the Riesz representers of the linear functionals $\{\ell_i\}_{i=1}^m$, the coefficients $\textbf{v}^*_{m,n}$ of the function $v^*_{m,n}$ in the basis $\cV_n$ are the solution to the normal equations
$$
\bP_{V_n | W_m} \bP_{W_m | V_n} 
\textbf{v}^*_{m,n} =  \bP_{V_n | W_m} \bG(\cW_m, \cW_m)^{-1} \textbf{w},
$$
where
$$
\bP_{V_n | W_m} = \bP_{V_n | W_m}^T
$$
since $P^*_{\Wm|V_n} = P_{\Wm|V_n}$ and $\textbf{w}$ is the vector of measurement observations
$$
\textbf{w} = (\left< u, \omega_i\right>)_{i=1}^m.
$$
Usually $\textbf{v}^*_{m,n}$ is computed with a QR decomposition or any other suitable method. Once $\textbf{v}^*_{m,n}$ is found, the vector of coefficients $\textbf{u}_{m,n}^*$ of $u^*_{m,n}$ easily follows.

\section{Details about the discretization of the Navier-Stokes equations and the ODE coupling}
\label{appendixDisc}
\fgal{
We use the following \textsl{semi-implicit} time discretization for the weak form of the Navier-Stokes equations \eqref{eq:Navier_Stokes}
\begin{equation}
\begin{aligned}
\rho \int_{\Omega}  \frac{u^{n+1} - u^n}{\Delta t} \cdot v ~\dx + \rho \int_{\Omega}   \( \nabla u^{n+1} \) u^{n}\cdot v~\dx + \mu \int_\Omega \nabla u^{n+1} : \nabla v ~ \dx - \int_{\Omega} \( \nabla \cdot v \) p^{n+1} ~\dx & \\ + \int_{\Omega} \(\nabla \cdot u^{n+1} \) q ~\dx + \int_{\Gamma_\outlet^1 \cup \Gamma_\outlet^2} \( \[ \mu \frac{\nabla u^{n+1} + \nabla^T \(u^{n+1}\)}{2} + p^{n+1} I_{3 \times 3} \] n \) \cdot v ~\ds = 0,
\end{aligned}
\label{eq:weak_form_continuous_NS}
\end{equation}
$\forall \(v, q \) \in [H^1(\Omega)]^3_0 \times L^2(\Omega)$. $[H^1(\Omega)]^3_0$ is the test space in $[H^1(\Omega)]^3$ with zero trace on the Dirichlet boundaries. $I_{3 \times 3}$ is an identity matrix of size $3$. This semi-implicit approach for the convective term allows us to avoid a root finding problem when computing the time-marching solutions. Thus, the problem is reduced to find the functions $u^1,\ldots,u^n,\ldots,u^{T/\Delta t}$ and $p^1,\ldots,p^n,\ldots, p^{T/\Delta t}$. In our computations, the time-step was set to $\Delta t = 2 \cdot 10^{-3} s$.} 

\fgal{The mixed problem for velocity and pressure is discretized in space using $\mathbb{P}_1-\mathbb{P}_1$ Lagrange elements. In order to avoid the inf-sup constraint imposed by the saddle point nature of the problem we use the Brezzi-Pitkäranta stabilization technique modifying the discrete equations \cite{brezzi1984}. Standard SUPG stabilization for convection dominated flows is used \cite{supg}. In addition, a backflow stabilization is added in order to address potential instabilities in the outlet boundaries (see, e.g., \cite{black_flow_benchmark_bertoglio} for a survey). Spatial discretization of the carotid geometry leads to a tetrahedron mesh with 42659 vertices.}

\fgal{Concerning the numerical solution of the ODE for the distal pressure in the Windkessel model, and its posterior coupling to the Navier-Stokes equations, we need to address the discretization of the boundary integrals on $\Gamma_\outlet^1$ and $\Gamma_\outlet^2$ in the weak form \eqref{eq:weak_form_continuous_NS}. The stress tensor, which we recall is defined as $\sigma (u,p) = \mu \( \nabla u + \nabla^T u \)/2 + p I_{3 \times 3}$ is used to couple the system with the ODE for the pressure $\bar p_{o,k}$, by using the non-homogeneous Neumann boundary condition $\(\sigma \) n = \bar p_{o,k} (1,1,1)^T$ on $\Gamma_\outlet^k$, $k=1,2$. As for the states $u$ and $p$, we have to compute the time marching solutions $ \(\bar p_{o,k}\)^1, \ldots, \(\bar p_{o,k} \)^n, \ldots, \(\bar p_{o,k} \)^{T/\Delta t}$. So as times goes by we calculate for every time step $n$
$$
\(\bar p_{o,k}\)^{n+1} = \(p_d^k\)^{n+1} + R_p^k \int_{\Gamma_\outlet^k} u^n \cdot n ~\ds ,\quad k=1,\,2,
$$
and the ODE for the distal pressures is discretized explictely so that their temporal evolution is given by
$$
\(p_d^k\)^{n+1} = \(p_d^k\)^n \left(1 - \frac{\Delta t}{C_d^k}\right) + \frac{\Delta t}{C_d^k} \int_{\Gamma_{\outlet}^k} u^n \cdot n ~\ds,\quad k=1,\,2.
$$}

\newpage
\bibliographystyle{unsrt}
\bibliography{GLM2020.bbl}

\begin{thebibliography}{10}

\bibitem{hata1987}
T.~Hata, S.~Aoki, K.~Hata, and M.~Kitao.
\newblock Intracardiac blood flow velocity waveforms in normal fetuses in
  utero.
\newblock {\em The American journal of cardiology}, 59(5):464--468, 1987.

\bibitem{galarce2020}
F.~Galarce, J.-F. Gerbeau, D.~Lombardi, and O~Mula.
\newblock Fast reconstruction of {3D} blood flows from doppler ultrasound
  images and reduced models.
\newblock {\em Accepted on Computer Methods in Applied Mechanics and
  Engineering}, 2020.

\bibitem{hatle1978}
L.~Hatle, A.~Brubakk, A.~Tromsdal, and B.~Angelsen.
\newblock Noninvasive assessment of pressure drop in mitral stenosis by doppler
  ultrasound.
\newblock {\em Heart}, 40(2):131--140, 1978.

\bibitem{hatle1979}
L.~Hatle, B.~Angelsen, and A.~Tromsdal.
\newblock Noninvasive assessment of atrioventricular pressure half-time by
  doppler ultrasound.
\newblock {\em Circulation}, 60(5):1096--1104, 1979.

\bibitem{hatle1980}
L.~Hatle, B.~Angelsen, and A.~Tromsdal.
\newblock Non-invasive assessment of aortic stenosis by doppler ultrasound.
\newblock {\em Heart}, 43(3):284--292, 1980.

\bibitem{mates1978}
R.~Mates, R.~Gupta, A.~Bell, and F.~Klocke.
\newblock Fluid dynamics of coronary artery stenosis.
\newblock {\em Circulation research}, 42(1):152--162, 1978.

\bibitem{funamoto2013}
K.~Funamoto and T.~Hayase.
\newblock Reproduction of pressure field in ultrasonic-measurement-integrated
  simulation of blood flow.
\newblock {\em International journal for numerical methods in biomedical
  engineering}, 29(7):726--740, 2013.

\bibitem{mehregan2014}
F.~Mehregan, F.~Tournoux, S.~Muth, P.~Pibarot, R.~Rieu, G.~Cloutier, and
  D.~Garcia.
\newblock Doppler vortography: A color doppler approach to quantification of
  intraventricular blood flow vortices.
\newblock {\em Ultrasound in medicine \& biology}, 40(1):210--221, 2014.

\bibitem{hirtler2016}
D.~Hirtler, J.~Garcia, A.~Barker, and J.~Geiger.
\newblock Assessment of intracardiac flow and vorticity in the right heart of
  patients after repair of tetralogy of fallot by flow-sensitive {4D} {MRI}.
\newblock {\em European radiology}, 26(10):3598--3607, 2016.

\bibitem{charonko2013}
J.~Charonko, R.~Kumar, K.~Stewart, W.~Little, and P.~Vlachos.
\newblock Vortices formed on the mitral valve tips aid normal left ventricular
  filling.
\newblock {\em Annals of biomedical engineering}, 41(5):1049--1061, 2013.

\bibitem{sotelo2018}
J.~Sotelo, J.~Urbina, J.~Mura, C.~Tejos, P.~Irarrazaval, M.~Andia, D.~Hurtado,
  and S.~Uribe.
\newblock Three-dimensional quantification of vorticity and helicity from {3D}
  cine {PC}-{MRI} using finite-element interpolations.
\newblock {\em Magnetic resonance in medicine}, 79(1):541--553, 2018.

\bibitem{gibson1993}
C.~Gibson, L.~Diaz, K.~Kandarpa, F.~Sacks, R.~Pasternak, T.~Sandor, C.~Feldman,
  and P.~Stone.
\newblock Relation of vessel wall shear stress to atherosclerosis progression
  in human coronary arteries.
\newblock {\em Arteriosclerosis and thrombosis: a journal of vascular biology},
  13(2):310--315, 1993.

\bibitem{reneman2006}
R.~Reneman, T.~Arts, and A.~Hoeks.
\newblock Wall shear stress--an important determinant of endothelial cell
  function and structure--in the arterial system in vivo.
\newblock {\em Journal of vascular research}, 43(3):251--269, 2006.

\bibitem{shojima2004}
M.~Shojima, M.~Oshima, K.~Takagi, R.~Torii, M.~Hayakawa, K.~Katada, A.~Morita,
  and T.~Kirino.
\newblock Magnitude and role of wall shear stress on cerebral aneurysm:
  computational fluid dynamic study of 20 middle cerebral artery aneurysms.
\newblock {\em Stroke}, 35(11):2500--2505, 2004.

\bibitem{kissas2020}
G.~Kissas, Y.~Yang, E.~Hwuang, W.~Witschey, J.~Detre, and P.~Perdikaris.
\newblock Machine learning in cardiovascular flows modeling: Predicting
  arterial blood pressure from non-invasive {4D} flow {MRI} data using
  physics-informed neural networks.
\newblock {\em Computer methods in applied mechanics and engineering},
  358(1):112623, 2020.

\bibitem{ledoux1997}
L.~Ledoux, P.~Brands, and A.~Hoeks.
\newblock Reduction of the clutter component in doppler ultrasound signals
  based on singular value decomposition: A simulation study.
\newblock {\em Ultrasonic imaging}, 19(1):1--18, 1997.

\bibitem{bjaerum2002}
S.~Bjaerum, H.~Torp, and K.~Kristoffersen.
\newblock Clutter filter design for ultrasound color flow imaging.
\newblock {\em IEEE transactions on ultrasonics, ferroelectrics, and frequency
  control}, 49(2):204--216, 2002.

\bibitem{demene2015}
C.~Demen{\'e}, T.~Deffieux, M.~Pernot, B.-F. Osmanski, V.~Biran, J.-L.
  Gennisson, L.-A. Sieu, A.~Bergel, S.~Franqui, J.-M. Correas, et~al.
\newblock Spatiotemporal clutter filtering of ultrafast ultrasound data highly
  increases doppler and fultrasound sensitivity.
\newblock {\em IEEE transactions on medical imaging}, 34(11):2271--2285, 2015.

\bibitem{MM2013}
Y.~Maday and O.~Mula.
\newblock A {G}eneralized {E}mpirical {I}nterpolation {M}ethod: application of
  reduced basis techniques to data assimilation.
\newblock In F.~Brezzi, F.~Colli, U.~Gianazza, and G.~Gilardi, editors, {\em
  {Analysis and Numerics of Partial Differential Equations}}, volume~4 of {\em
  Springer INdAM Series}, pages 221--235. Springer Milan, 2013.

\bibitem{MMPY2015}
Y.~Maday, O.~Mula, A.~T. Patera, and M.~Yano.
\newblock {The Generalized Empirical Interpolation Method: Stability theory on
  Hilbert spaces with an application to the Stokes equation}.
\newblock {\em {Computer Methods in Applied Mechanics and Engineering}},
  287(0):310--334, 2015.

\bibitem{MPPY2015}
Y.~Maday, A.~T. Patera, J.~D. Penn, and M.~Yano.
\newblock A parameterized-background data-weak approach to variational data
  assimilation: formulation, analysis, and application to acoustics.
\newblock {\em International Journal for Numerical Methods in Engineering},
  102(5):933--965, 2015.

\bibitem{MMT2016}
Y.~Maday, O.~Mula, and G.~Turinici.
\newblock Convergence analysis of the {G}eneralized {E}mpirical {I}nterpolation
  {M}ethod.
\newblock {\em SIAM Journal on Numerical Analysis}, 54(3):1713--1731, 2016.

\bibitem{BCDDPW2017}
P.~Binev, A.~Cohen, W.~Dahmen, R.~DeVore, G.~Petrova, and P.~Wojtaszczyk.
\newblock Data assimilation in reduced modeling.
\newblock {\em SIAM/ASA Journal on Uncertainty Quantification}, 5(1):1--29,
  2017.

\bibitem{ABGMM2017}
J.~P. Argaud, B.~Bouriquet, H.~Gong, Y.~Maday, and O.~Mula.
\newblock Stabilization of (g)eim in presence of measurement noise: Application
  to nuclear reactor physics.
\newblock In Marco~L. Bittencourt, Ney~A. Dumont, and Jan~S. Hesthaven,
  editors, {\em Spectral and High Order Methods for Partial Differential
  Equations ICOSAHOM 2016: Selected Papers from the ICOSAHOM conference, June
  27-July 1, 2016, Rio de Janeiro, Brazil}, pages 133--145, Cham, 2017.
  Springer International Publishing.

\bibitem{Taddei2017}
T.~Taddei.
\newblock An adaptive parametrized-background data-weak approach to variational
  data assimilation.
\newblock {\em ESAIM: Mathematical Modelling and Numerical Analysis},
  51(5):1827--1858, 2017.

\bibitem{CDDFMN2019}
A.~{Cohen}, W.~{Dahmen}, R.~{DeVore}, J.~{Fadili}, O.~{Mula}, and J.~{Nichols}.
\newblock {Optimal reduced model algorithms for data-based state estimation}.
\newblock {\em preprint arXiv:1903.07938}, march 2019.

\bibitem{RHP2007}
G.~Rozza, D.~B.~P. Huynh, and A.~T. Patera.
\newblock Reduced basis approximation and a posteriori error estimation for
  affinely parametrized elliptic coercive partial differential equations.
\newblock {\em Archives of Computational Methods in Engineering}, 15(3):1, Sep
  2007.

\bibitem{BMNP2004}
M.~Barrault, Y.~Maday, N.~C. Nguyen, and A.~T. Patera.
\newblock An {E}mpirical {I}nterpolation {M}ethod: application to efficient
  reduced-basis discretization of partial differential equations.
\newblock {\em C. R. Acad. Sci. Paris, S{\'e}rie I.}, 339(9):667--672, 2004.

\bibitem{sirovich1987}
L.~Sirovich.
\newblock Turbulence and the dynamics of coherent structures. i. coherent
  structures.
\newblock {\em Quarterly of applied mathematics}, 45(3):561--571, 1987.

\bibitem{berkooz1993}
G.~Berkooz, P.~Holmes, and J.L. Lumley.
\newblock The proper orthogonal decomposition in the analysis of turbulent
  flows.
\newblock {\em Annual review of fluid mechanics}, 25(1):539--575, 1993.

\bibitem{CDS2011}
A.~Cohen, R.~DeVore, and C.~Schwab.
\newblock Analytic regularity and polynomial approximation of parametric and
  stochastic elliptic {PDE}'s.
\newblock {\em Analysis and Applications}, 9(1):11--47, 2011.

\bibitem{CD2015acta}
A.~Cohen and R.~DeVore.
\newblock Approximation of high-dimensional parametric {PDE}s.
\newblock {\em Acta Numerica}, 24:1--159, 2015.

\bibitem{KGV2018}
M.~K{\"a}rcher, S.~Boyaval, M.A. Grepl, and K.~Veroy.
\newblock Reduced basis approximation and a posteriori error bounds for 4d-var
  data assimilation.
\newblock {\em Optimization and Engineering}, 19(3):663--695, Sep 2018.

\bibitem{BCMN2018}
P.~Binev, A.~Cohen, O.~Mula, and J.~Nichols.
\newblock Greedy algorithms for optimal measurements selection in state
  estimation using reduced models.
\newblock {\em SIAM/ASA Journal on Uncertainty Quantification},
  6(3):1101--1126, 2018.

\bibitem{Ai}
S.E. Aidarous, M.R. Gevers, and M.J. Installe.
\newblock Optimal sensors' allocation strategies for a class of stochastic
  distributed systems.
\newblock {\em International Journal of Control}, 22(2):197--213, 1975.

\bibitem{CK}
J.R. Cannon and R.E. Klein.
\newblock Optimal selection of measurement locations in a conductor for
  approximate determination of temperature distributions.
\newblock {\em Journal of Dynamic Systems, Measurement, and Control},
  93(3):193--199, 1971.

\bibitem{YS}
T.~K. Yu and J.~H. Seinfeld.
\newblock Observability and optimal measurement location in linear distributed
  parameter systems.
\newblock {\em International journal of control}, 18(4):785--799, 1973.

\bibitem{BCDDPW2011}
P.~Binev, A.~Cohen, W.~Dahmen, R.~DeVore, G.~Petrova, and P.~Wojtaszczyk.
\newblock Convergence rates for greedy algorithms in reduced basis methods.
\newblock {\em SIAM Journal on Mathematical Analysis}, 43(3):1457--1472, 2011.

\bibitem{VGCR2010}
A.~Viswanathan, A.~Gelb, D.~Cochran, and R.~Renaut.
\newblock On reconstruction from non-uniform spectral data.
\newblock {\em Journal of Scientific Computing}, 45(1-3):487--513, 2010.

\bibitem{bertoglio_pdrop}
C.~Bertoglio, R.~Nuñez, F.~Galarce, D.~Nordsletten, and A.~Osses.
\newblock Relative pressure estimation from velocity measurements in blood
  flows: State-of-the-art and new approaches.
\newblock {\em International Journal for Numerical Methods in Bio-medical
  Engineering}, 34(2):e2925, 2018.

\bibitem{koskinas2009}
Konstantinos C., Yiannis S., Aaron B., Elazer R., Peter H., and Charles L.
\newblock The role of low endothelial shear stress in the conversion of
  atherosclerotic lesions from stable to unstable plaque.
\newblock {\em Current Opinion in Cardiology}, 24(6):580--590, 2009.

\bibitem{heo2014}
K.~Heo, K.~Fujiwara, and J.~Abe.
\newblock Shear stress and atherosclerosis.
\newblock {\em Molecules and Cells}, 37(6):435--440, 2014.

\bibitem{zarins1983}
C.~Zarins, D.~Giddens, B.K. Bharadvaj, V.~Sottiurai, R.~Mabon, and Glagov S.
\newblock Carotid bifurcation atherosclerosis.
\newblock {\em Circulation Research}, 53(4):502--514, 1983.

\bibitem{bluestein2010}
D.~Bluestein, K.B. Chandran, and K.B. Manning.
\newblock Towards non-thrombogenic performance of blood recirculating devices.
\newblock {\em Annals of Biomedical Engineering}, 38(3):1236--1256, 2010.

\bibitem{garcia2013}
J.~Garcia, E.~Larose, P.~Pibarot, and L.~Kadem.
\newblock On the evaluation of vorticity using cardiovascular magnetic
  resonance velocity measurements.
\newblock {\em Journal of Biomechanical Engineering}, 135(12):124501, 2013.

\bibitem{gmj2005}
J.~Guermond, P.~Minev, and J.~Shen.
\newblock An overview of projection methods for incompressible flows.
\newblock {\em Computer Methods in Applied Mechanics and Engineering},
  195(44--47), 2005.

\bibitem{mmg3d}
C.~Dapogny, C.~Dobrzynski, and P.~Frey.
\newblock Three-dimensional adaptive domain remeshing, implicit domain meshing,
  and applications to free and moving boundary problems.
\newblock {\em Journal of Computational Physics}, 262:358--378, 2014.

\bibitem{blanco2014}
P.~Blanco, S.~Watanabe, M.~Passos, P.~Lemos, and R.~Feijoo.
\newblock An anatomically detailed arterial network model for one-dimensional
  computational hemodynamics.
\newblock {\em IEEE transactions on biomedical engineering}, 62(2):736--753,
  2015.

\bibitem{fqv2009}
L.~Formaggia, A.~Quarteroni, and A.~Veneziani.
\newblock {\em Cardiovascular Mathematics. Modeling and simulation of the
  circulatory system}, volume~1.
\newblock Springer-Verlag Mailand, 01 2009.

\bibitem{binev2017}
P.~Binev, A.~Cohen, W.~Dahmen, R.~DeVore, G.~Petrova, and P.~Wojtaszczyk.
\newblock Data assimilation in reduced modeling.
\newblock {\em SIAM/ASA Journal on Uncertainty Quantification}, 5(1):1--29,
  2017.

\bibitem{GYK2009}
L.~Grinberg, A.~Yakhot, and G.~E. Karniadakis.
\newblock Analyzing transient turbulence in a stenosed carotid artery by proper
  orthogonal decomposition.
\newblock {\em Annals of Biomedical Engineering}, 37(11):2200--2217, 2009.

\bibitem{GMMT2019}
H.~{Gong}, Y.~{Maday}, O.~{Mula}, and T.~{Taddei}.
\newblock {PBDW} method for state estimation: error analysis for noisy data and
  nonlinear formulation.
\newblock {\em e-prints. arXiv:1906.00810}, Jun 2019.

\bibitem{brezzi1984}
F.~Brezzi and J.~Pitk\"aranta.
\newblock On the stabilization of finite element approximations of the stokes
  equations.
\newblock In {\em Efficient Solutions of Elliptic Systems. Notes on Numerical
  Fluid Mechanics, vol 10.}, 1984.

\bibitem{supg}
A.~Brooks and T.~Hughes.
\newblock Streamline upwind/petrov-galerkin formulations for convection
  dominated flows with particular emphasis on the incompressible navier-stokes
  equations.
\newblock {\em Computer Methods in Applied Mechanics and Engineering},
  32(1--3):199--259, 1982.

\bibitem{black_flow_benchmark_bertoglio}
C.~Bertoglio, A.~Caiazzo, Bazilevs Y., M.~Braack, M.~Esmaily-Moghadam,
  V.~Gravemeier, A.L. Marsden, O.~Pironneau, I.E. Vignon-Clementel, and W.A.
  Wall.
\newblock Benchmark problems for numerical treatment of backflow at open
  boundaries.
\newblock {\em International Journal for Numerical Methods in Biomedical
  Engineering}, 34(2):e2918, 2017.

\end{thebibliography}

\end{document}